\documentclass[11pt]{amsart}
\usepackage{amsfonts}
\usepackage{amsmath,amsthm}
\usepackage{bbm}

\usepackage{latexsym}
\usepackage{latexsym}
\usepackage{array}
\usepackage{amssymb}
\usepackage{enumerate}
\usepackage{comment}
\usepackage{graphicx}
\usepackage{tikz}
\usepackage{color}
\usepackage{bbm}

\DeclareFontFamily{U}{wncy}{}
\DeclareFontShape{U}{wncy}{m}{n}{<->wncyr10}{}
\DeclareSymbolFont{mcy}{U}{wncy}{m}{n}
\DeclareMathSymbol{\Sh}{\mathord}{mcy}{"58}

\usepackage[english]{babel}
\usepackage{fdsymbol}
\usepackage{color}
\usepackage[latin1]{inputenc}

\numberwithin{equation}{section}

\theoremstyle{plain}
\newtheorem{theorem}{Theorem}[section]
\newtheorem*{theorem*}{Theorem}
\newtheorem{lemma}[theorem]{Lemma}
\newtheorem{proposition}[theorem]{Proposition}
\newtheorem{corollary}[theorem]{Corollary}

\newtheorem{conjecture}[theorem]{Conjecture}

\theoremstyle{remark}
\newtheorem{remark}[theorem]{Remark}
\newtheorem{example}[theorem]{Example}

\newtheorem*{lem*}{Lemma}
\newtheorem*{sublem*}{Sublemma}
\newtheorem*{remark*}{Remark}
\newtheorem*{NB*}{NB}

\newcommand{\supp}{\mathop{\rm supp}\nolimits}

\newcommand{\R}{ \mathbb{R} }

\newcommand{\C}{ \mathbb{C} }
\newcommand{\Z}{ \mathbb{Z} }
\newcommand{\N}{ \mathbb{N} }

\newcommand{\T}{ \mathbb{T} }

\newcommand{\mI}{\mathbb{I}}

\newcommand{\gD}{\mathfrak D}
\newcommand{\gF}{\mathfrak F}
\newcommand{\gC}{\mathfrak C}

\newcommand{\fB}{\mathfrak B}

\newcommand{\Del}{\Delta}

\newcommand{\ssum}{{\modtwosum}}
\newcommand{\vs}{ \vec s }

\newcommand{\cA}{ \mathcal{A} }

\newcommand{\cC}{ \mathcal{C} }

\newcommand{\Cc}{ \mathcal{C} }

\newcommand{\cY}{ \mathcal{Y} }
\newcommand{\cL}{ \mathcal{L} }
\newcommand{\cR}{ \mathcal{R} }
\newcommand{\cX}{ \mathcal{X} }

\newcommand{\cK}{ \mathcal{K} }
\newcommand{\cF}{ \mathcal{F} }

\newcommand{\EE}{ {\mathbb E}}

\newcommand{\Lc}{ \mathcal{L} }
\newcommand{\cZ}{ \mathcal{Z} }

\newcommand{\Rc}{ \mathcal{R} }

\newcommand{\cT}{ \mathcal{T} }

\newcommand{\om}{ \omega }

\newcommand{\ga}{\gamma }

\newcommand{\s}{ \sigma }

\newcommand{\ka}{ \kappa }

\renewcommand{\phi}{ \varphi }
\newcommand{\oms}{ \omega^{12}_{3s}  }

\newcommand{\eps}{\varepsilon}

\newcommand{\de}{ \delta }
\newcommand{\al}{ \alpha }
\newcommand{\zz}{\mathfrak z}

\newcommand{\la}{ \lambda }

\newcommand{\dess}{\delta'^{12}_{3s}}

\newcommand{\tauz}{\tau_0}

\newcommand{\be}{\begin{equation}}
\newcommand{\ee}{\end{equation}}
\newcommand{\ben}{\begin{equation*}}
\newcommand{\een}{\end{equation*}}

\newcommand{\ov}{ \overline }

\newcommand{\lan}{ \langle }

\newcommand{\ran}{ \rangle}

\newcommand{\p}{ \partial}

\newcommand{\e}{ \text{e} }
\newcommand{\ai}{ \mathfrak{a} }
\newcommand{\wt}{ \widetilde }

\newcommand{\lbl}{\label}
\newcommand{\non}{\nonumber}
\newcommand{\qu}{\quad}
\newcommand{\qmb}{\quad\mbox}
\newcommand{\qnd}{\qmb{and}\qu}

\newcommand{\cS}{\mathcal{S}}

\newcommand{\atw}[1]{a^{(2)}_{#1}}

\newcommand{\bo}[1]{\bar a^{(1)}_{#1}}

\newcommand{\nL}[1]{n_{#1,L}}
\newcommand{\no}{\mathfrak{n}}

\newcommand{\dep}{\delta'^{12}_{3s}}

\title[The large-period limit]
{The large-period limit for  equations of discrete turbulence}

\author{Andrey  Dymov}
\address{Andrey Dymov \\ Steklov Mathematical Institute of RAS, Moscow 119991, Russia 
	\& National Research University Higher School of
	Economics, Moscow 119048, Russia} \email{dymov@mi-ras.ru}
\author{Sergei Kuksin}
\address{Sergei Kuksin \\ Universit\'e Paris-Diderot (Paris 7), UFR de Math\'ematiques - Batiment Sophie Germain, 5 rue Thomas Mann, 75205 Paris,
	France  \& School of Mathematics, Shandong University, Jinan, PRC }
\email{ Sergei.Kuksin@imj-prg.fr}

\author{ Alberto Maiocchi }
\address{Alberto Maiocchi \\ Università degli Studi di Milano-Bicocca, Dipartimento di Matematica e Applicazioni - Edificio U5,
via Roberto Cozzi, 55, 20125 Milan, Italy} \email{alberto.maiocchi@unimib.it}

\author{Sergei  Vl\u adu\c t}
\address{Sergei  Vl\u adu\c t \\ Aix Marseille Universit\'e, CNRS, Centrale Marseille, I2M UMR 7373, 13453, Marseille,
France and IITP RAS, 19 B. Karetnyi, Moscow, Russia} \email{serge.vladuts@univ-amu.fr}

\begin{document}

\begin{abstract}
	
We consider the damped/driven cubic NLS equation on the torus of a large period $L$ with a small nonlinearity of size $\lambda$, 
 a properly scaled random forcing and dissipation.
We examine its solutions under the subsequent limit when first $\lambda\to 0$ and then $L\to \infty$.
The first limit, called the limit of discrete turbulence, is known to exist, and in this work we study the second limit $L\to\infty$ for solutions to the equations of discrete turbulence.
Namely, we decompose the solutions to formal series in  amplitude and study 
 the second order truncation of this series. We prove that the energy spectrum of the truncated solutions  becomes close to 
solutions of a damped/driven nonlinear wave kinetic equation. Kinetic  nonlinearity of the latter is similar to that 
which usually appears in works on wave turbulence, but is different from it (in particular, it is non-autonomous). 
 Apart from  tools  from analysis and stochastic analysis, our work uses two 
powerful results from the number theory. 
\end{abstract}

\date{}

\maketitle



\section{Introduction}

\subsection{The setting}

In this paper we continue the study of the Zakharov-L'vov stochastic model for wave turbulence (WT), initiated in \cite{DK1,DK2}; see also a survey \cite{DKsmall}. 
We start by recalling the classical and the Zakharov-L'vov stochastic settings of WT. See the introduction to \cite{DK1} for more detailed discussions of the two models.

\smallskip

{\it Classical setting.}
Let ${\mathbb{T}}^d_L={  \mathbb{R}   }^d/(L{  \mathbb{Z}   }^d)$ be the $d$-dimensional torus, $d\geq 2$, 
of  period $L\geq 2$. 
We denote by $\|u\|$ the normalized $L_2$-norm of a complex function $u$ on $\T_L^d$, 
$
\|u\|^2 =
L^{-d}\int_{{{\mathbb T}}^d_L} |u(x)|^2\,dx\,,
$
and  write the Fourier series of $u$ in the form
\begin{equation}\lbl{Fourier-def}
u(x)= L^{-d/2} \,{\sum}_{s\in{{\mathbb Z}}^d_L} v_s e^{2\pi  i s\cdot x}, 
\qquad\qu{{\mathbb Z}}^d_L = L^{-1} {{\mathbb Z}}^d\,.
\end{equation}
Here  the vector of  Fourier coefficients  $v= (v_s)_{s\in\Z^d_L}$ 
is given by the Fourier transform of $u(x)$,
$$
v=\cF(u), \quad v_s = L^{-d/2}\,\int_{\T^d_L} u(x) e^{-2\pi i s\cdot x}\,dx \quad \text{for}\; s\in \Z_L^d, 
$$
so the  Parseval identity takes form 
$
\|u\|^2 = L^{-d}\,{\sum}_{s\in\Z^d_L} |v_s|^2.
$
We will study  solutions $u(t,x)$ whose norms  satisfy  $\|u(t,\cdot)\|\sim 1$ as $L\to\infty$. This makes the chosen in
 \eqref{Fourier-def} scaling of Fourier series convenient for our purposes.

We consider  the cubic NLS equation with modified nonlinearity 
\be\label{NLS}
\frac{{{\partial}}}{{{\partial}} t}  u +i\Delta u- i\la \,\big( |u|^2 -2\|u\|^2\big)u=0 , 
\qquad x\in {\mathbb{T}}^d_L,
\ee
where $u=u(t,x)$,  $ \Delta=(2\pi)^{-2}\sum_{j=1}^d ({{\partial}}^2 / {{\partial}} x_j^2)$ and 
$\la\in(0,1]$ is a small parameter. 
The modification of the nonlinearity by the term $2i\la\|u\|^2u$  keeps the main features of the standard cubic NLS equation, reducing 
some non-crucial technicalities; see  the introduction to \cite{DK1}.

The objective of WT is to study solutions of \eqref{NLS}  under the limit $L\to\infty$ and $\la\to 0$ on long time intervals. There are plenty of physical works containing some different (but consistent) approaches
to the limit;  many references  may be found in \cite{ZLF92, Naz11, NR}.
 Despite the strong  interest in physical and  mathematical communities to the addressed 
questions, significant progress in 
the rigorous justification of the physical predictions was achieved only recently \cite{LS, FGH, BGHS18, Faou, DH, CG1, CG2, BGHS, DH1}.
See e.g. the introductions to \cite{DK1,CG1, DH1}  for discussions of  the obtained results. 

\smallskip

{\it Zakharov-L'vov setting.} When  studying  eq. \eqref{NLS}, members of the WT community  talk
about "pumping  energy to low modes and dissipating it in high modes".  
To make this rigorous, following Zakharov-L'vov \cite{ZL75}, in the present paper as well as in \cite{DK1,DK2}
we consider the NLS equation \eqref{NLS} dumped by a (hyper) viscosity and driven by a random force:
\begin{equation}\label{ku3s}
\frac{{{\partial}}}{{{\partial}} t}  u +i  \Delta u  - i\la\,\big( |u|^2 -2\|u\|^2\big)u =  -\nu\frak A(u)    + \sqrt\nu\frac{{{\partial}}}{{{\partial}} t} \eta^\omega(t,x).
\end{equation}
Here $\nu\in(0,1/2]$ is another  small parameter, which should be properly agreed with $\lambda$ and $L$. The 
 dissipative linear operator $\frak A$ is defined as 
 \be\lbl{diss_op}
\frak A(u(x)) = L^{-d/2}{\sum}_{s\in \Z^d_L} \ga_s v_s e^{2\pi is\cdot x},  \quad v= \cF(u),\;\;
\ga_s =\ga^0(|s|^2),
\ee
where $|s|$ stands for the Euclidean norm of a vector $s$ and $\gamma^0(y)$ is a  smooth real increasing function of $y>0$, 
satisfying\,\footnote{ For 
	example, if $\gamma_s =(1+|s|^2)^{r*}$, then $\frak A = (1-\Delta)^{r*}$. 
}
\be\label{gamma}
\gamma^0\ge1 \;\; \text{and}\;\; 
c(1+ y)^{r_*} \le \gamma^0(y)  \le  C(1+ y)^{r_*}\qquad
\forall\, y>0\,. 
\ee
The exponent $r_* >0$ and  $c, C$ are positive constants. 
Also we assume that 
	$$
	\text{ {\it  all derivatives of  $\ga^0$ have at most polynomial growths at infinity.}}
$$
The random noise  $\eta^\omega$  is given by a  Fourier series
$$
\eta^\omega(t,x)=
L^{-d/2} \, {\sum}_{s\in\Z^d_L} b(s) \beta^\omega_s(t) e^{2\pi  i s\cdot x},
$$
where
$\{\beta_s(t), s\in{{\mathbb Z}}^d_L\}$ are standard independent complex Wiener processes\,\footnote{
	i.e. $\beta_s = \beta_s^1 + i\beta_s^2$, where $\{\beta_s^j, s\in{{\mathbb Z}}^d_L, j=1,2  \}$ are
	standard independent real Wiener processes.} and  
	 $b(x)$
	  is a  Schwartz  function on ${{\mathbb R}}^d\supset \Z^d_L$. 
	  \footnote{Often it is assumed that  the intensity  $b(x)$ of the noise $\eta^\omega$ is non-negative, but we do not impose this condition. 
	  Note that if $b(x)\equiv 0$, then our results become trivial since below we will provide \eqref{ku3s} with the zero initial conditions.}

Solutions $u(\tau)$ of \eqref{ku3s} are random processes in
the space 
$
H=   L_2(\T^d_L, \C), 
$
equipped with the norm $ \|\cdot\| $.
If $r_*$ is sufficiently big, equation \eqref{ku3s} is known to be well posed. Moreover,  Ito's formula shows that  $\EE\|u(\tau)\|^2$ is bounded uniformly in $\tau$ and $L,\nu,\la$, once $\EE\|u(0)\|^2$ is bounded uniformly in these parameters, see in \cite{DK1}.

We will study the equation on 	time intervals of order  $\nu^{-1}$. So it is convenient to pass  from $t$ 
to the slow time $\tau=\nu t$ and write  eq.~\eqref{ku3s} as 
\begin{equation}\label{ku3}
\begin{split}
\dot u +i \nu^{-1} \Delta u  - i\rho\,\big( |u|^2 -2\|u\|^2\big)u &= -\frak A(u)    + \dot \eta^\omega(\tau,x),\\
\eta^\omega(\tau,x)&=
L^{-d/2} \, {\sum}_{s\in\Z^d_L} b(s) \beta^\omega_s(\tau) e^{2\pi  i s\cdot x}\,.
\end{split}
\end{equation}
Here $\rho=\la\nu^{-1}$, the upper-dot stands for $d/d\tau$  
and $\{\beta_s(\tau), \, s\in\Z^d_L\}$ is another set of standard independent complex Wiener processes.  
Below we use   $\rho$, $\nu$ and $L$ 	 as  parameters of the equation.

In the context of equation \eqref{ku3}, 
the objective of WT is to study its solutions $u(\tau)$ when
\be\label{lim}
L\to\infty \;\;\; \text{and}\;\; \; \nu\to 0,
\ee
while $\rho=\rho(\nu,L)$ is scaled appropriately, mostly paying attention to their  {\it energy spectra} 
\be\lbl{energy spectrum}
N_s(\tau): =  {{\mathbb E}} |v_s (\tau) |^2, \qmb{where}\qu 
v(\tau)=\cF(u(\tau)).
\ee  
Exact meaning of the limit \eqref{lim}   is unclear since no relation between the  parameters $\nu$ and
$L$ is postulated by the theory. 

Motivated by physical works,  in the present paper, as in \cite{DK1,DK2}, we study 
formal decompositions in $\rho$ of solutions to eq. \eqref{ku3} and of their  energy spectra $N_s$ under  the limit \eqref{lim}.
See the introduction to \cite{DK1} for a discussion of our motivation,
and see below Section~\ref{sec:quasisol}. 
In \cite{DK1,DK2} we understand  the limit \eqref{lim} as
\be\label{the_lim}
\text{
   first  $L\to\infty$ and then $\nu\to 0$, or  $L\gg\nu^{-2}$ while $\nu\to0$.  }
 \ee
There we have shown that    principal terms of 
the decomposition of $N_s$ in $\rho$  have  a non-trivial limiting   behaviour, provided that $\rho$ is scaled as $\rho \sim \nu^{-1/2}$, governed by  a 
nonlinear  {\it wave kinetic equation} (WKE) with added dissipation and a constant forcing. 
The WKE coincides with that,  arising in physical works, so this result agrees well  with the predictions of  the WT.

In the present paper we are interested in the opposite order of limits, which rarely appears in physical works:
\be\lbl{lim'}
\mbox{ firstly \quad $\nu\to 0$, \quad and then \quad  $L\to\infty$.}
\ee  
Roughly speaking, our main result is that under the double limit \eqref{lim'} the behaviour of  principal terms of the decomposition in $\rho$ for the energy spectrum $N_s$ is governed by a {\it modified} WKE.  The latter is similar to the WKE arising in \cite{DK1,DK2} and physical papers, but is different from them. The  scaling of $\rho$ now is $\rho\sim L\,\chi_d(L)$, where 
\be\lbl{chi_d}
\chi_d(L)\equiv 1 \qmb{if}\qu d\geq 3 \qmb{and}\qu \chi_d(L)=(\ln L)^{-1/2} \qmb{if}\qu d=2.
\ee
To the best of our 
knowledge this WKE did not appear in the literature before. 

For the proof  we start with  the result obtained in \cite{KM16,HKM}, where  the limiting as $\nu\to 0$ behaviour of equation \eqref{ku3}
is examined  (while $L$ and $\rho$ are kept  fixed). Then we pass to the limit as $L\to\infty$, 
 following  the approach   of \cite{DK1,DK2} and using the  developed there tools, such as a specific Feynman diagram presentation. Another  key ingredient of the proof is an obtained in \cite{number_theory} 
 refinement of the Heath-Brown circle method for quadratic forms  \cite{HB},  and certain upper bounds for the number of integer points 
 on intersections of quadrics.  In the next subsections we describe our results and methods in more detail.

In  \cite{KM15} a similar result concerning the iterated limit \eqref{lim'} 
 was found  heuristically, however there 
 $\rho$ was scaled as $\rho\sim\sqrt{L}$. The present paper shows that the correct scaling is different: $\rho\sim L{ \,\chi_d(L)}$.  

Similar regimes, when $L\to \infty$ slowly while $\nu>0$ fast decays to zero,  were studied in \cite{FGH,BGHS18}.
 However  the elegant description of the limit, obtained there, is
 far from the prediction of  WT. The works  \cite{FGH,BGHS18} should  rather  be regarded as a kind of averaging
  (similar to that of Krylov--Bogolyubov) 
 since the considered there  time scale is much shorter than the characteristic time scale of WT. Note that in \cite{BGHS18} a similar
  to ours \cite{number_theory} refinement of the  Heath-Brown method also is crucially used.

\subsection{The limit of discrete turbulence}

We first consider the limit 
\be\lbl{lim_dt} 
\nu\to 0 \qmb{while $L$ and $\rho$ stay fixed.}
\ee
It is known as
\textit{the limit of discrete turbulence} (see \cite[Section 10]{Naz11}) and has been
successfully studied in \cite{KM16,HKM}. To explain the result let  us take the Fourier transform of equation~\eqref{ku3}:
\be\label{ku33}
\begin{split}
	&  \dot v_s -i\nu^{-1}|s|^2 v_s + \gamma_s v_s
	=  i\rho  L^{-d} \Big( {\sum}_{1,2,3} \dep v_1 v_{2}
	\bar v_{3} -|v_s|^2v_s\Big)+b(s) \dot\beta_s 
\end{split}
\ee
for $s\in{{\mathbb Z}}^d_L$. Here, as it is common in WT, $v_j$ abbreviates $v_{s_j}$, 
$ {\sum}_{1,2,3} $ stands for  ${\sum}_{s_1,s_2,s_3\in\Z^d_L} $, and 
\be\lbl{def_del}
\dep={\delta'}^{s_1s_2}_{s_3s}: =
\left\{\begin{array}{ll}
	1,& \text{ if $s_1+s_2=s_3+s$ and $\{s_1, s_2\} \ne \{s_3, s\}$}
	\,,
	\\
	0,  & \text{otherwise}.
\end{array}\right.
\ee
 Note that 
\be\label{notethat}
\text{
if \; $\dess=1$, \;  then \; $\{s_1, s_2\} \cap \{s_3, s\} =\emptyset.$
}
\ee

We pass to the {\it interaction representation}, 
\be\lbl{interaction r}
\ai_s(\tau)=v_s(\tau) e^{-i\nu^{-1}\tau |s|^2}, \qu s\in\Z^d_L,
\ee
and denote
\be\label{omega}
\oms=\omega^{s_1s_2}_{s_3s}:= |s_1|^2+|s_2|^2-|s_3|^2 - |s|^2=  -2
(s_1-s)\cdot (s_2-s),
\ee
where the last equality holds if $\dep=1$ since then $s_3=s_1+s_2-s$.
Then equation \eqref{ku33} takes the form
\begin{equation}\label{ku44}
\begin{split}
\dot \ai_s + \gamma_s \ai_s
&=  i\rho Y_s(\ai,\nu^{-1}\tau) +b(s) \dot\beta_s,  \qquad s\in{{\mathbb Z}}^d_L,\\
Y_s(\ai, t )&=L^{-d} \Big({\sum}_{1,2,3} \dep \ai_{1}\ai_{2} \bar \ai_{3}
e^{  i t \oms} - |\ai_s|^2\ai_s\Big),
\end{split}
\end{equation}
where 
$\{\beta_s\}$ is yet   another set of standard independent complex Wiener processes (and, again $\ai_j$ stands for $\ai_{s_j}$). 
Note that the energy spectra  of solutions to equations \eqref{ku33} and \eqref{ku44} coincide: 
\be\lbl{i:es}
N_s(\tau)=\EE|v_s(\tau)|^2=\EE|\ai_s(\tau)|^2.
\ee
Sometimes we will write $N_s$ and $\ai_s$ as  $N_s(\tau;\nu,L)$ and $\ai_s(\tau;\nu,L)$. The limiting dynamics in equation \eqref{ku44} under the limit \eqref{lim_dt} 
is governed by  the \textit{effective equation} of discrete turbulence. The latter 
 has the  form  \eqref{ku44} with the modified nonlinearity  $Y^{res}$, in which  the sum is taken only over resonant vectors $s_1,s_2,s_3$:
\begin{equation}\label{ku5}
\begin{split}
\dot \ai_s+\gamma_s \ai_s &= i\rho Y_s^{res}(\ai)
+b(s) \dot\beta_s
\,,\quad s\in{{\mathbb Z}}^d_L  \,,\\
Y_s^{res}(\ai)&=L^{-d}\Big( {\sum}_{1,2,3} \dess \delta(\omega^{12}_{3s}) 
\ai_{1} \ai_{2} \bar \ai_{3} - |\ai_s|^2\ai_s\Big)\,.
\end{split}
\end{equation}
Here   $\delta(\omega^{12}_{3s})=1$ if $\omega^{12}_{3s} = 0$ and $\delta(\omega^{12}_{3s})=0$ otherwise.
The following result is proven in \cite{KM16, HKM}.
\begin{theorem}\label{t_discrturb}
	If $d\ge1$ and 
	$r_*\geq d/2 + 1$, 	 then equations \eqref{ku44} and 
	\eqref{ku5} are well posed. Under the limit \eqref{lim_dt}, on time intervals of order 1,  \\
	i) a solution  $\ai^\eps(\tau)$ of \eqref{ku44} converges in distribution to a solution  $\ai^0(\tau)$ of \eqref{ku5} with the same initial
	data at $\tau=0$;\\
	ii) the energy spectra $\EE |\ai_s^\eps(\tau)|^2$ converges to the energy spectra $\EE |\ai_s^0(\tau)|^2$.
\end{theorem}

If $b(x)>0$ for all $x$, then it is known that under the  assumptions of  Theorem~\ref{t_discrturb}, equations \eqref{ku44} and 
\eqref{ku5} are mixing, so they have unique stationary measures. Then in  \cite{KM16, HKM} it is also proven that the 
assertions of the theorem remain true if $\ai^\eps(\tau)$ and $\ai^0(\tau)$ denote stationary solutions of equations \eqref{ku44} and 
\eqref{ku5}.

\subsection{The main result }

In view of Theorem~\ref{t_discrturb}, to understand behaviour of the energy spectrum \eqref{i:es} of equation \eqref{ku44} under the limit \eqref{lim'}, it remains to study that 
 of the energy spectrum $\EE|\ai_s(\tau;L)|^2$ of the effective equation \eqref{ku5} under the limit $L\to \infty$.  
Instead, following the logic of \cite{DK1}, we study the energy spectrum corresponding to a principal part of decomposition in $\rho$ for the solutions $\ai_s(\tau;L)$ of 
equation \eqref{ku5}.  

\smallskip

{\it Quasisolutions and their energy spectra.}  To simplify presentation we 
assume that initially the system was  at rest, i.e. supplement equation \eqref{ku5} with the zero initial condition 
\be\label{in_cond}
\ai_s(0)=0 \qquad \forall s\in\Z^d_L.
\ee
We formally decompose the corresponding solution of \eqref{ku5}  in $\rho$,
\be\label{decompor}
\ai(\tau)=\ai^{(0)}(\tau)+\rho \ai^{(1)}(\tau)+ \rho^2 \ai^{(2)}(\tau) + \dots ,\quad\qquad \ai^{(k)}(0)=0,
\ee
$\ai^{(k)}(\tau)=\ai^{(k)}(\tau;L).$
The process $\ai^{(0)}(\tau)$
satisfies the linear equation  
$$
\dot \ai_s^{(0)} + \gamma_s \ai_s^{(0)}
= b(s) \dot\beta_s  \,,\quad s\in{{\mathbb Z}}^d_L\,,
$$
so it is Gaussian,
\be\label{a0}
\ai^{(0)}_s(\tau) = b(s) \int_{0}^\tau e^{-\gamma_s(\tau-l)}d\beta_s(l).
\ee
The process $\ai^{(1)}$ satisfies
$$
\dot \ai^{(1)}_s + \gamma_s \ai^{(1)}_s 
=  i Y_s^{res}(\ai^{(0)})\,,
$$
so that
\be\label{a1}
\ai^{(1)}_s(\tau) = iL^{-d}\!\! \int_{0}^\tau  e^{-\ga_s(\tau-l)}
\left(\sum_{1,2,3}\dess \delta(\omega^{12}_{3s}) (\ai^{(0)}_1
\ai^{(0)}_2{\bar \ai}^{(0)}_3) -|\ai_s^{(0)}|^2\ai_s^{(0)} \right)\!(l)dl\,
\ee
is a Wiener chaos of third order (see \cite{Jan}). 
Similar for $ n\ge1$, 
\be\label{an}
\begin{split}
	\ai^{(n)}_s&(\tau) 
	= iL^{-d}  \sum_{n_1+n_2+n_3=n-1}  \int_{0}^\tau  e^{-\ga_s(\tau-l)} \\ 
	& \times\!\!\,  \left( {\sum}_{1,2,3}
	 \dess \delta(\omega^{12}_{3s})\big(
	\ai_{1}^{(n_1)} \ai_{2}^{(n_2)} {\bar \ai_{3}}^{(n_3)}\big) - \ai_s^{(n_1)} \ai_s^{(n_2)} \bar\ai_s^{(n_3)} \right)(l)\,dl, 
\end{split}
\ee
is a Wiener chaos of order $2n+1$. 	

Next we consider the quadratic truncation of the series \eqref{decompor},
\be\lbl{A-intro}
\mathcal{A}_s(\tau; L)= 
\mathcal{A}_s(\tau) = \ai^{(0)}_s(\tau) +\rho \ai^{(1)}_s(\tau) +\rho^2 \ai^{(2)}_s(\tau)\,,
\ee 
which we call the {\it quasisolution}\,\footnote{by analogy with the  {\it quasimodes} in the spectral theory of the  Shr\"odinger operator.} 
 of the effective equation \eqref{ku5}, \eqref{in_cond}. 
It is traditional in  WT to analyse the quasisolution instead of the solution itself, postulating that the former well approximates the latter; see introduction to \cite{DK1} for a discussion.
The goal of the present paper is to study the behaviour of the energy spectrum of  $\mathcal{A}(\tau)$, 
\be\lbl{n_s-a}
\mathfrak{n}_{s,L}(\tau) = \EE |\cA_s(\tau;L)|^2,\quad s\in\Z_L^d, 
\ee  
as $L\to\infty$. 
Our results, formulated below, show that under this limit  
the energy spectrum $\mathfrak{n}_{s,L}(\tau)$ has a non-trivial behaviour (i.e.  stays finite and behaves differently from $\EE|\ai_s^{(0)}|^2$) only if  $\rho\sim L\,\chi_d(L)$, where $\chi_d$ is defined in \eqref{chi_d}.
 Accordingly, from now on we assume that 
\be\lbl{rho(L)}
\rho=\eps L \chi_d(L),
\ee
where $0<\eps\leq 1/2$ is a small but fixed constant (see a few lines below for its discussion). 
Then the energy spectrum $\no_{s,L}$ expands as 
\be\lbl{N_s_decomp}
\no_{s,L}(\tau)=\no_{s,L}^{(0)}(\tau) + \eps \, \no_{s,L}^{(1)}(\tau)+\eps^2 \no_{s,L}^{(2)}(\tau) +\eps^3 \no_{s,L}^{(3)}(\tau) +\eps^4 \no_{s,L}^{(4)}(\tau), 
\ee
$s\in\Z_L^d,$ 
where
\be\label{n_s^k}
\no_{s,L}^{(k)}(\tau) =  \big(L\chi_d(L)\big)^k \sum_{\substack{k_1+k_2=k \\ 0\leq k_1,k_2\leq 2}} \EE \ai^{(k_1)}_s(\tau) \bar \ai^{(k_2)}_s(\tau)\,. 
\ee
In particular, by \eqref{a0},
\be\label{n_s_0}
\no_{s,L}^{(0)} (\tau)= \EE |\ai_s^{(0)}(\tau)|^2
=\frac{b(s)^2}{\ga_s} \big( 1-e^{-2\ga_s\tau}\big),
\ee
and a simple computation shows that $\no_{s,L}^{(1)}(\tau)\equiv 0$.
For higher order terms we prove that 
\be\lbl{i:n est}
\no_{s,L}^{(2)}\sim 1\qnd |\no_{s,L}^{(3)}|,\,|\no_{s,L}^{(4)}|\lesssim 1 \quad\mbox{as }L\to\infty \mbox{ uniformly in }\tau\geq 0;
\ee
see a discussion in the next subsection. Thus, the parameter $\eps$ measures the 
properly scaled amplitude of the solutions, and   indeed it should be small for the methodology of WT to apply. 
Then, the term $\eps^2\no_{s,L}^{(2)}$ is the crucial non-trivial component of the energy spectrum $\no_{s,L}$ while the terms $\eps^3\no_{s,L}^{(3)},\; \eps^4\no_{s,L}^{(4)}$ are perturbative.  This well agrees with the prediction of physical works concerning various models of WT.

\smallskip

{\it Wave kinetic equation.} In view of \eqref{i:n est}, to study the limiting as $L\to\infty$ behaviour of the energy spectrum $\no_{s,L}(\tau)$ up to an error of size  $\eps^3$ it remains to investigate the behaviour of its principal component $\no^{(0)}_{s,L}(\tau)+\eps^2\no_{s,L}^{(2)}(\tau)$. 
We show that the latter is governed by a WKE. To state the result let us consider the {\it resonant quadric}
\be\lbl{i:quadr}
\Sigma_s=\{(s_1, s_2)\in\R^{2d}: \, (s_1-s)\cdot(s_2-s)=0\},
\ee
cf. \eqref{omega}, and a measure $\mu^{\Sigma_s}$ on it,  given by
\be\label{meas}
\mu^{\Sigma_s} (ds_1ds_2) =   {\big(|s_1-s|^2+|s_2-s|^2\big)}^{-1/2} ds_1 ds_2\!\mid_{\Sigma_s},
\ee
where $ds_1ds_2\!\!\mid_{\Sigma_s}$ denotes the volume element  on $\Sigma_s$,  corresponding to the standard  Euclidean 
structure on $\R^{2d}$.

Let us consider the following non-autonomous 
cubic {\it wave kinetic integral} operator $K(\tau)$, for any $\tau\geq 0$ sending a function $y_s$, $s\in\R^d$,  to the function $K_s(\tau)y$,
defined as 
\be\lbl{i:KI}
K_s (\tau)y
=4C_d\int_{\Sigma_s} \mu^{\Sigma_s}(ds_1 ds_2) 
\Bigl(\cZ^4 y_1y_2y_3 
+
\cZ^3y_1y_2 y_4- \cZ^2 y_1 y_3y_4
-  \cZ^1
y_2y_3y_4\Bigr).
\ee
Here $y_j:=y_{s_j}$ with  $s_4:=s$ and  $s_3:=s_1+s_2-s$, $C_d$ is the constant from  Theorem~B below,  the kernels 
 $\cZ^j=\cZ^j(\tau;s_1,s_2,s_3,s_4)$ are given by formulas \eqref{Z-edi}, \eqref{Z-ed}
and satisfy $0\leq \cZ^j(\tau)\leq 1$. When $\tau\to\infty$ the operator $K(\tau)$ exponentially fast converges to a limiting 
 kinetic integral operator $K(\infty)$,  
 	given by \eqref{i:KI} with $\cZ^j=(\ga_{s_1}+ \ga_{s_2}+\ga_{s_3}+\ga_{s_4})^{-1}$ for all $j$:
 \be\lbl{Kn_inf}
 K_s (\infty)y
 = 4C_d\int_{\Sigma_s} \frac{\mu^{\Sigma_s}(d s_1 ds_2) }{\ga_{s_1}+ \ga_{s_2}+\ga_{s_3}+\ga_{s_4}}
 \Bigl(y_1y_2y_3 
 +
 y_1y_2 y_4- y_1 y_3y_4
 -  
 y_2y_3y_4\Bigr).
 \ee
It is similar 
to the standard four-waves kinetic operator of WT (e.g. see in \cite{Naz11}), which has the form \eqref{i:KI} with $\cZ^j\equiv const$,
 but still is different from the latter 
 since  $K(\infty)$ depends on the spectrum $\{\gamma_s\}$ of the dissipation operator $\frak A$.\footnote{Earlier the kinetic operator $K(\infty)$ was 
 heuristically obtained in \cite{KM15}. 
 }

For $r\in\R$ we denote by $\cC_r(\R^d)$ a space of continuous complex  functions on $\R^d$ with finite norm 
\be\lbl{r-norm}
|f|_r=\sup_{z\in\R^d}|f(z)|\lan z\ran^r, \qmb{where}\qu \lan z\ran=\max(|z|,1).
\ee
In Section~\ref{sec:kin_eq}, following \cite{DK1}, 
 we show  that if $r>d$, then for any $\tau$ the operator $K(\tau)$ defines a continuous $3$-homogeneous mapping $K(\tau):\cC_r(\R^d)\mapsto \cC_{r+1}(\R^d)$,  and for any $y\in\cC_r(\R^d)$ the curve $\tau\mapsto K(\tau)(y)$ is H\"older continuous in  $\cC_r(\R^d)$.

Now consider the following damped/driven non-autonomous WKE 
\be\label{kin eq}
\dot \zz_s (\tau) = -2\ga_s\zz_s +\eps^2 K_s(\tau)(\zz) +2 b(s)^2, \qquad  \;\; \zz(0)=0,
\ee
where $\tau\ge0$ and $s\in\R^d$. In  Section~\ref{sec:kin_eq} we prove that for small $\eps$ 
it has a unique solution $\zz_s(\tau)$, which  can be written as 
$\zz_s(\tau)=\zz_s^0(\tau)+\eps^2\zz_s^1(\tau,\eps)$, 
where $\zz_s^0,\,\zz_s^1\sim 1$ and $\zz_s^0$  solves  the linear equation \eqref{kin eq}$|_{\eps=0}$. It is easy to see that $\zz_s^0$ equals 
 the component $\no_{s,L}^{(0)}$ of the energy spectrum $\no_{s,L}$, given by \eqref{n_s_0}, and we prove that $\zz_s^1$ is $\eps^4$-close to   $\no_{s,L}^{(2)}$   uniformly in $\tau$. Then, in view of \eqref{i:n est}, the energy spectrum $\no_{s,L}$ is $\eps^3$-close
  to the solution $\zz_s(\tau)$. 

Below we denote by $C^\#(s)$ various positive functions of $s$ which decay as $|s|\to\infty$ faster than any negative degree of $|s|$. These functions never depend on the parameters $L,\eps$ and $\tau$. By $C^\#(s;p)$ we denote functions $C^\#(s)$ depending on a parameter $p$.

\smallskip

\noindent {\bf Theorem A  (Main theorem)}.
{\it   Let $d\ge2$. Then the energy spectrum  $\no_{s,L}(\tau)$  of the quasisolution $\cA_s(\tau)$ of \eqref{ku5}, \eqref{in_cond} 
	satisfies the estimate $\no_{s,L}(\tau)\leq C^\#(s)$ and	is $\eps^3$-close to the solution $\zz_s(\tau)$ of WKE \eqref{kin eq}. 
	Namely,  under the scaling $\rho = \eps L \,\chi_d(L)$,  for 
	any $r$ there exists $\eps_r \in(0, 1/2]$  such that for $0<\eps \le \eps_r$ we have 
	\be\lbl{TC-est}
	| \no_{\cdot,L}(\tau) - \zz_\cdot(\tau)|_r \le C_r\eps^3\qquad \forall\, \tau \ge 0, 
	\ee
	if $L\ge \eps^{-2}$ for $d\geq 3$, and $L\ge e^{\eps^{-1}}$ for $d=2$.} 
	

\smallskip

See Theorem~\ref{t_k2}. Since the energy spectrum $\no_s$ is defined for $s\in\Z^d_L$ with finite $L$, then the
norm in \eqref{TC-est} is understood as $|f|_r=\sup_{z\in\Z^d_L}|f(z)| \lan z\ran^r$. 
\begin{remark}\lbl{rem:d=2}
	If $d=2$, the lower bound $L\ge e^{\eps^{-1}}$ can be relaxed in the following sense. In Appendix~\ref{app:d=2} we explain that
	 there is a bounded correction  $f(\tau,L)$ which can be written explicitly, such that
	\be\lbl{TC-est,d=2}
	\Big| \no_{\cdot,L}(\tau) - \zz_\cdot(\tau) - \frac{f(\tau,L)}{\ln L}\Big|_r \le C_r\eps^3\qquad \forall\, \tau \ge 0, 
	\ee
	if $L\ge \eps^{-6}$.
\end{remark}

In Lemma~\ref{l_k5} we show that for $\eps\ll 1$ equation~\eqref{kin eq}$\!{}\mid\!_{\tau=\infty}$ 
has a unique steady state $\zz^\eps \in\Cc_r(\R^d)$. It  is
close to the unique steady state $\zz^0_s:=b(s)^2/\ga_s$ of the linear equation  \eqref{kin
  eq}$|_{\eps=0}$ and is asymptotically stable. Jointly with Theorem~A
 this result implies the following  asymptotic in time behaviour of the  
energy spectrum $\no_{s,L}(\tau)$:
\be\label{time_ass}
| \no_{\cdot,L}(\tau) -\zz^\eps_\cdot|_r \le C_r(e^{-\tau} +\eps^3), \qquad \forall\, \tau\ge 0,
\ee
if $L\ge \eps^{-2}$  for $d\geq 3$, and $L\ge e^{\eps^{-1}}$ for $d=2$,  see  \eqref{k20}. \smallskip

  The cases $d\ge3$ and $d=2$ 
are similar, but should be treated  separately. To shorten the presentation we give a detailed proof of Theorem A only for  $d\geq 3$, when 
$$
	\chi_d=1 \qmb{and}\qu \rho=\eps L.
$$
The proof for $d=2$ can be obtained by a simple modification of the argument for $d\geq 3$.
 We sketch it  in Appendix~\ref{app:d=2}. So from now on,  except Section~\ref{sec:diagram} which gives  a brief account of the 
 method of Feynman diagrams from \cite{DK1, DK2},  we assume that $d\geq 3$.

\medskip
In paper \cite{DK1} we examine the 
 behaviour of the energy spectrum $\no_{s,L,\nu}(\tau)$ of a quasisolution to equation~\eqref{ku44} under the double limit \eqref{lim},
  when firstly $L\to\infty$ and then $\nu\to0$ (or $L\gg\nu^{-2}$ while $\nu\to0$), assuming $\rho=\eps\nu^{-1/2}$.\,\footnote{In \cite{DK1} the notation is slightly different: there we set $\rho=\eps^{1/2}\nu^{-1/2}$.}
We got there 
a similar result which states that $\no_{\cdot,L,\nu}(\tau)$ is $\eps^4$-close to a solution of the damped/driven 
 four-wave kinetic equation as in \cite[Section~6.9.1]{Naz11} 
 (in contrast with \eqref{i:KI}-\eqref{Kn_inf},  the kinetic nonlinearity there does not depend on the 
dissipation   in equation \eqref{ku3s}). 
\medskip

\noindent {\it What next?} In this work and in \cite{DK1} we obtained wave kinetic limits for the energy spectra of quasisolutions for the NLS
equation \eqref{ku3} under limit  \eqref{lim'} with the scaling 
$\rho = \eps L \, \chi_d(L)$ and limit \eqref{the_lim} with the scaling $\rho = \eps\nu^{-1/2}$. Our next goal is to show that an exact solution $\ai_{\cdot} (\tau)$ of 
  eq.~\eqref{ku44} is $\eps^3$-close to its quasisolution $\cA_{\cdot}(\tau)$ (uniformly in 
$L\ge2$ and $\tau\in[0,T]$, for any $T>0$). And that
a solution of eq.~\eqref{ku3}
is $\eps^3$-close to the quasisolution of the equation (uniformly in $\nu$, $L$ and $\tau\in[0,T]$, if $L\ge \nu^{-2-\bar\ga}$, 
$\bar\ga>0$). 
This would imply that the energy spectra of solutions of eq.~\eqref{ku3} under limit 
\eqref{lim'} and limit  \eqref{the_lim} are $\eps^3$-close to solutions of the two WKE (namely, 
 eq.~\eqref{kin eq} and the WKE from \cite{DK1}). To prove this, say, 
for a solution $\ai_{\cdot} (\tau)$  of eq.~\eqref{ku44} we consider the equation on any fixed time-interval $[0,T]$ and regard it as a nonlinear 
equation $F_T(\ai_{\cdot} (\cdot) ) =0$. Then the quasisolution $\cA$ satisfies the equation with a disparity $\lesssim\eps^3$. 
By analogy with some stochastic problems for nonlinear PDEs, recently successfully resolved by the KAM-techniques (e.g. see \cite{Kuk}), 
 we believe that KAM
also applies to the equation $F_T=0$. Its application would imply that $\ai$ is $\eps^3$-close to $\cA$, as stated. We also believe that 
analysis of the KAM-iterations which build $\ai$ from $\cA$ will show that the energy spectrum of the solution $\ai_{\cdot} (\tau)$ 
 of eq.~\eqref{ku44} under the limit $L\to\infty$ converges to a solution of the WKE \eqref{kin eq}. A similar logic should apply to the energy 
 spectra of solutions for eq.~\eqref{ku3} under the limit \eqref{the_lim}.

\subsection{Outline of the proof: Feynman diagrams and number theory}

	It is well understood that to write down formulas for the terms $\no_{s,L}^{(k)}$ of decompositions as \eqref{N_s_decomp} it is instrumental to use the language of Feynman diagrams. In application to similar problems this 
 goes back at least to  the  works \cite{EY00,ESY07},  and then 
 was successfully used for the purposes of WT   in \cite{LS, DH, CG1, CG2, BGHS, DH1} and other papers. 
We use this  techniques  in the form developed in \cite{DK2} which gives  a 
 convenient presentation of the terms $\no_{s,L}^{(k)}$ (see \eqref{n_s^k}). 
 Namely, by iterating the Duhamel formula \eqref{an} we express $\ai^{(n)}(\tau)$ in
 terms of the Gaussian processes $\ai_s^{(0)}$, and next evoking the Wick formula for moments of  $\ai_s^{(0)}$ write the terms 
 $\no_{s,L}^{(k)}$ as multiple sums. Then the just mentioned diagram  techniques  allows to `integrate' these sums. That is,  to write any
   $\no_{s,L}^{(k)}$ 
 as a  sum over an intersections of $k-1$ \ quadrics in $(\Z_L^d)^k$ in a form,   convenient to pass to a limit as $L\to\infty$. 
 The term  $\no_{s,L}^{(2)}$  is a sum over a single 
  quadric and may be analised without the diagram's machinery. This and  some other similar terms play 
  a leading role in our analysis and dictate the form of the  limiting WKE. The terms may be written as  sums 
 \be\label{the_form}
 G_s(\tau, L) = L^{2(1-d)} \sum_{\substack{z_1,z_2 \in\Z_L^d:
  \\ z_1\cdot z_2=0, \, z_1,z_2\ne 0}} \Phi_s(\tau; z_1, z_2), 
 \ee
 well known in works on WT. To study them  under the limit $L\to\infty$ we make use of the celebrated circle method of Heath-Brown \cite{HB}.
 Since the result of \cite{HB} does not completely fit our  purposes, we specified it in the accompanying paper \cite{number_theory} (also see
 \cite[Section~5]{BGHS18} for another specification of the Heath-Brown method, used for the purposes of WT). This implies

\smallskip

\noindent {\bf Theorem B.} {\it  For any $L\ge2$, 
	\be\non
	\Big|
	G_{s}(\tau, L) - C_d\int_{\Sigma_0} \Phi_s(\tau;z_1,z_2)\,\mu^{\Sigma_0}(dz_1dz_2)\Big| \le 
	K_d \, \frac{\| \Phi_s(\tau; \cdot)\|_{N_1, N_2} }{L^{d-5/2}}\,,
	\ee
	where $\Sigma_0=\{z_1,z_2\in\R^d:\, z_1\cdot z_2=0\}$, 
	$\mu^{\Sigma_0}$ is the measure on it, defined by \eqref{meas} with $s=0$, 
	  $C_d$ is a number-theoretical constant, satisfying $C_d\in(1, 1+2^{2-d})$, the norm $ \|  \cdot \|_{N_1, N_2}$ is defined in \eqref{the_norm} and the 
	  constants $N_1, N_2 \in \N$ depend only on $d$. 
}
\smallskip

In particular, the term $\no_{s,L}^{(2)}(\tau)$ admits a limit  when $L\to\infty$. 

The terms $\no_{s,L}^{(3)}$ and $\no_{s,L}^{(4)}$ in \eqref{N_s_decomp} correspond to multiple intersections of quadrics, 
and the Heath-Brown method 
does not apply to them. Still the diagram technique  allows to write the terms in a convenient compact form. Then
 next in Section~\ref{sec:numbertheory}
and Appendix~\ref{app:intersection} we use Theorem~B jointly with another  powerful result from
 the number theory --  Bezout's theorem for finite fields -- to prove\,\footnote{ In fact, in Section \ref{sec:numbertheory} 
we prove some  abstract results, more general  than the theorem below; see there  Theorems~\ref{t:numbertheory} and \ref{t:countingterms}.}

\smallskip

\noindent {\bf Theorem C.} {\it For $k=3,4$,
	$\ 
	|\no_{s,L}^{(k)}(\tau)|\leq C^\#(s).
	$}
\medskip

 Theorems B and C imply \eqref{i:n est}. 
So to establish Theorem~A it remains to show that the term $\no_{s,L}^{\leq 2}(\tau):=\no_{s,L}^{(0)}(\tau)+\eps^2\no_{s,L}^{(2)}(\tau)$ (or equivalently its limit as $L\to\infty$,  provided by Theorem~B) can be well approximated by a solution of the WKE \eqref{kin eq}.
To this end, following the  lines of \cite{DK1} (and  the logic of the Krylov--Bogolyubov averaging) 
 we consider increments $\Delta\no_{s,L}^{\leq 2}:=\no^{\leq 2}_{s,L}(\tau+\theta)-\no^{\leq 2}_{s,L}(\tau)$ and
  express them through the processes $\ai^{(0)}_m$ via the Duhamel formula \eqref{an} and  the Wick theorem. Then the 
  increments approximately   take the form \eqref{the_form},  and we use 
 Theorem~B  to show that they are close to the r.h.s. of the WKE, multiplied by $\theta$.

Although the computation of the increments $\Delta\no_{s,L}^{\leq 2}$ is similar to that in \cite{DK1}, a mechanism leading to a
 WKE is rather different. Namely, in \cite{DK1} components of the   terms  $\no_{s,L}^{(k)}$ 
 are approximated by  formulas analogous to \eqref{the_form}, where 
 the summation over the lattice $(\Z^d)^k =\{z\}$ is replaced by an integration over $\R^{dk}$.
 The  integrals in those formulae 
  involve fast oscillating Gaussian kernels. The zero sets of these kernels define quadrics, related to the quadrics $\{(z_1, z_2): z_1 \cdot z_2=0\}$ in \eqref{the_form}. 
   Due to the fast oscillations a crucial component of the increments $\Delta\no_{s,L}^{\leq 2}$ is given by the terms,
  associated with short-range correlations in $\tau$ of the processes $\ai_m^{(0)}(\tau)$. On the contrary, in the present situation a crucial contribution 
  is given by other terms, associated with long-time correlations of the processes $\ai_m^{(0)}(\tau)$, while the short-range correlations only give a
  small correction.   As the result, the kinetic integral in the WKE \eqref{kin eq} depends on the viscosity operator $\frak A$, while that in 
  the WKE in \cite{DK1} does not.
  \smallskip

   Finally we note that, 
    as we explain in Section \ref{s:ext_countingterms}, it is plausible that Theorem~C holds for all $k\ge3$. If so, then for $\rho=\eps L$
   the energy spectrum of a solution $\ai(\tau)$, written as \eqref{decompor}, defines 	 a formal series in $\eps$, uniformly in $L\ge2$. Then
   the partial sums of this series, made by the terms of order $\eps^m$, $m\le M$,   with any fixed 
   $M\ge2$, also satisfy Theorem~A with the     constants $C_r$,    depending on $M$.  Cf. Conjecture~\ref{con_3.6}.

 \section{Series expansion: approximating equation and diagrammatic
   representation for solutions}\lbl{sec:diagram}

In this section, assuming that $d\ge2$,  we approximate  processes  
  \eqref{an} by more convenient processes $a^{(n)}$, and then obtain a compact and 
 instrumental representation for their correlations in terms
 of Feynman diagrams (see Lemma~\ref{l:E-fin}), following \cite[Sections~3-5]{DK2}.  This representation (as well as its analogy
 in \cite{DK1, DK2}) is used to estimate various disparity terms, related to quasisolutions $\cA(\tau;L)$, see \eqref{A-intro}, 
 and to their energy spectra.

 Our presentation is sketchy, but missing details may  be found in \cite{DK2}. For a general discussion of the language of  Feynman diagrams see
 \cite{Jan}.

 
\subsection{Approximate $a$-equation}

 We start by considering an approximation of the original equation
\eqref{ku5} by an equation, where  the term $L^{-d}|\ai_s|^2\ai_s$ is
removed:
\begin{equation}\label{ku4}
  \begin{split}
  \dot a_s+\gamma_s a_s = i\rho \cY_s(a)
  +b(s) \dot\beta_s
  \,,\quad s\in{{\mathbb Z}}^d_L  \,,\\
  \cY_s(a)=L^{-d} \sum_{1,2,3} \dess \delta(\omega^{12}_{3s}) 
  a_1 a_2 \bar a_3\,.
  \end{split}
\end{equation}
Similar to  processes $\ai_s$ we decompose
\be\label{decomp}
	a=a^{(0)}+\rho a^{(1)}+  \dots \,.
\ee
Here $a^{(0)} = \ai^{(0)}$ and the processes $a^{(n)}_s(\tau)$ with $n\ge1$ are built by the recursive formula \eqref{an} with the term 
$\ai_s^{n_1} \ai_s^{n_2} \bar\ai_s^{n_3}$  being dropped. I.e. with the nonlinearity $Y_s^{res}$ replaced by the $\cY_s$ above.

Results of \cite{DK2} together with Theorem~\ref{t:countingterms} below (which is an abstract version of Theorem~C from the introduction) imply 
\begin{proposition}\label{p:approx}
  For all $m,n\ge 0$, satisfying $N:=m+n\leq 4$, 
\begin{equation}\lbl{eq:comp_moment}
\big| \EE a^{(m)}_s(\tau_1)\bar a^{(n)}_s(\tau_2) - \EE
 \ai_s^{(m)}(\tau_1)\bar \ai_s^{(n)}(\tau_2) \big|
\le
\frac{L^{-N-d+1}}{\chi_d(L)^{N-1}} C^\#(s;n,m)\,,
\end{equation}
uniformly in $\tau_1,\,\tau_2\ge0$.
\end{proposition}
We prove  the proposition in Appendix~\ref{app:approx}  for $d\ge 3$ and  discuss an adaptation of the proof to  the case
$d=2$ in Appendix~\ref{app:d=2}. Doing  that
we use the relation $N:=m+n\leq 4$ only to apply 
Theorem~\ref{t:countingterms} (or   Theorem~\ref{t:countingterms'} if $d=2$). 
 So if  the assertion \eqref{assertion} of  the latter theorem holds for larger $N$'s, 
  then for those $N$'s estimates \eqref{eq:comp_moment}  remains  true as well (we believe that \eqref{assertion} is fulfilled for all $N$, 
 see in Section~\ref{s:ext_countingterms}).
 
Relations \eqref{eq:comp_moment} imply that moments of  processes $a^{(m)}_s(\tau)$ well approximate those of  processes 
$\ai^{(m)}_s(\tau)$ as $L\to\infty$. Accordingly, from now on we will mostly study  processes $a_s(\tau)$ and their decompositions \eqref{decomp}.

\subsection{Diagrams for solutions}

  For what follows it is convenient to re-write   operator $\cY$ from 
   \eqref{ku4}, using  a fictitious index $s_4$:
 $$
 \cY_s(a)=L^{-d} \sum_{1,2,3,4} \de'^{12}_{34}\,
 \de(\om^{12}_{34})\de^s_{4} \,a_1 a_2 \bar a_3\,,
 $$   
 where $\de_{4}^s$ is the Kronecker symbol.
 Then analogous of the expression \eqref{an} for $a^{(m)}$, $m\geq 1$,  takes the form 
 \be\label{an2}
 \begin{split}
 	a^{(m)}_s(\tau) 
 	= &\sum_{m_1+m_2+m_3=m-1} i \int_{0}^\tau   dl \,e^{-\ga_{s}(\tau-l)} \\
 	& L^{-d}\sum_{1,2,3,4} 
 	 \de'^{12}_{34}\,\de(\om^{12}_{34}) \,
 	\de^s_{4}\big( a_1^{(m_1)}
        a_2^{(m_2)} {\bar a_3}^{(m_3)}\big)(l)\,.
 \end{split}
 \ee
 We will call the objects as those 
  in the r.h.s. of \eqref{an2} {\it sums}, despite they involve
  integrating in $dl$. {The r.h.s.  of \eqref{an2} contains
    several sums, corresponding to all admissible  choices of numbers $m_1,m_2,m_3$.

We apply  Duhamel's formula \eqref{an2} to the terms $a_{s_i}^{(m_i)}(l)$ in the right-hand side of \eqref{an2} 
 with $m_i>0$,  and iterate the procedure till $a_s^{(m)}(\tau)$ is expressed 
  through the processes $a^{(0)}$ and $\bar a^{(0)}$. 
    Then $a_s^{(m)}$ becomes represented  as a finite sum of {\it
      sums}; we denote such sums by $I_s$. Below we will  associate with each sum $I_s$  an appropriately constructed diagram $\gD$. 
  Thus  we will write $a_s^{(m)}(\tau)$ as 
\be
\lbl{a^m=}
a_s^{(m)}(\tau)=\sum_{\gD\in\gD_m} I_s(\gD;\tau),
\ee
where $\gD_m$ is a set of all diagrams, corresponding to  the just explained 
representation of  $a^{(m)}$ via the processes  $a^{(0)}$ and $\bar a^{(0)}$.
Similarly by $\ov\gD_n$ we denote the set of diagrams, 
parametrizing the terms in the sum, representing  $\bar a^{(n)}(\tau)$ in a form, analogous to \eqref{a^m=}: 
$\bar a_s^{(n)}(\tau)=\sum_{\bar \gD\in\ov\gD_m} I_s(\bar \gD;\tau).$ 

\subsubsection{Construction of the sets of diagrams $\gD_m$ and $\ov\gD_n$.}

 We start with discussing the set  $\gD_2$ and the sums $I_s(\gD)$ with $\gD\in\gD_2$.

When iterating the Duhamel formula \eqref{an2} (or its complex conjugation) for a $j$-th time, we will denote the corresponding  time 
 $l\in [0, \tau] $ by $l_j$ and  will write the set of indices 
 $\{s_1,s_2,s_3,s_4\}$ as $\{\xi_{2j-1},\xi_{2j},\s_{2j-1},\s_{2j}\}$, 
 where we enumerate by  $\xi_i$ the indices of   non-conjugated variables $a_{s'}^{(k)}$ in \eqref{an2} and  by  $\s_i$~-- those 
 of   conjugated variables $\bar a_{s''}^{(n)}$. 
 We  write the corresponding fictitious index 
  $s_4$  as $\s_{2j}$ if we  apply  \eqref{an2}, or as $\xi_{2j}$ if we  apply the complex conjugation of \eqref{an2}. 
 More precisely, when applying \eqref{an2} we denote $s_1=\xi_{2j-1}$, $s_2=\xi_{2j}$, $s_3=\s_{2j-1}$ and $s_4=\s_{2j}$, 
 and when applying its complex conjugation, we write  $s_1=\s_{2j-1}$, $s_2=\s_{2j}$, $s_3=\xi_{2j-1}$ and 
$s_4=\xi_{2j}$.
 We will abbreviate 
\be\lbl{de,om}
\de_j=\de'^{\xi_{2j-1} \xi_{2j}}_{\s_{2j-1}\s_{2j}} \qnd 
\om_j=\om^{\xi_{2j-1}\xi_{2j}}_{\s_{2j-1}\s_{2j}}, \qu j\geq 1,
\ee
(these terms correspond to   $\de'^{12}_{34}$ and $\om^{12}_{34}$ in  \eqref{an2}). We also set
\be\lbl{xi,s_0}
\xi_0=\s_0:=s.
\ee
Applying \eqref{an2} to $a_s^{(2)}$ and 
using the notation above with $j=1$ we find
\be\lbl{a^2=}
\begin{split}
	a^{(2)}_{s}(\tau)=a^{(2)}_{\xi_0}(\tau) 
	&= \sum_{m_1+m_2+m_3=1} i \int_{0}^\tau   dl_1 \,e^{-\ga_{\xi_0}(\tau-l_1)} \\
	&L^{-d}\sum_{\xi_1,\xi_2,\s_1,\s_2} 
	\de_1\,\de(\om_1) \,
	\de^{\xi_0}_{\s_2}\big( a_{\xi_1}^{(m_1)}
	a_{\xi_2}^{(m_2)} {\bar a_{\s_1}}^{(m_3)}\big)(l_1)\,.
\end{split}
\ee
Let us consider the summand with $m_1=m_2=0$ and $m_3=1$.
Applying the conjugated formula \eqref{an2} to $\bar a_{\s_1}^{(1)}$ 
and using the introduced  notation   with $j=2$ we get
\be\lbl{a^1=}
\begin{split}
	\bar a^{(1)}_{\s_1}(l_1) 
	= &-i \int_{0}^{l_1}   dl_2 \,e^{-\ga_{\s_1}(l_1-l_2)} \\
	& L^{-d}\sum_{\xi_3,\xi_4,\s_3,\s_4}
	\de_2\,\de(\om_2) \,
	\de^{\s_1}_{\xi_4}\big( a_{\xi_3}^{(0)}
	\bar a_{\s_3}^{(0)} {\bar a_{\s_4}}^{(0)}\big)(l_2)\,.
\end{split}
\ee
Inserting \eqref{a^1=} into the summand in \eqref{a^2=} with  $m_1=m_2=0$ and $m_3=1$, we get a sum 
$I_s(\gD;\tau)$ which we associate with the diagram $\gD$ from fig.~\ref{f:D(2)}(c); further on we will denote this diagram by $\gD^c$.
\begin{figure}[t]
	a)\qu
	\parbox{3cm}{ 
		\begin{tikzpicture}[]
		\node at (-1,0) (c0) {$c_0^{(2)}$};
		
		\node at (-2.1, -1.2) (c1) {$c_1^{(1)}$};
		\node at (-1.4, -1.2)  {$c_2^{(0)}$};
		\node at (-0.7, -1.2)  {$\bar c_1^{(0)}$};
		\node at (0, -1.3) (w2)  {$\bar w_{2}$};  
		
		\node at (-2.1, -2.5)  {$c_3^{(0)}$};
		\node at (-1.4, -2.5)  {$c_4^{(0)}$};
		\node at (-0.7, -2.5)  {$\bar c_3^{(0)}$};
		\node at (0, -2.6) (w4) {$\bar w_{4}$}; 
		
		\draw [line width=0.25mm] (c0.south)--(w2.north);
		\draw [line width=0.25mm](c1.south)--(w4.north);
		
		\end{tikzpicture}
	}
	\hfill
	b) \qu\parbox{3cm}{
		\begin{tikzpicture}[]
		\node at (-1,0) (c0) {$c_0^{(2)}$};
		
		\node at (-2.1, -1.2)  {$c_1^{(0)}$};
		\node at (-1.4, -1.2) (c2) {$c_2^{(1)}$};
		\node at (-0.7, -1.2)  {$\bar c_1^{(0)}$};
		\node at (0, -1.3) (w2)  {$\bar w_{2}$};  
		
		\node at (-2.1, -2.5)  {$c_3^{(0)}$};
		\node at (-1.4, -2.5)  {$c_4^{(0)}$};
		\node at (-0.7, -2.5)  {$\bar c_3^{(0)}$};
		\node at (0, -2.6) (w4) {$\bar w_{4}$}; 
		
		\draw [line width=0.25mm] (c0.south)--(w2.north);
		\draw [line width=0.25mm](c2.south)--(w4.north);
		
		\end{tikzpicture}
	}
	\hfill
	c) \qu \parbox{3cm}{
		\begin{tikzpicture}[]
		\node at (-1,0) (c0) {$c_0^{(2)}$};
		
		\node at (-2.1, -1.2)  {$c_1^{(0)}$};
		\node at (-1.4, -1.2)  {$c_2^{(0)}$};
		\node at (-0.7, -1.2) (bc1) {$\bar c_1^{(1)}$};
		\node at (0, -1.3) (w2)  {$\bar w_{2}$};  
		
		\node at (-2.1, -2.5)  {$c_3^{(0)}$};
		\node at (-1.4, -2.6) (w4) {$w_4$};
		\node at (-0.7, -2.5)  {$\bar c_3^{(0)}$};
		\node at (0, -2.5)  {$\bar c_{4}^{(0)}$}; 
		
		\draw [line width=0.25mm] (c0.south)--(w2.north);
		\draw [line width=0.25mm](bc1)--(w4.north);
		
		\end{tikzpicture}
	}
	\caption{The set of diagrams $\gD_2.$}
	\lbl{f:D(2)}
\end{figure}
The  {\it non-conjugated} vertices $c_i^{(k)}$ of the diagram are associated with the variables $a_{\xi_i}^{(k)}$ in 
\eqref{a^2=}, \eqref{a^1=};  the corresponding to them indices are $\xi_i$. The {\it conjugated} vertices $\bar c_j^{(n)}$  are associated with the variables $\bar a_{\s_j}^{(n)}$ and the corresponding indices are $\s_j$. In particular, the root $c_0^{(2)}$ is associated with 
$a_{\xi_0}^{(2)}=a_{s}^{(2)}$ and the corresponding index is $\xi_0$. 
In the notation $c_i^{(k)}$ and $\bar c_j^{(n)}$ we sometimes 
omit the upper indices $k$ and $n$ which we call the {\it degrees} of the vertices $c_i^{(k)}$, $\bar c_j^{(n)}$.
The vertices $\bar w_{2}$ and $w_{4}$ are called conjugated (non-conjugated) {\it  virtual} vertices and the corresponding indices are $\s_2$ and $\xi_4$; these vertices
 are associated with the Kronecker symbols $\de_{\s_2}^{\xi_0}$ and $\de_{\xi_4}^{\s_1}$ in \eqref{a^2=} and \eqref{a^1=}.
Vertices which are not virtual are called {\it real.} 
Every edge of the diagram couples a non-conjugated (conjugated) vertex $c_i^{(k)}$ ($\bar c_i^{(k)}$) of positive degree $k\geq 1$ with a conjugated (non-conjugated) virtual vertex $\bar w_{i'}$ ($ w_{i'}$). It is associated with an 
application of formula \eqref{an2}  (or its complex conjugation)
 to the variable $a_{\xi_i}^{(k)}$ (or $\bar a_{\s_i}^{(k)}$), corresponding to the vertex $c_i^{(k)}$ (or $\bar c_i^{(k)}$).

The set of four vertices 
\be\lbl{block}
c_{2j-1}, c_{2j}, \bar c_{2j-1}, \bar w_{2j} \qmb{or} \qu c_{2j-1}, w_{2j}, \bar c_{2j-1}, \bar c_{2j}
\ee 
(in dependence whether the virtual vertex is conjugated or not) to which correspond the indices $\xi_{2j-1},\xi_{2j},\s_{2j-1},\s_{2j}$
is called  {\it the   $j$-th block\,}; the diagram $\gD^c$ has two blocks. The index $i$ of a virtual vertex $w_i$ ($\bar w_i$) is always pair, $i=2j$.  
 Each block corresponds to an application of formula \eqref{an2} (or its complex conjugation) to its {\it parent}, i.e. to the vertex of positive degree coupled with the virtual vertex of the block. 
The virtual vertex is conjugated if the parent is non-conjugated and the other way round.
The {\it time variable} $l_j$ is associated with the $j$-th block.

The  {\it leaves} are the vertices of zero degree, that is, the vertices
$c_i^{(0)}$ and $\bar c_j^{(0)}$.

The diagrams from fig.~\ref{f:D(2)}(a,b) correspond to the summands in \eqref{a^2=} with 
$m_1=1, \, m_2=m_3=0$ and $m_1=m_3=0, \, m_2=1$; they are constructed by the same rules as the diagram $\gD^c$. The three diagrams from fig.~\ref{f:D(2)} form the set $\gD_2$. The set of diagrams $\ov\gD_2$, corresponding to $\bar a_{s}^{(2)}(\tau)$, is obtained by conjugating the vertices in the
three diagrams above and re-ordering the elements of each block in such a way that the pair of non-conjugated vertices is followed by the pair of conjugated vertices, i.e. the blocks have the form \eqref{block}. 

The sets $\gD_m$ and $\ov\gD_n$ with arbitrary $m,n \ge0$ and the diagrams which are their elements, are constructed similarly. Namely, the 
sets $\gD_0$ and $\ov\gD_0$ are trivial -- they contain one diagram each, made by the root $c_0^{(0)}$ (or $\bar c_0^{(0)}$).  The sets 
 $\gD_1$ and $\ov\gD_1$ also  contain only one diagram each; 
  e.g. the  diagram in  $\gD_1$ consists of the root $c_1^{(0)}\!$, joint by an edge with 
 $\bar w_2$ in the only block $B_1=( c_1^{(0)}, c_2^{(0)}, \bar c_1^{(0)}, \bar w_2)$.  
 Arbitrary sets  $\gD_m$ and $\ov\gD_n$ may be constructed
 by induction. Indeed, consider a process $a^{(m+1)}(\tau)$ with $m\ge 1$  and apply to it 
\eqref{an2} with $m:=m+1$. In the r.h.s. of \eqref{an2} the sum in $m_1, m_2, m_3$ contains $(m+2)(m+1)/2$ terms. Consider any one of them,
\be\label{term}
 i \int_{0}^\tau   dl \,e^{-\ga_{s}(\tau-l)} 
  L^{-d}\sum_{1,2,3,4} 
 \de'^{12}_{34}\,\de(\om^{12}_{34}) \,
 \de^s_{4}\big( a_1^{(m_1)}
 a_2^{(m_2)} {\bar a_3}^{(m_3)}\big)(l),
 \ee
 draw the block 
 $
 B_1=( c_1^{(m_1)}, c_2^{(m_2)}, \bar c_1^{(m_3)}, \bar w_2),
 $
 and join $\bar w_2$ by an edge with the root $c_0^{(m+1)}$. Next   consider the sets $\gD_{m_1}, \gD_{m_2}, \ov\gD_{m_3}$ and do the following:
 
 a) Firstly  take  $\gD_{m_1}$. If $m_1=0$,  do nothing. Otherwise  choose any diagram 
 $
 \gD^1 \in \gD_{m_1},
  $
   place it below $ c_1^{(m_1)}$ and identify its root with $ c_1^{(m_1)}$. Do this for each diagram in $\gD_{m_1}$, thus obtaining $|\gD_{m_1}|$ diagrams with
   roots in $c_0^{(m+1)}$. 
  
 b) Then  consider the set $\gD_{m_2}$ and do the same  with the just obtained $|\gD_{m_1}|$ diagrams, 
 identifying their roots with the vertex $c_{2}^{(m_2)}$, and next -- the set  $ \ov\gD_{m_3}$,
  identifying the roots with $\bar c_{3}^{(m_3)}$. 
 
 c) It remains to convert thus obtained  $|\gD_{m_1}|\times |\gD_{m_2}|\times |\gD_{m_3}|$ diagrams to elements of the set  $\gD_{m+1}$ by  re-numerating properly their
 blocks and accordingly re-numerating the vertices in the blocks as in \eqref{block}. Do this by numerating the blocks from top to the bottom and from
  left to right, as in the examples above with $m=2$. 
  
  d) Doing the same for all  blocks, corresponding to  all possible $(m+2)(m+1)/2$ terms \eqref{term},  get the diagrams, forming the set $\gD_{m+1}$.

The set  $\ov\gD_{m+1}$ is constructed inductively in the same way. 
 
 \smallskip 
 
For further needs we note that due to the factors $\de'^{12}_{34}$ and  $\de_{s_4}^s$ in \eqref{an2}, 
the indices $\xi_i,\s_j$ entering the formula for the sums $I_s(\gD)$ from \eqref{a^m=}  satisfy the relations
\be\lbl{lin-rel} 
\begin{split}
&1)\ \de'^{\xi_{2j-1} \xi_{2j}}_{\s_{2j-1}\s_{2j}}=1 \;\forall j, \qu \\
&2)\ \mbox{indices $\xi_i$, $\s_j$ corresponding to adjacent in $\gD$ vertices are equal.}
\end{split}
\ee

\subsection{Feynmann diagrams for expectations}
The main objects we are interested in are the  correlations  $\EE a_{s_1}^{(m)}(\tau_1)\bar a_{s_2}^{(n)}(\tau_2)$. 
It can be shown that they vanish if $s_1\ne s_2$.\footnote{As well vanish the correlations   $\EE a_{s_1}^{(m)}(\tau_1) a_{s_2}^{(n)}(\tau_2)$ and
$\EE \bar a_{s_1}^{(m)}(\tau_1)\bar a_{s_2}^{(n)}(\tau_2)$ for all $s_1,s_2$.
}
To represent an expectation $\EE a_{s}^{(m)}(\tau_1)\bar a_{s}^{(n)}(\tau_2)$ we consider 
 the set of diagrams 
$$
\gD_m\times\ov\gD_n:=\{\gD^1\sqcup \bar\gD^2:\, \gD^1\in\gD_m, \bar\gD^2\in\ov\gD_n\}.
$$
Here a diagram $\gD^1\sqcup\bar \gD^2$ is obtained by drawing  $\gD^1$ and $\bar \gD^2$ side by side, where the 
 blocks of $\gD^1$ are enumerated  from $1$ to $m$,
 while those  of  $\bar\gD^2$ together with the corresponding  time variables $l_j$  are 
 enumerated from $j=m+1$ to $m+n$. The vertices together with the corresponding indices $\xi_{2j-1},\xi_{2j},\s_{2j-1},\s_{2j}$ are enumerated
 accordingly, see fig.~\ref{f:D(m,n)}. 
 The diagram $\gD^1\sqcup\bar \gD^2$ has two roots $c_0^{(m)}$ and $\bar c_0^{(n)}$. For any  $\gD=\gD^1\sqcup \bar\gD^2$  consider
 \be\non
I_s(\gD;\tau_1,\tau_2)=I_s(\gD^1;\tau_1)I_s(\bar\gD^2;\tau_2),
 \ee
 so that 
 $
 a_{s}^{(m)}(\tau_1)\bar a_{s}^{(n)}(\tau_2) = \sum_{\gD\in\gD_m\times\ov\gD_n} 
 I_s(\gD;\tau_1,\tau_2).
$
 Our next task is  to compute $\EE I_s(\gD)$ for each
$\gD\in\gD_m\times\ov\gD_n$. 
\begin{figure}[t]
	a)\qu
	\parbox{4cm}{ 
		\begin{tikzpicture}[]
		\node at (-1,0) (c0) {$c_0^{(2)}$};
		\node at (1,0)  {$\bar c_{0}^{(0)}$};
		
		\node at (-2.1, -1.2) (c1) {$c_1^{(1)}$};
		\node at (-1.4, -1.2)  {$c_2^{(0)}$};
		\node at (-0.7, -1.2)  {$\bar c_1^{(0)}$};
		\node at (0, -1.3) (w2)  {$\bar w_{2}$};  
		
		\node at (-2.1, -2.5)  {$c_3^{(0)}$};
		\node at (-1.4, -2.5)  {$c_4^{(0)}$};
		\node at (-0.7, -2.5)  {$\bar c_3^{(0)}$};
		\node at (0, -2.6) (w4) {$\bar w_{4}$}; 
		
		\draw [line width=0.25mm] (c0.south)--(w2.north);
		\draw [line width=0.25mm](c1.south)--(w4.north);
		
		\end{tikzpicture}
	}
	\hfill
	b)\qu \parbox{6cm}{
		\begin{tikzpicture}
		\path (-1.2,0) node(a0) {$c_0^{(1)}$}
		(-0.2, -1.3) node(b_w2) {$\bar w_{2}$}; 
		\draw [line width=0.25mm] (a0.south)--(b_w2.north);
		\node at (-0.9, -1.2)  {$\bar c_1^{(0)}$};
		\node at (-1.6, -1.2)  {$c_2^{(0)}$};
		\node at (-2.3, -1.2) {$c_1^{(0)}$};

		\node at (1.8,0) (c0) {$\bar c_0^{(2)}$};
		
		\node at (0.8, -1.2)  (c1) {$c_3^{(1)}$};
		\node at (1.5, -1.3) (w2) {$w_4$};
		\node at (2.2, -1.2)  {$\bar c_3^{(0)}$};
		\node at (2.9, -1.2)  {$\bar c_{4}^{(0)}$};  
		
		\node at (0.8, -2.5)  {$c_5^{(0)}$};
		\node at (1.5, -2.5)  {$c_6^{(0)}$};
		\node at (2.2, -2.5)  {$\bar c_5^{(0)}$};
		\node at (2.9, -2.6) (w4) {$\bar w_{6}$}; 
		
		\draw [line width=0.25mm] (c0)--(w2.north);
		\draw [line width=0.25mm](c1.south)--(w4.north);
		\end{tikzpicture}
	}
	\caption{A diagram from the set a) $\gD_2\times \ov\gD_0$ and b) $\gD_1\times\ov\gD_2$.}
	\lbl{f:D(m,n)}
\end{figure}

Randomness enters the term $I_s(\gD)$ via the random variables $a_{\xi_i}^{(0)}$, $\bar a_{\s_j}^{(0)}$, corresponding to the leaves of the diagram 
$\gD=\gD^1\sqcup\bar\gD^2$. They are Gaussian with correlations
\be\label{corr_a_in_time}
	\EE a_s^{(0)}(l_1) a_{s'}^{(0)}(l_2)=0, \quad
	 \EE a_s^{(0)}(l_1) \bar a_{s'}^{(0)}(l_2)
	= \delta^s_{s'} \,\operatorname{Corr}(\ga_s,b(s),l_1,l_2),
\ee
where
\be\lbl{corr()}
\operatorname{Corr}(\ga_s,b(s),l_1,l_2) ={ B_s} \big( e^{-\gamma_s|l_1-l_2|} -
	e^{-\ga_s(l_1+ l_2)}\big),{ \quad B_s = \frac{b(s)^2}{\gamma_s}} \,.
\ee
So the Wick theorem \cite{Jan} implies that the expectation $\EE I_s(\gD)$ is given by a sum over all Wick pairings of variables $a_{\xi_i}^{(0)}$,
 corresponding to non-conjugated leaves $c_i^{(0)}$,  with variables $\bar a_{\s_j}^{(0)}$ corresponding to conjugated leaves  
$\bar c_j^{(0)}$. Moreover, the leaves $c_i^{(0)}$ and $\bar c_j^{(0)}$ should belong to different blocks since otherwise the  summand corresponding to such Wick pairing vanishes due to \eqref{notethat} and 
 item (1) in \eqref{lin-rel}. 
We parametrize the sum over the Wick pairings by the defined below set  $\gF(\gD)$ of Feynman diagrams.
 Denoting by $ J_s(\gF)$ a term (i.e. a sum), corresponding to a specific Feynman diagram $\gF$, we have:
$$
\EE I_s(\gD)=\sum_{\gF\in\gF(\gD)} J_s(\gF)\,.
$$

\subsubsection{Definition of Feynman diagrams}
\lbl{s:FD_def}

To construct the set of Feynman diagrams $\gF(\gD)$, corresponding to some diagram $\gD=\gD^1\sqcup \bar\gD^2$, 
 we consider all possible partitions of the set of leaves of $\gD$
 to non-intersecting pairs  $(c_i^{(0)}, \bar c_j^{(0)})$, such that the paired leaves $c_i^{(0)}$ and $\bar c_j^{(0)}$ do not belong to the same block.
To each such partition we associate a diagram $\gF$ obtained from $\gD$ by joining with  an edge  the two leaves in every pair, see fig.~\ref{f:Feynman}(a). 
 \begin{figure}[t]
 	a)\qu	\parbox{5cm}{ 
 		\begin{tikzpicture}[]
 		\node at (-1,0) (c0) {$c_0^{(2)}$};
 		\node at (1,0) (bc0) {$\bar c_{0}^{(0)}$};
 		
 		\node at (-2.1, -1.2) (c1) {$c_1^{(1)}$};
 		\node at (-1.4, -1.2) (c2) {$c_2^{(0)}$};
 		\node at (-0.7, -1.2) ( bc1) {$\bar c_1^{(0)}$};
 		\node at (0, -1.3) (bw2)  {$\bar w_{2}$};  
 		
 		\node at (-2.1, -2.5) (c3) {$c_3^{(0)}$};
 		\node at (-1.4, -2.5) (c4) {$c_4^{(0)}$};
 		\node at (-0.7, -2.5) (bc3) {$\bar c_3^{(0)}$};
 		\node at (0, -2.6) (bw4) {$\bar w_{4}$}; 
 		
 		\draw [line width=0.25mm] (c0.south)--(bw2.north);
 		\draw [line width=0.25mm](c1.south)--(bw4.north);
 		\draw [line width=0.25mm] (c3.south) -- (-2.1,-3) --(1,-3) --  (bc0.south);
 		\draw [line width=0.25mm](c2.south)--(-0.9, -2.2);
 		\draw [line width=0.25mm](c4.north)--(-0.9, -1.5);
 		\end{tikzpicture}
 	}
 	\hfill
 	b)\qu\parbox{5cm}{ 
 		\begin{tikzpicture}
 		\node at (-1,0) (c0) {$c_0^{(2)}$};
 		\node at (1,0) (bc0) {$\bar c_{0}^{(0)}$};
 		
 		\node at (-2.5, -1.2) (c1) {$c_1^{(1)}$};
 		\node at (-1.6, -1.2) (c2) {$c_2^{(0)}$};
 		\node at (-0.5, -1.2) ( bc1) {$\bar c_1^{(0)}$};
 		\node at (0.4, -1.3) (bw2)  {$\bar w_{2}$};  
 		
 		\node at (-2.5, -2.5) (c3) {$c_3^{(0)}$};
 		\node at (-1.6, -2.5) (c4) {$c_4^{(0)}$};
 		\node at (-0.5, -2.5) (bc3) {$\bar c_3^{(0)}$};
 		\node at (0.4, -2.6) (bw4) {$\bar w_{4}$}; 
 		
 		\draw [line width=0.25mm] (-0.8,-0.3)--  (0.35, -1) ;
 		\draw [line width=0.25mm](c1.south)--(bw4.north);
 		\draw [line width=0.25mm] (c3.south) -- (-2.5,-3.2) --  (1,-3.2) --   (bc0.south);
 		\draw [line width=0.25mm](c2.south)--(-0.9, -2.2);
 		\draw [line width=0.25mm](c4.north)--(-0.8, -1.5);
 		
 		\draw [dashed,line width=0.25mm](-0.8,-0.1)--(0.6,-0.1);
 		\draw [dashed,line width=0.25mm](-1.5, -1.3)--(-0.9, -1.3);
 		\draw [dashed,line width=0.25mm] (-2.2, -0.9) .. controls (-1.1,-0.4) .. (0.2, -1.1);
 		\draw [dashed,line width=0.25mm] (-2.4, -2.8) .. controls (-1.7,-3) .. (-0.7, -2.8);
 		\draw [dashed,line width=0.25mm] (-1.5, -2.8) .. controls (-0.8,-3) .. (0.2, -2.8);
 		\end{tikzpicture}
 	}
 	\hfill
 	\caption{
 		a) A Feynman diagram $\gF$ obtained from the diagram $\gD$ in fig.~\ref{f:D(m,n)}(a). 
 		\ b) A cycle obtained from the Feynman diagram $\gF$.
 	}
 	\lbl{f:Feynman}
 \end{figure}
So  in each  diagram $\gF \in  \gF(\gD)$  all   vertices  of $\gD$
  are joint by edges,   every edge couples  a  conjugated  vertex with a non-conjugated in another block (or with a root), and every vertex belongs to exactly one edge. 
  
  Since $\EE a_s^{(0)}\bar a_{s'}^{(0)}=0$ if $s\ne s'$, the indices $\xi_i,\s_j$ entering the formulas for  sums $J_s(\gF)$ satisfy \eqref{lin-rel},
   where in item (2) the diagram $\gD$ is replaced by the Feynman diagram $\gF$; below we denote these relations as \eqref{lin-rel}$_\gF$.
 In particular, due to the item (2), the vector of indices $\s=(\s_i)$ is a function of the vector $\xi=(\xi_j)$. Accordingly  below we write $\s=\s_\gF(\xi)$.

 
  Let
 $$
 \gF_{m,n}=\bigcup_{\gD\in\gD_m\times\ov\gD_n}\gF(\gD)
 $$
 be the set of all  Feynman diagrams associated with the product $a_s^{(m)}\bar a_s^{(n)}.$
  Each diagram $\gF\in\gF_{m,n}$ has $N:=m+n$ blocks and $4N+2$ vertices, including   $2N+2$ leaves Half of edges (and of leaves) 
  are conjugated, while another half is not. 
 
 By construction a diagram $ \gF \in  \gF_{m,n}$  never pairs leaves from the same block. This alone  does not exclude that $ \gF $ is
 such that in \eqref{lin-rel}$_\gF$ the assumptions (1) and (2) are incompatible since for some $j$ we may have 
 $\xi_{2j-1}, \xi_{2j}=\sigma_{2j-1}$ or $\sigma_{2j}$ once $\s=\s_\gF(\xi)$. Analysis shows that this cannot happen if $m+n\le4$, but  may happen if
 $m+n\ge 5$. Accordingly, we denote by
 $$
  \gF_{m,n}^{true} \subset  \gF_{m,n}
 $$
 the set of Feynman diagrams for which the set of indices $\xi_i,\,\s_j$ satisfying the relations \eqref{lin-rel}$_\gF$ is not empty. For any diagram $\gF\notin\gF_{m,n}^{true} $ we have $J_s(\gF)=0$  due to the factors $\de'^{12}_{34}$ and  $\de_{s_4}^s$ in \eqref{an2}.


\subsection{Transformation, resolving linear relations on indices}
\lbl{s:transform}

Let us take a Feynman diagram $\gF\in\gF_{m,n}$, denote $N:=m+n\geq 1$ and consider the sum  $J_s(\gF)$.
The relations \eqref{lin-rel}$_\gF$ on indices $\xi_i,\s_j\in\Z^d_L,$ $0\leq i,j\leq 2N$,\,\footnote{Recall that 
$\xi_0=\s_0=s$.} 
entering the formula for $J_s(\gF)$ are involved,  which makes the sum $J_s(\gF)$ difficult for further analysis. 
In \cite{DK2} it was found a convenient way to ``integrate the sums  $J_s(\gF)$", i.e. to 
 parametrise the indices $\xi_i,\s_j$ by  $N$-vector $z=(z_1,\ldots,z_N)$ from a domain in  $(\Z_L^d)^N$, free from any relations on its components.  In this section we present this parametrisation, referring the reader to \cite{DK2} for a proof. 

Since item (2) of \eqref{lin-rel}$_\gF$ is equivalent to the relation $\s=\s_\gF(\xi)$, it suffices to parametrise the  set 
of {\it admissible} multi-indices $\xi$, i.e. of those $\xi$ for which the multi-indices $\xi$ and $s=\s_\gF(\xi)$ satisfy item (1) in \eqref{lin-rel}$_\gF$.
  The construction starts with defining for 
 each $\gF \in  \gF_{m,n}$   a skew-symmetric $N\times N$ {\it incidence} matrix $\alpha^\gF =(\alpha^\gF_{ij})$ whose elements are integers from the set
 $\{0, +1, -1\}$. In terms of this matrix we define the set of polyvectors 
 \be\label{ZZ}
 \cZ(\gF) = \{ z= (z_1, \dots, z_N)
 \in (\Z_L^{d})^{N} :\, z_j \ne0 \;\text{and}\; (\alpha^\gF z)_j\ne 0 \quad  \forall\,j\}\,.
 \ee
 Here and below for an $M\times N$-matrix $A$ we denote by $Az$ the polyvector with components $(Az)_j:=\sum_{i}A_{ji}z_i\in\Z^d$.\,\footnote{That
 is, abusing notation   we denote by   $A$ an operator in $(\Z^d)^M$  with the block-matrix 
 $A\otimes  \mathbbm{1} $.}
  The matrix $\al^\gF$  has no zero rows and zero columns if and only if $\gF\in\gF^{true}_{m,n}$, and, accordingly, the set
 $\cZ(\gF)$ is non-empty if and only if  $\gF\in \gF_{m,n}^{true}$.  
 Next, it  turns out that the  vectors $z\in  \cZ(\gF) $ may be used to parametrize  the set of admissible 
 indices  $\xi_i$ by means of an affine mapping 
 \be\lbl{z->xi}
 \xi(z)= s+ A^{ \gF}  z ,\qquad z\in \cZ(\gF).
 \ee
 Here $A^{ \gF} $ is an $(2N+1)\times N$-matrix, whose elements again are integers from the set  $\{0, +1, -1\}$. 
 Transformation \eqref{z->xi} provides a presentation of the terms $J_s(\gF)$, forming the correlation  $\EE a_{s}^{(m)}(\tau_1)\bar a_{s}^{(n)}(\tau_2)$, and so for 
 the correlation itself. The corresponding result is proved in Theorem~5.5 of \cite{DK2}. In our setting its statement, where for the function 
 $\theta(x,t)$ is chosen $\mI_{\{0\}}(x)$ -- the indicator function of the point $x=0$ --  takes the following form:

\begin{lemma}\lbl{l:E-fin}
  For any integers $m,n\geq 0$ satisfying $N=m+n\geq 1$, any $s\in\Z^d_L$ and\
  $\tau_1,\tau_2\geq 0$, 
  
  1) for each $\gF\in \gF_{m,n}^{true}$   parametrisation \eqref{z->xi}  (depending on $s$ and $\gF$) 
  is such that the quantity $\omega_j$ in \eqref{de,om}, written in the $z$-coordinates, takes the form 
\be\lbl{om_j^F(z)}
\om_j^\gF(z)=2z_j\cdot \sum_{i=1}^N \al_{ji}^\gF z_i= 2z_j\cdot (\al^\gF z)_j. 
\ee 

2) We have 
  \be\lbl{final_Ea^ma^n}
L^N\EE a^{(m)}_s(\tau_1) \bar a^{(n)}_s(\tau_2)
  = \sum_{\gF\in\gF^{\,true}_{m,n}} c_\gF  J_s(\tau_1,\tau_2; \gF),
  \ee
 where the constants  $c_\gF\in\{\pm 1,\pm i\}$ and
 \be
 \lbl{J(F)-z}
 J_s(\tau_1,\tau_2; \gF)=
 \int_{\R^N}
 dl\,
 L^{N(1-d)}\sum_{z\in\cZ(\gF), \  \omega_j^\gF (z) =0 \, \forall j}\,F_s^\gF(\tau_1,\tau_2,l,z)\,.
\ee
 The density $F^\gF_s(\tau_1,\tau_2,l,z)$ is a real  function,  smooth in  $(s,z)\in\R^d\times\R^{dN}$,    satisfying 
\be\lbl{F-prop1}
|\p_{s}^\mu\p_{z}^{\ka}F_s^\gF(\tau_1,\tau_2,l,z)|\leq C_{\mu,\ka}^{\#}(s)C_{\mu,\ka}^{\#}(z)\,e^{-\de \big(\sum_{i=1}^m|\tau_1-l_i|                                    
        +\sum_{i=m+1}^N|\tau_2-l_i|\big)}
\ee
with a suitable $\delta=\de_N>0$, for any vectors $\mu\in(\N\cup\{0\})^d$, $\ka\in(\N\cup\{0\})^{
  dN}$ and any $s\in\R^d$, $z\in\R^{dN}$, $l\in\R^N$.
\end{lemma}

Let us briefly explain the way to construct  the parametrization \eqref{z->xi}. We first add to the Feynman diagram $\gF$ dashed edges that couple non-conjugated vertices with conjugated {\it inside} all blocks, as in fig.~\ref{f:Feynman}(b). For each block there are two ways of doing that. We prove that there exists a choice (possibly, not unique) of a  dashed edge in each block 
  such that the diagram becomes a cycle, as in fig.~\ref{f:Feynman}(b).  Then, for each $j$ we set $x_{2j-1}:=\xi_{2j-1}-\s_{2j-1}$ and $x_{2j}:=\xi_{2j}-\s_{2j}$ or $x_{2j-1}=\xi_{2j-1}-\s_{2j}$ and $x_{2j}:=\xi_{2j}-\s_{2j-1}$, according to the choice of the dashed edges in the $j$-th block, where we substitute $\s=\s_\gF(\xi)$.
The fact that the Feynman diagram with added dashed edges forms a cycle implies that the transformation $\xi\mapsto x$ is invertible. Item (1) of \eqref{lin-rel}$_\gF$ implies that $x_{2j}=-x_{2j-1}$. Then we set $z_j:=x_{2j-1}$ and get \eqref{z->xi}. 
The incidence  matrix $\al^\gF$  also is constructed in terms of this cycle.

Since the choice of the dashed edges in general is not unique, the parametri\-za\-tion $z\mapsto\xi$ is not unique as well. However, if $z'\mapsto\xi$ is another parametrization,  obtained by the procedure above, and $\al^{\gF\,\prime}$ is the associated incidence 
matrix, then for each $j$ we have either $z_j'(\xi)=z_j(\xi)$ or $z_j'(\xi)=(\al^\gF z(\xi))_j$. In the latter case we also have the symmetric relation $z_j(\xi)=(\al^{\gF\,\prime} z'(\xi))_j$.

	Computing in \eqref{J(F)-z} the integral over $dl$ and using estimate \eqref{F-prop1}, we obtain a form of  integrals~$J_s$, 
	more convenient for some of  the subsequent analysis:
	\begin{corollary}\lbl{cor:Diag}
		In terms of Lemma~\ref{l:E-fin}, the integrals $J_s$ from \eqref{final_Ea^ma^n} can be written as 
		\be\lbl{corr-sums}
		 J_s(\tau_1,\tau_2;\gF) = L^{N(1-d)}\sum_{z\in\cZ(\gF), \  \omega_j^\gF (z) =0 \, \forall j}\,\Phi_s^\gF(\tau_1,\tau_2,z)\,,
		\ee
		where the real-valued functions $\Phi_s^\gF$ are Schwartz in $(s,z)$ and satisfy
		\be\lbl{cor_diag_est}
			|\p_{s}^\mu\p_{z}^{\ka}\Phi_s^\gF(\tau_1,\tau_2,z)|\leq C_{\mu,\ka}^{\#}(s)C_{\mu,\ka}^{\#}(z),             
		\ee
			uniformly in $\tau_1,\tau_2\geq 0$, for any vectors $\mu\in(\N\cup\{0\})^d$ and $\ka\in(\N\cup\{0\})^{
			dN}$.
	\end{corollary}

 \section{Main estimates for the sums}\lbl{sec:numbertheory}
In this section we focus  on  estimates for the sums  \eqref{corr-sums}
and on their dependence on $L$ and $N$.
 We recall that $d\ge3$.
It is convenient to study the problem we consider in the following abstract setting.
Let $\al=(\al_{ij})$, $N\geq 2$, be an $N\times N$ skew-symmetric matrix whose elements belong to the set $\{-1, 0, 1\}$,
 without zero lines and rows. 
 \footnote{ The theorems below and their proofs remain valid as well for arbitrary skew-symmetric matrices 
 with integer elements without zero lines and rows, but in this case the notation used in the proof becomes  heavier.} 
Consider a family of quadratic forms on~$(\R^{d})^N$
$$
\om_j(z)=z_j\cdot{ (\al z)_j},  \qquad 1\leq j\leq N,
$$  
where  
$z$ is the polyvector 
$
(z_1,\ldots,z_N),
$
$z_j\in\R^d$, and
 $(\al z)_j: = \sum_{i=1}^N \al_{ji}z_i$. Let us  set
\be\lbl{Zabs}
\cZ=\{z\in  (\Z^{d}_L)^N :\, z_j \ne0 \qnd(\al z)_j\ne 0 \quad  \forall\,j\}.
\ee
Let a function $\Phi:
\mathbb R^{Nd}\to \mathbb \R$ be sufficiently smooth and
sufficiently fast decaying at 
infinity (see below for exact  assumptions).
Our goal is to study asymptotic as $L\to\infty$ behaviour of the sum 
 \be\lbl{S_L}
 S_{L,N}(\Phi) := L^{N(1-d)} \sum_{z\in\cZ: \  \omega_j (z) =0 \, \forall j} \Phi(z). 
 \ee

For a function $f\in C^k(\mathbb R^m)$,
 $n_1\in\N\cup\{0\}$ satisfying $n_1\le k$ and $n_2\in \mathbb R$,  we set
\be\label{the_norm}
\|f\|_{n_1,n_2}= \sup_{z\in \mathbb R^m} \max_{|\alpha|\le n_1} |\p^\alpha
f(z)| \lan z\ran^{n_2}\,, \qquad 
\langle x\rangle := \max\{1,|x|\}\,.
\end{equation}
 The first crucial result concerns the case $N=2$. 
 Then 
	$\om_1(z)=-\om_2(z)=\al_{12}z_1\cdot z_2$ and $\al_{12}\ne 0,$ so
	\be\lbl{S_L2}
	S_{L,2}(\Phi) = L^{2(1-d)}\sum_{\substack{z\in\Z^{2d}_L: \  z_1\cdot z_2 =0 \\ z_1\ne 0,\;z_2\ne 0}} \Phi(z).
	\ee
Then we write the sum above as $\sum_{z_1\cdot z_2 =0} - \sum_{z_1=0 \mbox{ \footnotesize or }z_2=0}.$ 
Since \\ $\Big| L^{-d}\sum_{z\in\Z_L^{2d}:\ z_i=0} \Phi(z)\Big|\leq C\|\Phi\|_{0,d+1}$ for $i=1,2$, we get
\be\lbl{S_L,2 approx}
\Big| S_{L,2}(\Phi) - L^{2(1-d)} \sum_{z\in\Z^{2d}_L: \  z_1\cdot z_2 =0} \Phi(z) \Big| \leq CL^{2-d}\|\Phi\|_{0,d+1}.
\ee
Now  an asymptotic for the sum $S_{L,2}(\Phi)$ immediately follows from Theorem~1.3 in \cite{number_theory}  where
the dimension is  $2d$,  $\eps=1/2$ and {$m=0$, by applying it to the sum $\sum_{z_1\cdot z_2 =0}$ in  \eqref{S_L,2 approx}
(we recall that $d\ge3$): 

\begin{theorem}\lbl{t:numbertheory}
Let  $N_1(d):=4d(4d^2+2d-1)$ and  $N_2(d):=N_1+6d+4$. If
$\|\Phi\|_{N_1,N_2}< \infty$, then  there exist constants $C_d\in (1,   1+2^{2-d})$,
 and  $K_d>0$ such that
  \be\label{H-B}
  \left|S_{L,2}(\Phi) -  C_d \int_{\Sigma_0}
{\Phi(z)} \, 
\mu^{\Sigma_0} (
dz_1dz_2) 
\right|\le K_d\frac{\|\Phi\|_{N_1,N_2}}{L^{d-5/2}} \,,
  \ee
where  $\Sigma_0$ is the quadric  $ \{z\in\R^{2d}: z_1 \cdot
z_2=0\}$ and the measure  $\mu^{\Sigma_0} $ is given by \eqref{meas} with $s=0$.
\end{theorem}
 In Appendix C of \cite{number_theory} we give the following explicit formula for the number-theoretical constants $C_d$:
\be\lbl{const_C_d} 
C_d =\frac{\zeta(d-1)\zeta(4d-2) }{\zeta(d)\zeta(2d-2)},
\ee
where $\zeta$ is the Riemann zeta-function. Due to \eqref{const_C_d}  $C_d$ satisfies $1<C_d<1+2^{2-d}$, as is  stated in the theorem. 
The integral in \eqref{H-B} converges if $\Phi(z)$ decays at infinity fast enough: 
\be\label{int_est}
\Big|\int_{\Sigma_0} {\Phi(z)} \, \mu^{\Sigma_0} (dz_1dz_2) \Big| \le C_r \| \Phi\|_{0,r} \quad\text{if} \;\; r>2d-1,
\ee
see Proposition 3.5 in \cite{DK1}.  So, it converges under the theorem's assumptions. 

From Theorem \ref{t:numbertheory}
another result can be deduced, whose proof is given in the next  Section~\ref{sec:intersection}:
\begin{theorem}\lbl{t:countingterms}
	For $N =  2, 3,4$ there exist constants $C_{d,N}$ such that
	\be\label{assertion}
	\big| S_{L,N}(\Phi) \big| \le C_{d,N}\|\Phi\|_{0,\bar N}\,,
	\ee
	for $\bar N:=\lfloor N/2 \rfloor N_2(d)+(N-2)(d-1)+ 1$, where  $N_2$ is 
	defined in Theorem~\ref{t:numbertheory}. 
\end{theorem}

 Since in view of estimate \eqref{cor_diag_est} 
  the functions $\Phi^\gF_s$ from Corollary~\ref{cor:Diag} satisfy 
	\be\lbl{est-Phi-F}
	\|\Phi^\gF_s(\tau_1,\tau_2,\cdot)\|_{n_1,n_2}\leq C^\#(s), \qquad \forall  n_1,n_2,
	\ee 
	 then 
 the two theorems above apply to study correlations \eqref{final_Ea^ma^n} with $N=m+n\leq 4$.  In fact, in the case $N=2$ the number of 
 Feynmann diagrams is small and the corresponding 
   correlations may be calculated directly without the machinery, developed in 
 Section~\ref{sec:diagram}.  In  Example~\ref{ex_corr} which   illustrats this computation, 
 as well as in a number of situations  below, we  apply Theorem~\ref{t:numbertheory} in the following setting:
 
 \begin{corollary}\lbl{cor:om-sum}
 	Let 
 	$$
 	\cS_{L,2} = L^{2(1-d)} \sum_{1,2,3} \dep \de(\oms) f_s(s_1,s_2,s_3;q),
 	$$
 	where $\oms$ is given by \eqref{omega}, $q\in\R^n$ is a parameter (in applications usually this will be  the time) and $f_s(s_1,s_2,s_1+s_2-s;q)$ 	\footnote{The formula for $s_3$ comes from the relation $\dep=1$. } 
 	is a Schwartz function of $(s_1,s_2,s)$ satisfying $|\p^\mu_{(s_1,s_2,s)} f_s|\leq C_\mu^\#(s_1)C_\mu^\#(s_2)C_\mu^\#(s)$ uniformly in $q$, for any multi-index $\mu$. 
 	Then 
 	\be\lbl{H-B_app}
 	  \left|\cS_{L,2} -  C_d \int_{\Sigma_s}
 	{f_s(s_1,s_2,s_1+s_2-s;q)} \, 
 	\mu^{\Sigma_s} (
 	ds_1ds_2) 
 	\right|\le \frac{C^\#(s)}{L^{d-5/2}} \,,
 	\ee
 	uniformly in $q$, where $\Sigma_s$ and $\mu^{\Sigma_s}$ are the quadric \eqref{i:quadr} and the measure \eqref{meas} on it. 
 \end{corollary}  
{\it Proof.} In the variables $z_1 = s_1 - s$, $z_2 = s_2 - s$ the quadratic form $\oms$ with $s_3 = s_1+s_2-s$ 
reads $\oms =- 2z_1\cdot z_2$ (see \eqref{omega}). 
Then, taking into account that the relation $\dep=1$ is equivalent to the relations $z_1, z_2 \ne 0$ and $s_3 =s_1+s_2-s$, we find that the sum $\cS_{L,2}$   takes the form \eqref{S_L2}. 
Applying next  Theorem~\ref{t:numbertheory} and changing  in \eqref{H-B} 
back to the variables  $s_1,s_2$   we get \eqref{H-B_app}.
\qed

    \begin{example}\label{ex_corr}
    	Let us calculate     the asymptotic as $L\to\infty$ of 
    $\EE |a_s^{(1)}(\tau)|^2$. Expanding     $a_s^{(1)}$ as in  \eqref{an2} and then using \eqref{corr_a_in_time}
    we get:
    \begin{equation*}
      \begin{split}
\EE|a_s^{(1)}(\tau)|^2 = 2L^{-2d}\sum_{1,2,3}&\dess
\delta(\omega^{12}_{3s}) \int_0^\tau dl_1 \int_0^\tau dl_2 B_{123} \\
&\times 
\prod_{j=1}^3 \left(e^{-\ga_j |l_1-l_2|} -e^{-\ga_j(l_1+l_2)}\right)
e^{\ga_s(l_1+l_2- 2\tau)}\ 
      \end{split}
    \end{equation*}
    with $B_{123}=B_1B_2B_3$, where $B_s$ is defined in \eqref{corr()}.
In the case of $\tau=\infty$ the formula simplifies since  by changing the integration variables 
 as $r_j:=\tau-l_j$ and passing to the limit we get
    \begin{equation*}
      \begin{split}
\EE|a_s^{(1)}(\infty)|^2 =& 2L^{-2d} \sum_{1,2,3}\dess
\delta(\omega^{12}_{3s}) \int_0^\infty dr_1 \int_0^\infty dr_2 \\
&\qquad B_{123}\, e^{-(\ga_1+\ga_2+\ga_3) |r_1-r_2|} e^{-\ga_s(r_1+r_2)}\\
=& \frac{2 L^{-2d}}{\ga_s} \sum_{1,2,3}\dess
\delta(\omega^{12}_{3s}) \frac{B_{123}}{\ga_1+\ga_2+\ga_3+\ga_s} \ .
      \end{split}
    \end{equation*}
   Then, by Corollary~\ref{cor:om-sum}, 
    $$
\left| L^2\EE|a_s^{(1)}(\infty)|^2 -
{ \frac{2C_d}{\ga_s}}\int_{\Sigma_s}\frac{B_{123}\;\mu^{\Sigma_s}(ds_1ds_2)}{\ga_1+\ga_2+\ga_3+\ga_s}\right|
\leq { \frac{C^\#(s)}{L^{d-5/2}}}\,, \quad s_3:= s_1+s_2-s.
$$
\end{example}

\subsection{Proof of Theorem~\ref{t:countingterms}}\label{sec:intersection}
Let us  define the
 geometric quadrics $Q_j:=\{z\in  (\R^{d})^N:\, \omega_j(z)=0\}$ and consider their intersection
  $Q=\cap_{j=1}^N Q_j$. 
  Note that $Q=\cap_{j=1}^{N-1} Q_j$ since the skew symmetry of the matrix $\al$ implies 
  $\om_1+\ldots+\om_N=0$.
 Denote by $B_R^{Nd}$ the open cube $|z|_\infty<R$ in $\R^{Nd}$, where by $|\cdot|_\infty$ we denote the $l_\infty$-norm.

\begin{proposition}\label{p:finite_support}
If $w:\R^{Nd}\mapsto \R$ is such that $ | w|_{L_\infty}< \infty $ and 
$\supp(w)\subset B_R^{Nd}$, where 
 $R\geq 1$,  then for   $N=2,3,4$ we have 
\be\lbl{est_prop3.3}
\left|\sum_{z\in Q\cap \cZ} w(z)\right| \le 
C(N,d) R^{\lfloor N/2\rfloor N_2(d)+(N-2)(d-1)}L^{N(d-1)}| w|_{L_\infty}\,.
\ee
Here $N_2$ is defined in Theorem~\ref{t:numbertheory} and $\cZ$ -- in \eqref{Zabs}.
\end{proposition}

{\it Proof.} 
 Below in this proof  for any subset $\mathcal Q \subset \R^{md}$ we denote 
 \be\lbl{QQ_L}
 \mathcal Q_L = \mathcal Q  \cap \Z_L^{md}. 
 \ee
By suitably rearranging  indices $i$  and possibly
multiplying $\om_i$ by $-1$, $\omega_1$ may be assumed to be of the form 
$
\omega_1(z) =z_1\cdot \sum_i \alpha_{1i} z_i
$ with $\alpha_{1N}=1$. Define $v= \sum_i\alpha_{1i} z_i$ so that 
\be\lbl{z->v}
\omega_1(z) =z_1\cdot v \qnd z_N= \alpha_{1N} z_N=
v-\sum_{1<i<N}\alpha_{1i}z_i,
\ee
since $\al_{11}=0$ by the skew symmetry of the matrix $\al$.

For  $N>2$, fix $(z_1,v)\in \mathbb R^{2d}$. Then the remaining quadratic forms
$\omega_j$ with $1<j<N$ as functions of $(z_2,\ldots,z_{N-1})\in \R^{(N-2)d}$ 
 become polynomials $q_j$  of degree at most two, with no
constant term. Namely
\begin{equation}
 \label{eq:q_tilde}
q_j (z_2,\ldots,z_{N-1};z_1,v)=z_j\cdot\Big(\alpha_{j1} z_1 +
\alpha_{jN} v+
\sum_{1<i<N}(\alpha_{ji}-\alpha_{jN}\alpha_{1i})z_i\Big)\,.
\end{equation}

For $1<j< N$  consider the sets 
$$
\tilde Q_j(z_1,v)=\{(z_2, \dots, z_{N-1}) :q_j(z_2,\ldots,z_{N-1};z_1,v)=0\} \subset \mathbb R^{(N-2)d}, 
$$
and their intersection $\tilde Q(z_1,v)=\cap_{1<j<N}\tilde Q_j(z_1,v)$.  We denote 
 $Q_1^0=\{(z_1,v)\in\mathbb R^{2d}:z_1\cdot v=0\}$ (cf. \eqref{z->v}) and set 
\be\label{A2}
A_2=\{(z_1,v)\in\mathbb R^{2d}: z_1\neq
0, v\neq 0\}.
\ee
Since $| \al_{ij} |\le1$, then on the support of $w$ we have $|(z_1,
v)|_\infty \le (N-1)R$. So,  recalling \eqref{QQ_L},  for $N>2$ we get
\begin{equation}\label{eq:sum_intersec}
  \begin{split}
\left|\sum_{z\in \cZ\cap Q } w(z)\right| &
\le C(N,d)| w|_{L_\infty}\sum_{(z_1,v)\in Q^0_{1\,L}
  \cap B_{ (N-1)R}^{2d}} 1  
\\
&\times\sup_{(z_1,v)\in Q^0_{1\,L}\cap A_2\cap  B_{ (N-1)R}^{2d}}\,\;
\sum_{(z_2,\ldots,z_{N-1})\in 
\tilde Q_L(z_1,v) \cap B_R^{(N-2)d}}1\, .
  \end{split}
\end{equation}
 For $N=2$ the same estimate holds with the
  second line replaced by 1.

 To estimate the sum in the first line, we take any smooth function $w_0(x) \ge0$, equal one for $x\le1$ and vanishing for
$x\ge2$. Then 
$$
\sum_{(z_1,v)\in Q^0_{1\,L}\cap B_{ (N-1)R}^{2d}} 
1 \le \sum_{(z_1,v)\in Q^0_{1\,L}} w_R(z_1,v),
$$
where $w_R(z_1, v) :=  w_0\Big( |(z_1,v)| / \big(   (N-1) R\sqrt{2d} \Big)$.  Since for $R\ge1$  and any $a\in\N\cup\{0\}$, $b\ge0$ 
we have 
$
\| w_R\|_{{ a, b}} \le C({ a,b,N,d})R^{b}, 
$
then in view of Theorem~\ref{t:numbertheory} and \eqref{int_est}, 
\be\lbl{est box}
\sum_{(z_1,v)\in Q^0_{1\,L}\cap B_{ (N-1)R}^{2d}} 1 \le 
C  L^{2(d-1)} \big[ R^{2d} + R^{N_2} L^{-d+5/2}\big] \le 
C'  L^{2(d-1)}  R^{N_2},
\ee
where $C, C'$ depend on $d, N, N_1$ and $ N_2$.

To estimate  the second line of \eqref{eq:sum_intersec}
we use the following lemma, proved in Appendix~\ref{app:intersection}.
\begin{lemma}\label{l:intersection}
	Assume that the matrix $\al$ is irreducible. Then for $N= 2,3,4$, any $R\geq 1$ and any  $(z_1,v)\in B^{2d}_{(N-1)R}$ satisfying $(z_1,v)\in  Q^0_{1\,L} \cap A_2$ we have:
	\be\label{estim}
	\big | \tilde Q_L(z_1,v)  \cap B_R^{(N-2)d}  \big|
	\le  2^{(N-2)d}(NRL)^{(N-2)(d-1)}\,.
	\ee
\end{lemma}
This completes the proof of Proposition~\ref{p:finite_support} in the case of irreducible matrix $\al$: indeed, we get
\be\lbl{sum_w_est}
\left|\sum_{z\in \cZ\cap Q } w(z)\right|
\le C(N,d)| w|_{L_\infty} R^{N_2 + (N-2)(d-1)} L^{N(d-1)}.
\ee
If the matrix $\al$ is reducible, it can be reduced through
permutations to a block diagonal matrix with $m$
blocks which are irreducible square matrices of sizes $N_i$ satisfying $\sum_i
N_i=N$. Since $N_i\ge 2$
(otherwise there would be a zero row or column  in $\al$), $ m \leq \lfloor N/2 \rfloor$.
Applying estimate \eqref{sum_w_est} to each block we get the assertion of the proposition.
\qed
\medskip

Now we derive the theorem from the proposition. Let
$\varphi_0(t) = \chi_{(-\infty,1]}(t)$ and for $k\ge 1$, $\varphi_k(t)
  = \chi_{(2^{k-1},2^k]}(t)$. Then $1=\sum_k\varphi_k(t)$ and
    $$
\Phi=\sum_{k=0}^\infty { f}_k(z)\,,\quad f_k(z)=\varphi_k(|z|_\infty)
\Phi(z)\,.
    $$
Then $\mathrm{supp}\,f_k\subset B_k=\{|z|_\infty\le 2^k\}$ and $\|f_k\|_{\infty}\le
C2^{-k\bar N}\|\Phi\|_{0,\bar N}$, {  for any $\bar N$}. Therefore,  by Proposition~\ref{p:finite_support}, 
$$
|S_{L,N}(\Phi)|\le C(N,d)\|\Phi\|_{0,\bar N} \sum_{k=0}^\infty
2^{k(\lfloor N/2 \rfloor N_2+(N-2)(d-1)-\bar N)}\,, 
$$
which converges if $\bar N>\lfloor N/2 \rfloor N_2+(N-2)(d-1)$. This
completes the proof of Theorem~\ref{t:countingterms}.
\qed

\begin{remark}\label{r_alg_set}
For any fixed vector $(z_1,v)$, 
 $\tilde Q(z_1,v)$ is a real algebraic set in $\R^{(N-2)d}$ of codimension $(N-2)$. If  $\tilde Q(z_1,v)$ 
were a smooth manifold of that codimension, then estimate \eqref{estim}, modified by a multiplicative constant 
$C_{\tilde Q(z_1,v)}$, would be obvious. But $\tilde Q(z_1,v)$ is a stratified analytic manifold (with singularities), and to obtain for it a modified 
version of the estimate \eqref{estim} as above, using analytical tools, seems to be a heavy job  since we need a good control for the 
factor $C_{\tilde Q(z_1,v)}$. Instead in Appendix~\ref{app:intersection} we prove the lemma, using arithmetical tools.
\end{remark}

\subsection{On extension of Theorem~\ref{t:countingterms} to any $N$}
\lbl{s:ext_countingterms}

The  restriction on $N$ in the statement
of Theorem~\ref{t:countingterms} comes from  estimate \eqref{estim} in Lemma~\ref{l:intersection}, proved 
only for $N=3,4$. We know that for every $N$ 
 the system of polynomials $q_j(\cdot; z_1, v), 1<j<N$, defining the set { $\tilde Q_L(z_1,v)$} in 
Lemma~\ref{l:intersection}, is linearly independent for any $(z_1,v)$ and any irreducible incidence matrix $\alpha$. Also we know that all 
polynomials  $q_j(\cdot; z_1, v)$ are irreducible; see Lemmas~\ref{l:ind} and \ref{l:irr} in Appendix~\ref{app:intersection}  (there the independence and 
reducibility are understood over some specific algebraically closed field $K$, but the argument also works for $K$ replaced by $\C$). These two facts certainly 
are insufficient to prove Lemma~\ref{l:intersection} for any $N$, but they naturally lead to

\begin{conjecture}\label{con_3.6}
Under assumptions of Lemma~\ref{l:intersection},  for any $N\ge2$ 
$$
 \big | \tilde Q_L(z_1,v)  \cap B_R^{(N-2)d}  \big| \le C(N,d)(RL)^{(N-2)(d-1)}.
$$
\end{conjecture}

One may try to prove this assertion using either arithmetical or analytical tools; cf. 
 Appendix~\ref{app:intersection} and  Remark~\ref{r_alg_set}. 
It is straightforward to see that, if the conjecture is true, then
Theorem~\ref{t:countingterms} holds for any $N$, so  in view of
Lemma~\ref{l:E-fin} any expected  value $L^N\EE a_s^{(m)}(\tau_1)\bar a_s^{(n)}(\tau_2)$ admits a uniform in $L$ upper
bound. 

\section{Quasisolutions}
\lbl{sec:quasisol}

In this section we start to study  a quasisolution $A(\tau)=A(\tau;L)$ of eq.~\eqref{ku4} with  $a_s(0) = 0$, 
which is the  second order truncations of series \eqref{decomp}:
\be\lbl{2nd trunc}
A(\tau) = (A_s(\tau), s\in\Z^d_L), \qquad A_s(\tau) = a^{(0)}_s(\tau) +\rho a^{(1)}_s(\tau) +\rho^2 a^{(2)}_s(\tau)\,.
\ee
We focus on its energy spectrum 
\be\lbl{energ_sp}
\nL{s} = \nL{s}(\tau)
= \EE |A_s(\tau)|^2,\quad s\in\Z_L^d, 
\ee
when $L$ is large and the parameter $\rho$ is chosen to be  $\rho=\eps L$. Our goal is to  show  that it approximately satisfies a
wave kinetic equation  (WKE).   Using Proposition~\ref{p:approx}, we will then show that the  same applies
to the quantities $\mathfrak n_{{ s,L}}$, considered in the Introduction.

The energy spectrum  $\nL{s}$ is a polynomial in $\eps$ of degree four, 
\be\lbl{n_s-decomp}
\nL{s}=\nL{s}^{(0)} + \eps \, \nL{s}^{(1)}+\eps^2 \nL{s}^{(2)} 
+ \eps^3 \nL{s}^{(3)} + \eps^4 \nL{s}^{(4)}, \qu s\in\Z_L^d,
\ee
where the terms $\nL{s}^{(k)}(\tau)$ are defined by
\be\label{n_s^(k)}
\nL{s}^{(k)}(\tau)= L^k\sum_{\substack{k_1+k_2=k \\ 0\leq k_1,k_2\leq 2}} \EE a^{(k_1)}_s(\tau) \bar a^{(k_2)}_s(\tau)\,. 
\ee
By Corollary~\ref{cor:Diag}, 
\begin{equation}\label{schw_ext}
\begin{split}
&\text{the second moments  $\EE a_s^{(k_1)} \bar a_s^{(k_2)}$
}\\
&\qquad\text{naturally extend to a Schwartz function of $s\in\R^d$, }
\end{split}
\end{equation} 
given by \eqref{final_Ea^ma^n}, \eqref{corr-sums}. Accordingly 
from now on we always regard  the second moments and the terms $\nL{s}^{(k)}(\tau)$ as Schwartz functions of $s\in\R^d$.

As customary in WT, we aim at considering the limit of $\nL{s}(\tau)$
as $L\to \infty$, that is, the limits of the terms $\nL{s}^{(j)}$.
The term $\nL{s}^{(0)} = \mathfrak{n}_{s,L}^{(0)}$  is given by
 \eqref{n_s_0} and is $L$-independent, while by a direct computation we see that 
 \be\lbl{n_s^1}
\nL{s}^{(1)} = 2\Re\EE \bar a_s^{(0)} a_s^{(1)}=0, \quad s\in\R^d,
\ee
 (here we use \eqref{an2}, the Wick theorem and \eqref{notethat}).
 Writing explicitly $\nL{s}^{(i)}$ with $2\leq i\leq 4$, we find that 
 \be\lbl{n^j_s}
 \begin{split}
\nL{s}^{(2)}=L^2\EE \big(|a_s^{(1)}|^2 + 2\Re \bar a_s^{(0)}
a_s^{(2)}\big)\,,\\
\nL{s}^{(3)}=2L^3\Re\EE\bo{s}\atw{s}\,, \qquad
\nL{s}^{(4)}=L^4\EE|\atw{s}|^2\,.
\end{split}
\ee
The function $\R^d\ni s\mapsto  \nL{s}^{(2)}(\tau)$ is made by two terms. By 
Corollary~\ref{cor:Diag} with $N=2$,  Theorem~\ref{t:numbertheory} applies to the both of them. Since $d\ge3$, we get
\be\label{decompp}
|\nL{s}^{(2)}(\tau) - n^{(2)}_s (\tau) | \le
C^\#(s)/L^{   1/2}\,,
\ee
where 
\begin{equation*}
\begin{split}
n^{(2)}_s (\tau): = C_d \Biggl(
\sum_{\gF\in\gF_{1,1}^{true}}   + 2 \Re\sum_{\gF\in\gF_{2,0}^{true}}\Biggr)
c_{\gF}
\int_{\Sigma_0}\mu^{\Sigma_0}(dz_1dz_2) 
\Phi_s^{\gF}(\tau,\tau, z) ,
\end{split}
\end{equation*}
and we have used estimate \eqref{est-Phi-F}. 
Thus, we see that the processes $\nL{s}^{(0)}$, $\nL{s}^{(1)}$ and $\nL{s}^{(2)}$ admit the limits
$$n_s^{(j)} (\tau):= \lim_{L\to\infty}\nL{s}^{(j)}(\tau;L).$$
The limits  satisfy  \eqref{decompp}, and  for all $\tau$
\be\label{bounds}
n_s^{(0)} (\tau) = B(s) \big( 1-e^{-2\ga_s(\tauz+\tau)}\big), \quad
n_s^{(1)}(\tau) = 0 ,\quad |n_s^{(2)}(\tau) |\,\le C^\#(s)\, ,
\ee
 where the last inequality  follows from Theorem~\ref{t:countingterms}.

 We do not know   if the 
 terms  $\nL{s}^{(3)}, \nL{s}^{(4)}$ admit limits as $L\to \infty$,  but  in view of Corollary~\ref{cor:Diag} both of them may 
 be estimated through Theorem~\ref{t:countingterms}:
 \be\label{n_s_34}
 |\nL{s}^{(3)}(\tau)|\le C^\#(s), \quad
 |\nL{s}^{(4)}(\tau)| \le   C^\#(s) \,, 
 \ee
 uniformly in $L\ge2$ and $\tau\ge0$. We then decompose
 $$
 \nL{s}=\nL{s}^{\leq 2}+\nL{s}^{\geq 3},
 $$
 where 
 $$
 \nL{s}^{\leq 2}=\nL{s}^{(0)}+ \eps\nL{s}^{(1)} +\eps^2 \nL{s}^{(2)} 
 \qnd
 \nL{s}^{\geq 3}=\eps^3 \nL{s}^{(3)} + \eps^4 \nL{s}^{(4)}
 $$
(we recall that $\nL{s}^{(1)} \equiv 0$),  and similarly define 
 $$n_{s}^{\leq 2}:=n_{s}^{(0)} +\eps^2 n_{s}^{(2)} .$$
 Due to \eqref{decompp},  
  \be\lbl{diff-limit<2}
 |n^{\le 2}_s(\tau) - \nL{s}^{\le 2}(\tau)| \le C^\#(s)\eps^2L^{-  1/2},
 \ee
 so by \eqref{n_s_34},
 \be\lbl{diff-limit}
 |n^{\le 2}_s(\tau) - \nL{s}(\tau)| \le C^\#(s)\eps^2(L^{-  1/2} + \eps).
 \ee
 Thus, the cut
  energy spectrum $n^{\le 2}_s$ governs the limiting as $L\to\infty$ behaviour of the energy spectrum $\nL{s}$ with precision $\eps^3 C^\#(s) $, 
  where we   regard the constant 
 $\eps\le1/2$ (which measures the size of
 solutions for \eqref{ku3} under the proper scaling) as a fixed small
 parameter. Accordingly, our next goal is to show that $n^{\le 2}_s(\tau)$ approximates the solution of a WKE.

\subsection{Increments of the energy spectra $n_s^{\leq 2}$.} \label{s_8.1}

In this section we will show that the process $ n_s^{\leq 2} (\tau)$ approximately satisfies a  WKE.
 We denote $s_4:=s$,  $\ga_j: =\ga_{s_j}$ and set 
\be\label{grr}
\ga_{1234} = \ga_{1}+\ga_{2}+\ga_{3}+\ga_{4}, \quad \vs = (s_1, s_2, s_3, s_4) \in   (\R^d)^4.
\ee

Now, for a fixed  $\tauz\ge 0$ and for $j=1,2,3,4$ we define the functions
$\cZ^j(\tau_0)=\cZ^j(\tauz;\vs)$  as  
\be\lbl{Z-edi}
  \cZ^j(\tauz;\vs) := \int_0^{\tauz}dl\,\e^{-\ga_j(\tauz-l)} \prod_{\substack{m=1,2,3,4
  \\ m \neq j}}
  \frac{\sinh(\ga_m
    l)}{\sinh(\ga_m \tauz)}\,\quad \text {if} \  \tau_0>0,
  \ee
   and $\cZ^j(0;\vs) = 0$. 
  Computing this integral we get  
  \be
  \begin{split}\lbl{Z-ed}
\cZ^j (\tauz;\vs)= &\left(\prod_{l\neq j}\frac{1}{1-e^{-2\ga_l\tauz}}\right)\cdot
\Biggl[\frac{1-e^{-\ga_{1234}\tauz}}{\ga_{1234}} - 
\frac{e^{-2(\ga_{1234}-\ga_j)\tauz}-e^{-\ga_{1234}\tauz}}{2\ga_j-\ga_{1234}}\\
&
 + \sum_{l\neq j}\left(  \frac{
    e^{-2(\ga_{1234}-\ga_j-\ga_l)\tauz}-e^{-\ga_{1234}\tauz}}{2(\ga_j+\ga_l) -
    \ga_{1234}} -\frac{
    e^{-2\ga_l\tauz}-e^{-\ga_{1234}\tauz}}{\ga_{1234}-2\ga_l} \right)
  \Biggr]\,,
\end{split}
\ee 
where each fraction  from the square brackets should be substituted by
$\tauz e^{-\ga_{1234}\tauz}$ if its  denominator vanishes. 

 For any
real number $r$ let $\Cc_r(\mathbb R^d)$ be the space of continuous
complex functions on $\mathbb R^d$ with the finite norm
\be\lbl{rnorm}
\left|f\right|_r = \left|f(z) \langle z\rangle^r\right|_{L_\infty}\,.
\ee
We naturally 
 extend this norm to $f\in L_\infty(\mathbb R^d)$ and set
\be\lbl{L_inf_r}
L_{\infty,r}(\mathbb R^d) = \left\{f\in L_\infty(\mathbb R^d) :
\left|f\right|_r < \infty\right\}.
\ee
Consider also the linear operator $\cL$,  given by 
\be\label{operL}
(\Lc v)(s) = 2\ga_s v(s), \quad s \in \R^d. 
\ee
 Below we often write the value $v(s)$ of a function $v$ at $s\in\R^d$ as $v_s$ and the function $v$ itself as $(v_s,\, s\in\R^d)$.
Now,  for  $v\in \Cc_r(\R^d)$, where  $r>d$, and for $\tauz\ge0$, 
 $\tau\in(0,1]$, 
we define  the kinetic integral 
 $
 K^\tau(\tauz)(v) =(K^\tau_s (\tauz)(v), \,s\in \R^d)
 $:
 \be\label{I_tau}
K^\tau(\tauz)(v) = \int_0^{\tau} e^{-t\Lc} K(\tauz)(v)dt\,.
\ee
Here the operator $K(\tauz)=K^1(\tauz)+\dots+K^4(\tauz)$ sends a function $v=(v_s, s\in\R^d)$ to the function 
\be\label{k11}
\begin{split}
	K_s &(\tauz)(v)
	=4C_d\int_{\Sigma_s} \mu^{\Sigma_s}( ds_1ds_2) 
        \Bigl(\cZ^4(\tauz;\vs) v_1v_2v_3 \\
         &+
	\cZ^3(\tauz;\vs) v_1v_2 v_4- \cZ^2(\tauz;\vs) v_1 v_3v_4
        -  \cZ^1(\tauz;\vs)
         v_2v_3v_4\Bigr)\\
        &=: \,  K^4_s (\tauz) (v) +  K^3_s (\tauz)(v) +  K^2_s (\tauz)(v) +
         K^1_s (\tauz)(v)
\end{split}
\ee
(note the reversed signs for $K^2$ and $K^3$). 
Here $v_j:= v(s_j)$, where $s_4=s$ and $s_3:= s_1+s_2-s_4$  (in view of the factor $\dep $).  While 
$\mu^{\Sigma_s}$ is the measure \eqref{meas}
 on the quadric 
$
\Sigma_s =\{  (s_1,s_2)\in \R^{2d}:\, (s_1-s)\cdot (s_2-s) =0\}$.
 Computing the integral in $t$ in \eqref{I_tau}, we find
\be\lbl{kin int^tau}
K^\tau_s (\tauz) (v)=  \frac{1-e^{-2\ga_s\tau}}{2 \ga_s} K_s(\tauz)(v)=
	 \frac{1-e^{-2\ga_s\tau}}{2\ga_s}\sum_{j=1}^4  K^j_s (\tauz)(v). 
\ee
We study the kinetic integral $K^\tau$ in Section~\ref{sec:kin_eq} while now we formulate a result  
which is the main step in deriving  the wave kinetic limit. 

\begin{theorem}\lbl{th:kinetic_for_n^2}
	For any $0<\tau\le1$  the function $(n^{\leq 2}_s, \, s\in \R^d)$ satisfies 
	\be\label{n_increment}
	\begin{split}
		n^{\leq 2}(\tauz+\tau)& = e^{-\tau \Lc} n^{\leq 2}(\tauz) +2 \int_0^\tau e^{-t\Lc} {b}^2\, dt 
		+\eps^2 K^\tau  (\tauz ) (n^{\le2}(\tauz))
		 +\eps^2\Rc,  \\
	\end{split}
	\ee
where $b^2 =( b^2(s) , s\in \R^d)$
 and   the remainder
$\Rc_s(\tau)$ satisfies  
	\be\lbl{R_s}
	|\Rc(\tau)|_r \le 
	 C_{r} \tau\, (\tau + \eps^2)\,,\quad \forall r.
	 \ee
\end{theorem}

\subsection{Proof of Theorem \ref{th:kinetic_for_n^2}.}
\lbl{s:th_kinetic}
We first  fix a value for $L$ and decompose the processes
$\tau\mapsto a^{(i)}_s(\tauz+\tau)$, where $\tauz\ge0$ and $0\le
\tau\le  1$,
 as
\be\lbl{a-a_del}
a_s^{(i)}(\tauz+\tau) = c_s^{(i)}(\tau;\tau_0) + \Del a_s^{(i)}(\tau;\tau_0), \qquad i=0,1,2,\qu s\in \Z^d_L\,.
\ee
Here 
$$
c_s^{(i)}(\tau;\tau_0)= e^{-\ga_s\tau}a_s^{(i)}(\tauz)
$$
and $\Del a_s^{(i)}$ is defined via relation 
\eqref{a-a_del}. Below we write $ c_s^{(i)}(\tau;\tau_0)$ and $  \Del a_s^{(i)}(\tau;\tau_0)$ as 
$ c_s^{(i)}(\tau)$ and $  \Del a_s^{(i)}(\tau)$  since $\tau_0$ is fixed.

Obviously, 
$$
c(\tau) := c^{(0)}(\tau) +\rho  c^{(1)}(\tau) +\rho^2  c^{(2)}(\tau)
$$
with 
$\tau\ge0$ is a solution of the linear equation \eqref{ku4}${}_{\rho=0, b(s)\equiv 0}$, equal $A(\tauz)$ at $\tau=0$, and 
$
\Delta a(\tau) = \sum_{j=0}^2 \rho^j\Delta a^{(j)}(\tau)
$ 
equals $ A(\tauz+\tau) - c(\tau)$. By \eqref{schw_ext},  for $0\le i,j\le2$ 
 \begin{equation}\label{smooth_ext}
\begin{split}
\text{the functions }
&\EE c_s^{(i)} \bar c_s^{(j)} ,\quad \EE c_s^{(i)} \Delta \bar a_s^{(j)} ,\quad   \EE   \Delta  a_s^{(i)} \Delta \bar a_s^{(j)} 
\\
&\text{ naturally extend to  Schwartz functions of $s\in\R^d$. }
\end{split}
\end{equation} 
Due to \eqref{n_s^1} and \eqref{n^j_s}, 
$$e^{-2\ga_s\tau}\nL{s}^{\leq 2}(\tauz)
=\EE|c_s^{(0)}(\tau)|^2 + \rho^2\EE \big(|c_s^{(1)}(\tau)|^2 + 2\Re \bar c_s^{(0)}(\tau) c_s^{(2)}(\tau)\big),\quad \forall\, s\in \R^d.
$$ 
Then,
 \begin{equation}\label{Delta_A^2}
\begin{split}
\nL{s}^{\leq 2}&(\tauz+\tau)-e^{-2\ga_s\tau}\nL{s}^{\leq 2}(\tauz)
=\EE\Big( |a_s^{(0)}(\tauz+\tau)|^2  - |c_s^{(0)}(\tau)|^2 \\
	&+\rho^2  \big(|a_s^{(1)}(\tauz+\tau)|^2  - |c_s^{(1)}(\tau)|^2
	 +2\Re\big(a^{(2)}_s\bar a^{(0)}_s(\tauz+\tau)- c^{(2)}_s\bar c^{(0)}_s(\tau)\big)\Big).
\end{split}
\end{equation}
 Let us set 
 $$
\cY_s(u,v,w):=L^{-d} \sum_{1,2,3} \dess \delta(\omega^{12}_{3s}) 
u_1 v_2 \bar w_3\,.
$$  
Writing explicitly the  processes  
$
\Del a_s^{(i)}(\tau)
$,
$s\in\Z^d_L$, 
we find
\begin{equation}\lbl{deltas_a}
\begin{split}
	\Del a_s^{(0)}(\tau)&=
	b(s)\int_{\tauz}^{\tauz+\tau} e^{-\ga_s(\tauz+\tau-l)}\,d\beta_s(l), 
	\\
	\Del a_s^{(1)}(\tau)&=
	i\int_{\tauz}^{\tauz+\tau} e^{-\ga_s(\tauz+\tau-l)} \cY_s(a^{(0)})\,dl,
	\\ 
	\Del a_s^{(2)}(\tau)&=
	i\int_{\tauz}^{\tauz+\tau} e^{-\ga_s(\tauz+\tau-l)}
        \Biggl(\cY_s(a^{(0)},a^{(0)},a^{(1)})\\
        &\qquad\qquad+
        \cY_s(a^{(0)},a^{(1)},a^{(0)}) + \cY_s(a^{(1)},a^{(0)},a^{(0)})\Biggr)dl,
	\end{split}
\end{equation}
where $a^{(i)} = a^{(i)}(l)$.
Note that  to get explicit formulas for $c_s^{(i)}(\tau)$, $i=0,1,2$, it suffices to replace in the r.h.s.'s  of the relations in 
\eqref{deltas_a} the range of  integrating from $[\tauz,\tauz+\tau]$  to
$[0,\tauz]$.  For example,
$
c_s^{(0)}(\tau) = e^{-\ga_s\tau}  a_s^{(0)}(\tauz) = b(s)\int_{0}^{\tauz} e^{-\ga_s(\tauz+\tau-l)}\,d\beta_s(l). 
$

Using that 
$\EE c_s^{(i)}(\tau) \Del \bar a_s^{(0)}(\tau)=\EE c_s^{(i)}(\tau)  \,\EE\Del \bar a_s^{(0)}(\tau)=0$ for any $i$ and $s$, 
we obtain 
\be\lbl{aa0}
\EE \big(a_s^{(2)} \bar a_s^{(0)}(\tauz+\tau)  - c_s^{(2)} \bar c_s^{(0)}(\tau) \big)
=\EE \Del a_s^{(2)}(\tau) \bar a_s^{(0)}(\tauz+\tau) ,
\ee
and from \eqref{a-a_del} we get that 
\begin{equation}\lbl{a^2-a^2}
\begin{split}
&|a_s^{(1)}(\tauz+\tau)|^2  -  |c_s^{(1)}(\tau)|^2=|\Del a_s^{(1)}(\tau)|^2  + 2\Re \Del a_s^{(1)} \bar c_s^{(1)}(\tau),
	\\
&\EE\big( |a_s^{(0)}(\tauz+\tau)|^2  -   |c_s^{(0)}(\tau)|^2 \big)=\EE|\Del a_s^{(0)}(\tau)|^2.
	\end{split}
\end{equation}
Then, inserting \eqref{aa0} and \eqref{a^2-a^2} into \eqref{Delta_A^2}
and using that $\rho=\eps L$,
we find
\be\non
\nL{s}^{\leq 2}(\tauz+\tau)-e^{-2\ga_s\tau}\nL{s}^{\leq 2}(\tauz)
= \EE\big|\Del a_s^{(0)}(\tau)\big|^2 + \eps^2 Q_{s,L}(\tauz,\tau),\quad s\in\R^d, 
\ee
where 
\be\lbl{Q_s-def}
\begin{split}
Q_{s,L}(\tauz,\tau):=&L^2\Big(\EE |\Del a_s^{(1)}(\tau)|^2 +
2\Re\EE\big[\Del a^{(1)}_s(\tau) \bar c^{(1)}_s(\tau)\\
&+ \Del a^{(2)}_s(\tau) \bar a^{(0)}_s (\tauz+\tau)\big]\Big)\,,
\end{split}
\ee
and we recall \eqref{smooth_ext}. 
Since 
$$\displaystyle{\EE\big|\Del a_s^{(0)}(\tau)\big|^2=\frac{b(s)^2}{\ga_s}(1-e^{-2\ga_s\tau})} =
2 \int_0^\tau e^{-t\Lc} {b}^2(s)\, dt,
$$
 then 
\be\non
\nL{\cdot}^{\leq 2}(\tauz+\tau)-e^{-\tau\Lc}\nL{\cdot}^{\leq 2}(\tauz)= 2 \int_0^\tau e^{-t\Lc} {b}^2\, dt + \eps^2 Q_{ \cdot,L}(\tauz,\tau),
\ee
for $\nL{\cdot}^{\leq 2} = (\nL{s}^{\leq 2}$, $s\in \R^d)$.
In order to pass to the limit $L\to \infty$ we  recall the relation \eqref{diff-limit<2}.
Then the desired formula \eqref{n_increment} is an 
 immediate consequence  of  the assertion below:
\begin{proposition}\lbl{l:N^2_s}
  We have 
  \be\lbl{Q_s=}
 \lim_{L\to \infty} Q_{s,L}(\tauz,\tau) = K^\tau_s (\tauz) (n^{\le2}(\tauz))   +  \Rc_s(\tau), \quad s\in \R^d, 
	\ee
	where the remainder $\Rc$ satisfies \eqref{R_s}.
	\end{proposition}
\begin{proof} 
The first step in the proof of \eqref{Q_s=}  is the following result, established  in Appendix~\ref{sec:Q_s-prop}:
\begin{proposition}\lbl{l:N2}
	One has
	\be\lbl{l-N2}
	\big| Q_{s,L} (\tauz,\tau) - \cX_{ s,L} (\tauz,\tau) \big| \leq
        C^{\#}(s)\tau^2, \quad s\in \R^d\,, 
	\ee
	where  
	\be\lbl{ZZ_s}
\begin{split}
\cX_{s,L}(\tauz,\tau)
:=4L^{  2(1-d)}\tau\sum_{1,2,3}\dep \delta(\oms)
		\big(\cZ^4n^{(0)}_1n^{(0)}_2n^{(0)}_3
		+\cZ^3n^{(0)}_1n^{(0)}_2n^{(0)}_s\\
		-\cZ^1n^{(0)}_2n^{(0)}_3n^{(0)}_s- \cZ^2n^{(0)}_1n^{(0)}_3n^{(0)}_s\big)\,.
\end{split}
\ee
The terms $\cZ^j=\cZ^j(\tauz;s_1,s_2,s_3,s)$ are defined by
\eqref{Z-edi} and $n_{i}^{(0)}:=n_{s_i,L}^{(0)}(\tauz)$, 
$n_{s}^{(0)}:=n_{s,L}^{(0)}(\tauz)$.

\end{proposition}
By \eqref{bounds} $n_{i}^{(0)}= n_{s_i}^{(0)}$ are Schwartz functions in   $s_i$. Besides, 
 the functions $\cZ^j(\tauz,\vs)$ have at
  most polynomial growth in $\vs$ together with their derivatives,
   uniformly in $\tau_0\geq 0$: 
   
  \begin{lemma}\lbl{l:Z^j-growth}
    For any vector $\mu\in (\N\cup\{0\})^{4d}$, uniformly in $\tauz\ge 0$,    we have
    $$
\left| \p^{\mu}_{\vs} \cZ^j(\tauz,\vs)\right| \le P(\vs;\mu)\,,
$$
where $P(\vs;\mu)$  has at most polynomial growth in $\vs$.
  \end{lemma}
  By the lemma, which is proven in Section~\ref{s:Z^j-growth},
  $\cX_{s,L}$ satisfies the hypotheses of Corollary~\ref{cor:om-sum}. So 
\be\lbl{X-K0}
|\cX_{s ,L}(\tauz,\tau)
-  \tau K_s(\tauz) (n^{(0)}) |\leq C^\#(s)L^{-{  1/2}}\tau\,.
\ee
Next, note that
$ |n_s^{(0)}(\tauz)
-n_s^{\leq 2} (\tauz) |\leq
C^\#(s)\eps^2\, 
$
due to \eqref{bounds}. Then the estimate on the Lipschitz constants of
the operators $K^j(t)$, given in \eqref{k3}, implies that 
$$
\big|K(\tauz) (n^{(0)} (\tauz)) - K(\tauz) (n^{\le2}(\tauz) ) \big|_r \le  C_r \eps^2 \quad \forall\, r.
$$
So that 
\be\lbl{K0-K}
| \tau K_s(\tauz)(n^{(0)}(\tauz)) - \tau K_s(\tauz)( n^{\leq 2}(\tauz) )  |\leq C^\#(s)\tau \eps^2.
\ee
On the other hand, on account of the definition \eqref{kin int^tau},
for $0\le\tau\le 1$ we have the bound
\be\lbl{K-Kt}
\left| \tau K_s(\tauz) (n^{\leq 2}) -K^\tau_s(\tauz)(n^{\leq 2})\right|\leq
C\ga_s \tau^2 |K_s(\tauz)(n^{\leq 2})|\le C^\#(s)\tau^2\,,
\ee
where the last inequality follows from the estimate of the norm of the operator
$K^j(t)$, given in \eqref{k2}, and from \eqref{bounds}.

Putting together \eqref{l-N2}, \eqref{X-K0}, \eqref{K0-K}, \eqref{K-Kt} and letting
$L$ grow to infinity, we conclude the proof.
\end{proof}

\section{Kinetic equation}
\lbl{sec:kin_eq}
At this section we examine the  wave  kinetic equation
\be\label{k1}
\dot \zz_s (\tau) = -(\cL \zz)_s +\eps^2 K_s(\tau)(\zz) +2 b(s)^2, \qquad \tau\ge0, \;\; \zz(0)=0
\ee
($\cL$ is defined in \eqref{operL} and 
the operator $K=K^1+\dots +K^4$ is defined in \eqref{k11}), 
and next derive from this analysis and \eqref{n_increment} the proximity of $n_s^{\leq 2}(\tau)$ to a solution of \eqref{k1}. 
We will need the following result, which is Lemma~4.2 from \cite{DK1}: 
\begin{lemma} \lbl{l_4.2}
For $j, l=1,\dots, 4$ and $u^j \in \cC_r(\R^d)$ consider the  operators 
$$
J_l(u^1, \dots, u^4) (s) = 
 \int_{\Sigma_s} \mu^{\Sigma_s} \!(ds_1 ds_2) \prod_{i\ne l} u^i(s_i)
$$
(see \eqref{meas}), 
where $s_4=s$ and $s_3= s_1+s_2-s$. Then for each $l$, 
\be\label{4.2}
| J_l(u^1, \dots, u^4) |_{r+1} \le C_r \prod_{i\ne l} |u^i|_r 
 \qquad \text{if}\;\; r>d.
\ee

\end{lemma}

\subsection{Kinetic integrals.}\lbl{ss_kin_int}
We recall  notation \eqref{grr},  \eqref{Z-edi}. 

\begin{lemma} \lbl{l_k1}
For $j=1,\dots 4$, any $\tau\ge0$ and any $\vs=(s_1,\dots,s_4) \in(\R^d)^4$,

i)   $0\le \cZ^j(\tau;\vs)  \le \min(\tau, 1/\ga_{s_j})\le1$, 

ii) $| \cZ^j(\tau;\vs) -\cZ(\infty{;\vs})| \le  Ce^{-2\tau}$, where
$
\cZ(\infty;\vs) = 1/ \ga_{1234}. 
$	
\end{lemma}

 \begin{proof}
 The first assertion follows from \eqref{Z-edi} since $\sinh(x)$ is an increasing non-negative 
  function of $x\ge0$, so in the integrand in \eqref{Z-edi} we have 
  $\ 0\le  \sinh(\ga_m l) / \sinh(\ga_m \tau' )\le 1$. For   $0\le \tau\le1$ 
  the second estimate  follows from   the first one   as
  $$
  |\cZ^j(\tau;\vs) - \cZ(\infty;\vs) | \le |\cZ^j(\tau;\vs) |+ | \cZ(\infty;\vs) | ,
  $$
  while for $\tau\ge1$ it  follows from \eqref{Z-ed} since
   $\ga_{123j} - \ga_j \ge1$ and  $\ga_{1234} - \ga_j -\ga_l \ge1$ for $j,l\in \{1,2,3,4\}, j\ne l$.
  \end{proof}
 
 Since the kernels $\cZ^j$ are non-negative by the first assertion of the lemma above, then denoting 
 $
 \kappa_1=\kappa_2=1,$ $\kappa_3=\kappa_4 =-1
 $
 we achieve that the operators $\kappa_j K^j$, $1 \le j\le 4$, are positive (in the sense that they send positive functions to positive). Due to  the first assertion of the lemma and \eqref{4.2}, 
 for any $\tau\ge0$ they  define positive  3-homogeneous mappings 
 $
 \cC_r(\R^d) \to \cC_{r+1} (\R^d)
 $
 if $r>d$, and 
 \be\label{k2}
  |\kappa_j K^j (\tau)(v)|_{r+1} =
| K^j (\tau)(v)|_{r+1} \le C_r  \min(\tau, 1) |v|_r^3, \quad j=1,\dots, 4, 
\ee
 for $\tau\ge0$. So the mappings $ K^j (\tau)$   are locally Lipschitz:
 \be\label{k3}
| K^j (\tau)(v^1) - K^j (\tau)(v^2) 
|_{r+1} \le 3C_r  \min(\tau, 1) R^2 |v^1-v^2|_r,  \quad \text{if} \;\; |v^1|_r, |v^2|_r \le R. 
\ee
 Since for $j=1,\dots,4$ and   any $s\in  \R^d$, 
 
 $\bullet$ for non-negative functions $n,m\in L_{\infty, r}$  (see \eqref{L_inf_r})  such that $m\le n$ we have 
  $\kappa_j K^{j}_s (\tau)(m) \le \kappa_j  K^{j}_s (\tau)(n)\le\infty$,
 
   $\bullet$     $|K^{j}_s (\tau)(v) | \le  \kappa_j
   K^{j}_s (\tau)(|v|) \le\infty $ for any complex function $v\in L_\infty$,
  
   $\bullet$ $|v_s| \le |v|_r \lan s\ran^{-r}$ for all $v\in L_{\infty, r}$, 
  
  \noindent 
  then the relations \eqref{k2},  \eqref{k3} remain true  for functions from $L_{\infty, r}$.

\begin{lemma} \lbl{l_k2}
If $|s_l | \le R$ for $l=1,\dots,4$, then 
\be\lbl{k5} 
\left| \frac{\p}{\p\tau} \cZ^j(\tau;\vs)\right| \le  C \ga^0(R^2)
\ee
(see \eqref{gamma}). 
\end{lemma}
 
 \begin{proof}
 For any $m\in\{ s_1,\dots, s_4\}$ and $0\le l\le \tau$ 
 we have 
 $$
\left| \frac{\p}{\p\tau} \, \frac{\sinh \ga_m l}{ \sinh \ga_m \tau}  \right| \le \ga_m \, \frac{\cosh \ga_m \tau}{ \sinh \ga_m \tau} \le
\ga_m C\max(1, 1/ (\ga_m \tau)). 
 $$
 Considering separately the cases $\tau\ge1$ and $0\le \tau<1$, using \eqref{Z-edi} and 
  the estimate above we get the result.
 \end{proof} 
 
 This lemma implies that for any $v\in \cC_r(\R^d)$ and any $j$ the curve $\tau\mapsto K^j(\tau)(v)\in \cC_r(\R^d)$ is H\"older-continuous:

\begin{lemma} \lbl{l_k3}
For any $\tauz \ge0$, $0\le \tau\le1$, $j=1,\dots, 4$ and  
any $r>d+1$, 
\be\lbl{k8} 
| K^j(\tauz+\tau)( v)  - K^j(\tauz)(v) |_r \le C_r |v|_r^3 \tau^{\kappa_*} \qquad\forall\, v\in \cC_r(\R^d), 
\ee
where 
 $ \kappa_* = 1/( 1+2 r_*)$. 
\end{lemma}
 
 \begin{proof}  
 By the homogeneity we may assume that $|v|_r =1$. 
 For $R\ge1$ let us set 
 $
 v^R(s) = v(s) \chi_{|s|\le R} \in L_{\infty}. 
 $
Then
\be\lbl{k7} 
|v^R|_r \le 1, \quad |v- v^R|_{r-1} \le  R^{-1}. 
\ee
Now let us write the increment $ K^j(\tauz+\tau)( v)  - K^j(\tauz)(v) $ as 
\[ \begin{split} 
 \big( K^j(\tauz+\tau)( v)  - K^j(\tauz+\tau)( v^R) \big) + \big( K^j(\tauz+\tau)( v^R)  - K^j(\tauz)( v^R) \big)  \\
 + \big( K^j(\tauz)( v^R)  - K^j(\tauz)(v) \big) =: \Delta_1 +  \Delta_2 + \Delta_3.
 \end{split} 
\]
Recalling that \eqref{k2} and  \eqref{k3} hold for functions from $L_{\infty,r'}$ with $r'>d$, we get from \eqref{k7} that 
$
|\Delta_1 |_r + |\Delta_3 |_r \le C_r  R^{-1}. 
$
To estimate $\Delta_2$ we 
 set $ {\Delta_2^R} = {\Delta_2}\chi_{|s| \le R}$. Since  by \eqref{k2}, 
$
|{\Delta_2}|_{r+1} \le 2 C_r, 
$
then
$
|{\Delta_2} - {\Delta_2^R}|_{r} \le 2 C_r  R^{-1}. 
$
For $|s|>R$ the function $ {\Delta_2^R}$ vanishes, while for $|s|\le R$
 in view of Lemma~\ref{l_k2}  we have
\[\begin{split}
|{\Delta^R_2}_s| &= |{\Delta_2}_s|=
\big| K^j_s (\tauz+ \tau)( v^R) -  K^j_s (\tauz)( v^R)  \big| \\
& \le C_r \int_{\Sigma_s} \mu_s(dv_1 dv_2) 
\big| \cZ^j(\tauz+\tau;  \vs) -  \cZ^j(\tauz;  \vs) \big|  \frac{|v_1|  \dots |v_4|}{|v_j|} 
\chi_{\{|s_j|\le R\ \forall j\}}\\
& \le C_{1r} \ga^0(R^2) \tau   \int_{\Sigma_s} \mu_s(dv_1 dv_2)  \frac{|v_1|  \dots |v_4|}{|v_j|} 
\le C_{2r} \lan s\ran^{-r-1} 
 \tau R^{2r_*},
\end{split}
\]
where to get the last inequality we used \eqref{4.2}. 
We have seen that the $\cC_r$-norm of the increment is bounded by 
$
C_r (R^{-1}+ \tau R^{2r_*}),
$
for any $R\ge1$. Choosing $R= \tau^{-1/(1+2r_*)}$ we achieve \eqref{k8}. 
 \end{proof} 
 
 \subsection{Kinetic equation}\lbl{ss_kin_eq}
 Now we will apply the obtained results to  the kinetic equation \eqref{k1}. Since the function $b(\cdot)^2 := \{ b(s)^2\} 
  \in \cC_r(\R^d)$ for
 all $r$, since $\cL$ is the operator of multiplying by the function $2\ga_s$ as in \eqref{diss_op}, 
  \eqref{gamma}, and the operator $K$ satisfies 
 \eqref{k2},  \eqref{k3}, then for small enough $\eps>0$  eq.~\eqref{k1} has a unique solution, belonging to 
  $\cC_r(\R^d)$ for each $r$, which in a Lipschitz way depends on the r.h.s. of the equation, when the latter 
    deviates from $b(\cdot)^2 $. 
  Namely, the following result, where $X^r$ stands for the space
  $C(0,\infty; \cC_r(\R^d))$,   given the norm
  $
  |v(\cdot)|_{X^r} = \sup_{ t\ge0} |v(t)|_r, 
  $
  may be easily verified (a proof of a similar fact may be found in Section~4 of \cite{DK1}).

\begin{lemma} \lbl{l_k4}
	 For any $r>d$,
	
1) There exists $\eps_*$, depending on $b(\cdot), r$ and $r_*$, 
 such that for $0\le \eps \le \eps_*$ eq.~\eqref{k1} has a unique solution $\zz(\tau)$, belonging to $X^r$.
  It satisfies 
\be\lbl{k99}
|\zz|_{X^r} \le C_r | b^2|_r.
\ee

2) If $\zz^0(\tau)$ is a solution of the linear equation 
 \eqref{k1}$\!\!{}\mid_{\eps=0}$, then $| \zz -\zz^0|_{X^r} \le C_r \eps^2$. 
If a curve $ \zz'(\tau)$ solves 
 \eqref{k1} with $2b(s)^2$ replaced by $2b(s)^2 +\xi_s(t)$, where $\xi\in X^{{r}}$ and $|\xi|_{{X^r}} \le1$, then 
 $| \zz -\zz'|_{X^{{r}}} \le C_{{r}} | \xi|_{X^{{r}}} $.
\end{lemma}
\medskip
The lemma's assertion holds as well for non-zero initial conditions $\zz(0)\in\cC_r(\R^d)$ in \eqref{k1}, but we do not need this.

Let $K(\infty)$ be the operator, obtained from $K(\tau_0)$ by replacing in \eqref{k11} the kernels $\cZ^j(\tau_0; s)$, $s\in \R^d$, by 
$\cZ(\infty;\vs)$ (see Lemma~\ref{l_k1}). 
Let $r>d$ and $\zz^\eps\in \cC_r(\R^d)$ be a solution of the limiting stationary equation 
\be\lbl{z_eps}
\cL \zz^\eps - \eps^2 K(\infty) (\zz^\eps) = 2b(\cdot)^2
\ee
in the vicinity of $\cL^{-1} (2b^2)$, existing for small $\eps$
by the inverse function theorem. Since $b^2(\cdot) \in \bigcap_r \cC_r(\R^d)$ and, 
 as in \eqref{k2}, the map $K(\infty)$ is one-smoothing, then decreasing $\eps_*$ if needed
we achieve that $\zz^\eps \in \bigcap_r \cC_r(\R^d)$ for  $\eps\le\eps_*$ and  
\be\label{prob}
| \zz^\eps  - 2 \cL^{-1} (b^2)|_m \le   C_m \eps^2 \quad \forall \, m. 
\ee 
Here and below the constants depend on $b$ and $r_*$. 
 
Let us consider the curve $w(t) = \zz(t) -\zz^\eps$. It satisfies the equation 
\[\begin{split}
\dot w +\cL(w)& =\eps^2\big (K(t) (\zz) - K(\infty)(\zz^\eps)\big) \\
&=
\eps^2\big [
\big( K(t) (\zz)  - K(t) (\zz^\eps)\big) -\big( K(t)(\zz^\eps)-  K(\infty)(\zz^\eps)\big)\big]
\end{split} 
\]
and $ w(0)=-\zz^\eps.$
 Denote
  $
 K(\tau)(\zz^\eps) - K(\infty) (\zz^\eps) =: -\eta(\tau).
 $
 In view of Lemmas~\ref{l_k1} and \ref{l_4.2}, 
 $
 |\eta(\tau)|_r \le C_r e^{-2\tau}
 $
 for $\tau \ge0$. Next, regarding the difference 
 $
 K(\tau)(\zz(\tau)) -  K(\tau)(\zz^\eps)
 $
 as an operator, linear in $w(\tau) = \zz(\tau)- \zz^\eps$ and quadratic in $(\zz(\tau), \zz^\eps)$, we write it as 
 $
   \cK(\tau)(w(\tau)).
 $
 Then by \eqref{k3} and \eqref{k99}, 
 $
 | \cK(\tau) w|_{r+1} \le C_r |w|_r$,  $\forall r>d$. 
 Finally, we substitute 
 $$
 w(\tau) = v(\tau) +y(\tau), \quad v(\tau) = -e^{- \tau \cL}  \zz^\eps, 
 $$
 and re-write the equation on $w$ as an equation on $y$: 
 $$
\dot y +\cL y = \eps^2 \cK(	\tau) \big( v(\tau) + y(\tau) \big) + \eps^2 \eta(\tau), \qquad y(0)=0. 
$$
Or
 \be\lbl{k13}
y(t) = \eps^2 \int_0^t e^{-(t-s)\cL} \big[ \cK({s}) \big( v({s}) + y({s}) \big) +  \eta({s})\big] ds.
\ee 
Let $Y^r$ be the space of continuous curves $y: \R_+ \to \cC_r(\R^d)$, vanishing   at zero, with finite norm 
$
| y|_{Y^r} = \sup_{t\ge0} e^t |y(t)|_r. 
$ 
Let $\frak B$ be the linear operator 
$$
\fB(y)(t) = \int_0^t e^{-(t-s) \cL} \cK(s) (y(s)) \,ds.
$$
Then the equation for $y$ may be written as 
 \be\lbl{k14}
y(t) = \eps^2 \Big(\fB(y)(t)  + \fB(v)(t) +
\int_0^t e^{-(t-s)\cL}  \eta({s}) ds
\Big). 
\ee 
If $ | \tilde y|_{Y^r} =1$, then
$$
| \fB (\tilde y(t))|_{r+1} \le \int_0^t \big| e^{-(t-s)\cL} \cK(s) (\tilde y(s))\big|_{r+1} ds \le C'_r \int_0^t e^{-2(t-s)} e^{-s} ds < C'_r e^{-t}. 
$$
So $\fB: Y^r \to Y^{r+1}$ is a bounded linear operator if $r>d$, and accordingly the operator 
(id$\, - \eps^2\fB)$  is a linear isomorphism of $Y^r$ if $r>d$ and  $\eps$ is sufficiently small. It easy to see that $\fB (v)$ and $\int_0^t e^{-(t-s)\cL}  \eta({s}) ds$ 
both belong to all spaces $Y^r$. Then in view of \eqref{k14}, $| y|_{ Y^{r+1}} \le C \eps^2$. Since the operator $\fB$ is 1-smoothing, then by induction 
we get that $y$ belongs to all spaces $Y^r$. We have proved that 

\begin{lemma} \lbl{l_k5}
The solution $\zz(\tau)$, constructed in Lemma \ref{l_k4}, may be written as 
$$
\zz(\tau) =( \text{id}\, - e^{-\tau \cL} ) \zz^\eps +y(\tau), \;\; \text{where} \;\; 
|y(t)|_{r} \le C_r\eps^2 e^{-t}\;\; \forall\, t\ge0, \; \forall\,r\,.
$$
Here   $\zz^\eps$ is defined in \eqref{z_eps} and satisfies \eqref{prob}. 
\end{lemma}

\subsection{Energy spectra of quasisolutions and  kinetic equation}\lbl{ss_en_spectr} 
In this section we prove our main result. Namely, we show that the energy spectrum \eqref{energ_sp} of the quasisolution 
$\nL{s}(\tau) = \EE |A_s(\tau)|^2$ of eq. \eqref{ku4}  with large $L$  is
$\eps^3$-close to the solution $\zz(\tau)$  
of the WKE \eqref{k1}, constructed in Lemmas~\ref{l_k4},~\ref{l_k5}. 
 By  \eqref{diff-limit},
  it suffices to prove this for $\nL{s}$ replaced by $ n_s^{\le2}$. Let us denote $ w_s(\tau) =  n_s^{\le2}(\tau) -\zz_s(\tau)$;  
  then $w_s(0)=0$. Recall that $\eps_*$ is defined in Lemma~\ref{l_k4}.

\begin{lemma} \lbl{l_k6}
If $r>d+1$ and $\eps \le C_{1r}^{-1} \le \eps_*$  for an appropriate constant $C_{1r}$, then for    
 any ${\tauz}\ge0$ and $0<\tau\le1/2$,  
 \be\lbl{k15}
|w({\tauz} +\tau)|_r \le (1-\tau/2)| w({\tauz})|_r + C_{2r} \tau
\eps^2(\tau^{\kappa_*} +{   \eps^2}),
\ee  
where $\kappa_* =1/(1+2r_*)$. 
\end{lemma}

\begin{proof}
Since by \eqref{k1} 
$$
\zz({\tauz} +\tau) =e^{-\tau\cL} \zz({\tauz}) +2 \int_0^\tau e^{- t\cL} b^2dt+ \eps^2 \int_{{\tauz}}^{{\tauz}+\tau}
e^{-({\tauz}+\tau-t) \cL} K(t) (\zz(t)) dt, 
$$
then in view of \eqref{n_increment} and  \eqref{I_tau}
 \be\lbl{k16}
w(\tau +{\tauz}) = e^{-\tau\cL} w({\tauz}) +\eps^2 \Delta +\cR,
\ee 
where $\cR$ is as in \eqref{n_increment} and 
$$
\Delta =\int_{{\tauz}}^{{\tauz}+\tau} e^{-({\tauz}+\tau-t) \cL}   \Big( K({\tauz}) (n^{\le2}({\tauz}))-    K(t) (\zz(t))  \Big)dt.
$$
Note that in view of Lemma \ref{l_k4}  and estimates \eqref{bounds},
\be\label{boun}
| n^{\le2} (\tau)|_r, \ |\zz(\tau)|_r \le C_r \quad\text{for all $\tau$ and all $r$},
\ee
with  suitable constants $C_r$. 
Let us re-write $\Delta$ as follows:
\[\begin{split}
\Delta&= \int_{{\tauz}}^{{\tauz}+\tau} e^{-({\tauz}+\tau-t) \cL}   \big( K({\tauz}) (n^{\le2}({\tauz}))-    K({\tauz}) (\zz({\tauz})) \big)dt \\
&+ \int_{{\tauz}}^{{\tauz}+\tau} e^{-({\tauz}+\tau-t) \cL}   \big( K({\tauz}) (\zz({\tauz}))   -    K(t) (\zz({\tauz})) \big)dt \\
 &+\int_{{\tauz}}^{{\tauz}+\tau} e^{-({\tauz}+\tau-t) \cL}   \big( K(t) (\zz({\tauz}))  -    K(t) (\zz(t)) \big)dt =: \Delta^1+\Delta^2+\Delta^3.
\end{split}
\]
By \eqref{k3} and \eqref{boun}, 
$
|\Delta^1 |_r \le C_r\tau |w({\tauz})|_r.
$
Similar,
$$
|\Delta^3 |_r \le C_r \tau \sup_{{\tauz}\le t\le {\tauz}+\tau} |\zz(t) -\zz({\tauz})|_r\le 
 C_r\tau^2
$$
since
$
|\zz(t) -\zz({\tauz})|_r \le \int_{{\tauz}}^t |  -\cL\zz(l) + \eps^2 K(l) (\zz(l)) +2b^2|_r dl
$
and
 $
|\zz(t)|_{r+r_*} \le C'_r 
$
by \eqref{boun}. Now let us consider $\Delta^2$. By Lemma \ref{l_k3}, 
$
|K({\tauz}) (\zz({\tauz})) -K(t) (\zz({\tauz}))|_r \le C_r(t-{\tauz})^{\kappa_*}. 
$
So 
$
\Delta^2 \le C_r \int_0^\tau t^{\kappa_*} dt = C'_r \tau^{1+\kappa_*}. 
$

Since $\cL \ge2  \mathbbm{1}$ and ${\tau}\le 1/2$, then  
$
|e^{-\tau \cL} w({\tauz})|_r \le (1-\tau) |w({\tauz})|_r.
$
Now \eqref{k16},  \eqref{R_s} 
  and the bounds on $\Delta^j$ imply that
$$
|w({\tauz} +\tau)|_r \le (1-\tau) |w({\tauz})|_r + C_r\eps^2  \tau
\big( |w({\tauz})|_r +\tau+ \tau^{\kappa_*} + (\tau+ {   \eps^2})
\big),
$$
and \eqref{k15} follows if $C_{1r}^{-1}\ll 1$. 
\end{proof}

For any $0<\tau\le 1/2$, any 
 $N$ and for $k=0,\dots, N$ let us set $w_k=|w(k\tau)|_r$. Let the function $k\to w_k$ attains its maximum at a
point $k$ which we write as $k:= k_0+1$. If $ k_0+1=0$, then $w_k\equiv0$. Otherwise in view of \eqref{k15} we have 
$$
w_{k_0} \le w_{k_0+1} \le 
(1-\tau/2) w_{k_0} + C_{2r} \tau \eps^2 (\tau^{\kappa_*} + {   \eps^2}).
$$
So $ w_{k_0} \le  2 C_{2r} \eps^2  (\tau^{\kappa_*} + {   \eps^2}) $ and
$$
\max_{0\le k\le N} |w(k\tau)|_r = w_{k_0+1} \le 3  C_{2r} \eps^2
(\tau^{\kappa_*} + {   \eps^2})
$$
since $\tau\le 1/2$. Applying again \eqref{k15} with ${\tauz}=k\tau$ and $\tau$ replaced by any $\bar\tau \in (0, \tau)$, and using that in the 
formula above $N$ is any, we get that $|w(t)|_r \le 4 C_{2r} \eps^2
(\tau^{\kappa_*} + {   \eps^2})$, for any $t\ge0$. Sending $\tau\to0$ (and estimating norms $|\cdot|_{r}$ with $r<d+2$ via 
$|\cdot|_{d+2}$) and then using \eqref{diff-limit}  we finally get 

\begin{theorem} \lbl{t_k1}
For any $r$ there exist positive  constants 
 $C_{1r}, C_{2r}, C_{3r} $ such that if  $\eps \le C_{1r}^{-1}$, then 
 \be\lbl{k17}
 \sup_{\tau\ge0} | n^{\le2}(\tau) -\zz(\tau)|_r \le C_{2r} 
   {   \eps^4}
\ee 
and   if $L\ge \eps^{-2}$,  then 
\be\lbl{k18}
  \sup_{\tau\ge0} | \nL{\cdot} (\tau) -\zz(\tau)|_r \le C_{3r}
   \eps^3. 
\ee 
\end{theorem} 

 Relation  \eqref{k18} together with Lemma~\ref{l_k5} give a control over  the   long time behaviour 
of the spectra of quasisolutions of \eqref{ku4} in terms of the 
 stationary solution  $\zz_\eps$ of the limiting kinetic equation (see \eqref{z_eps}): 
  $$
| \nL{\cdot}(\tau) -\zz^\eps|_r \le C_r(e^{-\tau} +\eps^3),
\qquad \forall\, \tau\ge 0\,.
$$
By  Proposition~\ref{p:approx}  with $d\geq 3$ this result and \eqref{k18}  extend  to the 
spectra of quasisolutions of \eqref{ku5}, defined in 
  \eqref{n_s-a}, as expressed in
\begin{theorem} \lbl{t_k2}
For any $r$ there exist positive constants 
 $C_{4r}, C_{5r}$ such that if  $\eps \le C_{4r}^{-1}$ and $L\ge \eps^{-2}$, then 
\begin{align}\lbl{k19}
  \sup_{\tau\ge0} | \no_{\cdot,L} (\tau) -\zz(\tau)|_r \le C_{4r}
  \eps^3\,,\\\lbl{k20}
  | \no_{\cdot,L}(\tau) -\zz^\eps|_r \le
   C_{5r}(e^{-\tau} +\eps^3), \qquad \forall\, \tau\ge 0\,.
\end{align}
\end{theorem} 
 Relation \eqref{k17} extends to the energy spectra of quasisolutions of \eqref{ku5} analogously.

\appendix

\section{Proof of Lemma~\ref{l:intersection}}\label{app:intersection}

In this appendix we suppose that the dimension $d$ satisfies $d\ge2$. 

\subsection{Idea of the proof and general setting}
In Lemma~\ref{l:intersection} (up to an obvious scaling) 
we have  to estimate the number of integer points on a quadric 
inside a large box. The idea is to embed the integral points of the box in an affine space over a large finite field and then apply powerful algebraic geometry techniques to estimate the needed number (note that this identification
of bounded integers with elements of a finite field is ubiquitous in coding theory and combinatorics).  It is possible mainly due to the fact that this techniques permits to count points defined over a finite field using some geometric information (essentially the dimension, the degree and irredundant decomposition)  on the corresponding algebraic set over the algebraic closure of our finite field. We begin with recalling some basic definitions and results concerning such algebraic sets
(see, for example, the first chapter of the book \cite{Sh}).

{\bf Affine algebraic sets.} Let us   fix an algebraically closed field $K$. 
Let ${\mathbb A}^m=K^m$
be the $m$-dimensional affine space over $K$, and let $F_1,\ldots,F_s\in K[T_1, \ldots, T_m]$ 
be non-zero polynomials. 
Then an {\em affine algebraic set} (AAS) $X$ is just the set of common zeros of these polynomials:
\be\label{X}
X\!=\!\{(a_1,\ldots,a_m)\in K^m:F_1(a_1,\ldots,a_m)\!=\!\ldots\!=\!F_s(a_1,\ldots,a_m)\!=\!0\}. 
\ee
{\bf Irreducibility.} An AAS $X$ is {\em reducible} if $X=X_1\cup X_2$ with two non-empty AAS $X_1, X_2$
s.t. $X_1\ne X, X_2\ne X.$ If it is not the case, $X$ is called {\em irreducible,} or an {\em affine algebraic variety}
(see \cite{Sh}, Section I.3.1).

\begin{theorem} (Irredundant decomposition) Any non-empty AAS $X$ can be presented as  
	\be\label{ir} X=X_1\cup \ldots \cup X_l\ee
	for irreducible $X_1, \ldots, X_l$ such that  $X_i\not\subset X_j$ for $i\ne j$. The  decomposition is unique up to order.\end{theorem}

This decomposition is especially simple for a hypersurface $X$, i.e. when in \eqref{X}    $s=1$. Then 
$F=F_1(T_1, \ldots T_m)= \Pi_{j=1}^l Q_j$ for irreducible polynomials $Q_j$ which are
uniquely  defined up to multiplicative constants and permutation since the ring $K[T_1, \ldots, T_m]$ is a unique factorisation domain, see, e.g. Chapter IV of \cite{La}, and then $X_j=\{Q_j=0\}.$ This uniqueness is true under the condition which we can and will suppose to hold, namely, that the polynomial $P$ does not have multiple divisors, i.e., all $Q_j,j=1,\ldots, l$ are distinct. For further references we formulate a corollary of the  unique factorisation property (see \cite{Sh}, Section~I.3.1)

\begin{lemma}\label{dec} i). If $X$ and $Y$ are hypersurfaces, then $X=Y$ if and only if the corresponding
	polynomials $P_X$ and $P_Y$ are proportional. Moreover if $Y$ is irreducible and $X\subseteq Y$, then $X=Y.$\smallskip
	
	ii). If $\deg P_X=2$ then  there are exactly  two  possibilities: either $X$ irreducible   (in this case it cannot contain a
	hyperplane), 
	or $X=X_1\cup X_2$ for two affine hyperplanes,  defined by affine linear polynomials $l_1$ and $l_2$, and $P_X=l_1l_2$.
\end{lemma}

{\bf Dimension.} One can define the {\em dimension} $r=\dim X\in \{0,1,\ldots,m\}$ as follows:
$\dim X=\max\{\dim X_i,i=1,\ldots,l\}$ for \eqref{ir}, and for an irreducible AAS \, $X$
$$
\dim X=\max\{r:X=X_0\supset X_1\supset\ldots\supset X_r\ne \emptyset\},
$$
where all $X_i, i=0,\ldots, r$ are irreducible AAS and all inclusions are strict. The {\em codimension}
of $X$ is $\mathrm{codim}\, X=m-\dim X$.

In particular, if $ \dim X=m$ then $X={\mathbb A}^m$ (indeed, if $X \subset {\mathbb A}^m, X\ne{\mathbb A}^m$ the  definition implies that $\dim X<\dim {\mathbb A}^m=m$) and if $ \dim X=0$ then $X$ is a finite set.
The codimension of $X$ as in \eqref{X} is at most $s$.

From the definition we get immediately (see \cite{Sh}, Section I.6.2)
\begin{lemma}\label{dim}   If $Y\subset X, Y\ne X$ and $X$ is irreducible, then $\dim X>\dim Y$,  $\mathrm{codim}\, X<\mathrm{codim}\,Y.$ 
\end{lemma}

{\bf Degree.} Let  $X\subset{\mathbb A}^m$ be a non-empty AAS, $ \dim X=r$. Then its  {\em degree}
$\deg X$ is defined as follows:
$$ \deg X=\max\{\hbox{cardinality of}\;\;X\cap L: \dim\big(X\cap L\big)=0\},$$
where $L\subset{\mathbb A}^m$ is an affine plane with $\dim L=m-r$.

\begin{lemma}\label{deg}   If $X$ is a hypersurface (i.e. in \eqref{X}  $s=1$), then $\operatorname{codim} X = 1$ and
	$\deg X=\deg F_1$. 
\end{lemma}

The famous {\em Bezout theorem} in the most elementary setting over the field $\C$ states that
$$\deg X\le \Pi_{i=1}^s \deg F_i.$$

\subsection{Finite fields' Bezout theorem}  
From now on the field $K$ is the algebraic closure $\bar{\mathbb F}_p$ of a finite field ${\mathbb F}_p$, where $p$ is a large prime number (see \cite{La}, Section V.5). 

We will use a version of Bezout's theorem over finite fields which can be deduced from its general form, e.g. \cite{Fu}, and is also explicitly stated and proved in \cite[Corollary 2.2]{LR}.

\begin{theorem}\label{bz} Let $K=\bar{\mathbb F}_p$ and the AAS\  $X$ in \eqref{X} is such that  $F_j \in {\mathbb F}_p[T_1, \dots, T_m]$, 
	$\deg F_j=d_j$,  $j=1,\dots, s$, and 
	$\dim X=r$. Then 
	$$ |X\cap{\mathbb F}_p^m |\le p^r \Pi_{i=1}^s d_i.$$
\end{theorem}

\subsection{Preliminary result}
Let $q_1,\ldots, q_s$, $s\ge 1$, be  polynomials of degree at most two
in $m\ge s$ variables,  $q_i\in \Z[X_1,\ldots, X_m]$,
with $q_i(0)=0,$\linebreak $ i=1,...,s$. 
Consider the geometric quadrics $Q_i=\{x\in \R^m: q_i(x)=0\}$ and their intersection $Q=\cap_{i=1}^s Q_i$. The latter is not empty since $0\in Q$. 

Let
$B_M^{m}\subset \mathbb R^{m}$ 
be  an open cube $\{\left|x\right|_\infty<M\}$  with some  $M\ge 1$.  Consider the set
$$
S_m(M,Q) =Q\cap \Z^{m}\cap B_{M}^m .
$$

Let $p$ be a prime and  $q^{(p)}_{i}\in {\mathbb F}_p[X_1,\ldots, X_m]$ denote the polynomials $q_{i}\mod \,p$ over the finite field ${\mathbb F}_p$. 
Consider the sets 
$$
Q^{(p)}_{i}=\{x\in  K^m: q^{(p)}_{i}(x)=0\}
$$
and their intersection $Q^{(p)}=\cap_{i=1}^s Q^{(p)}_{i}$  (recall that now $K=\bar{\mathbb F}_p $ is the algebraic closure of ${\mathbb F}_p$). 
We will be interested mainly in  the cardinality of $Q^{(p)}({\mathbb F}_p):=Q^{(p)}\cap{\mathbb F}_p^m$
as a tool to estimate $|S_m(M,Q)| $.

\begin{proposition}\label{01} Let $M\ge 1$ and  suppose that  a  prime $p>2M $  satisfies $p=2M(1 +r(M))$, where $r(M)>0$. 
	Suppose also that 
	$\deg q_i\le2$ for each $i$ and that 
	the AAS $Q^{(p)}$   is of dimension $m-s\;$ (and of codimension s, that is,  the $s$ quadrics $Q^{(p)}_{i}$ intersect properly):
	\be\label{*} 
	\quad \quad\quad\quad \quad \dim \,Q^{(p)}=m-s.\hskip 3 cm
	\ee
	Then 
	\begin{equation}
		\label{1.1} |S_m(M,Q)|  \le   2^m\big(1+r(M)\big)^{m-s}M^{m-s} \, .
	\end{equation}
\end{proposition}

	By the Bertrand's postulate, for any $M\geq 1$ there is a $p$ satisfying $2M<p<4M$, so when applying Proposition~\ref{01} we will chose 
	\be\lbl{r<1}
	r(M)<1.
	\ee
	Moreover, by the Prime number theorem, for large $M$ one can chose $r(M)=o(1)$. 

{\em Proof of Proposition~\ref{01}.}  Let
$ \Pi: S_m(M,Q)\longrightarrow {\mathbb F}_p^m $
be defined by
$$\Pi(x_1, \ldots, x_m)=(x_1 \!\!\mod p, \ldots, x_m\!\!\mod p).$$
Then $ \Pi$ is injective and its image is contained in $Q^{(p)}\cap{\mathbb F}_p^m\subset {\mathbb F}_p^m$. Indeed, the last assertion is clear and 
the injectivity is established as  follows: if  
\begin{equation*}\hskip1 cm (x'_1\!\mod p, \ldots, x'_m \!\mod p) =(x_1\!\mod p, \ldots, x_m\!\mod p)\hskip2 cm \end{equation*}
but $x'\ne x$, then for some $  i\in\{1,...,m\}$ we have  $x'_i\!\mod p =x_i\!\mod p $, but $\,x_i'\ne x_i$. 
Consequently, $|x_i-x'_i|\ge p>2M$ which contradicts the condition $x_i,x'_i\in B_{M}^m$.
Apllying then Theorem \ref{bz} to $X=Q^{(p)}$ we get the conclusion since
\begin{equation*}
	|S_m(M,Q)|\le |Q^{(p)}({\mathbb F}_p)|\le
	2^sp^{m-s}= 
		2^m\big(1+r(M)\big)^{m-s}M^{m-s}\,.
\end{equation*}
\qed

\subsection{Main estimate for  $N=3$ and $4$}

Now we pass to the proof of Lemma~\ref{l:intersection} and denote
$
\big|  \tilde Q_L(z_1,v)  \cap B_R^{(N-2)d} \big| = s(R,{\tilde { Q}},L).
$
Consider the set
\begin{equation*}
	\begin{split}
		\quad S'\big(R,{\tilde {\mathcal Q}},L\big) &=
		{\tilde {\mathcal Q}}\cap \Z^{(N-2)d}\cap B^{(N-2)d}_{RL}\,,\qquad {\tilde {\mathcal Q}}=\tilde Q(Lz_1,Lv), 
	\end{split}
\end{equation*}
and denote by $s'(R,{\tilde {\mathcal Q}},L)$  its cardinality. Then 
$$
s' (R,{\tilde {\mathcal Q}},L) = s(R,{\tilde { Q}},L)\,,
$$
since 
the map 
$(z_2,\ldots,z_{N-1}) \mapsto (L z_2,\ldots,L z_{N-1})$ is  a
bijection between  the sets 
$\tilde Q_L(z_1,v)  \cap B_R^{(N-2)d} $ and $ S'(R,{\tilde {\mathcal Q}},L)$.

Let us estimate $s'(R,{\tilde {\mathcal Q}},L)$ through Proposition~\ref{01} with
	$m=(N-2)d$ and $s=N-2$, where   $N=3$ or $4$.  
	To this end it suffices to find $M\geq RL$ and $p>2M$ such that  assumption \eqref{*} is fulfilled  
	for any $(z_1,v) \in B^{2d}_{(N-1)R}$ satisfying $(z_1,v)\in  Q^0_{1\,L} \cap A_2$. 
	Lemma~\ref{01-34} below establishes this for $M=NRL/2$ and any $p>2M$. Then, applying \eqref{1.1} with 
	$r(M)<1$ (see \eqref{r<1}), we conclude the proof of Lemma~\ref{l:intersection}.
	
	For a prime $p$ and $a,b\in \mathbb{F}_p^{d}$ let us consider  algebraic sets $\tilde Q_{j}^{(p)}$  over  $K=\bar{\mathbb F}_p$: 
$$
\tilde Q_{j}^{(p)} (a,b): = \{(z_2,\dots,z_{N-2})\in  K^{(N-2)d}: \,
q^{(p)}_{j}(z_2,\ldots,z_{N-2}; a,b)=0\} ,
$$ 
where   $q^{(p)}_{j}(z_2,\ldots,z_{N-2}; a,b)$ are    the
	residues modulo $p$ of the polynomials  $q_{j}(z_2,\ldots,z_{N-2}; a,b)$, defined by \eqref{eq:q_tilde}. We set $\tilde Q^{(p)}=\cap_{1<j<N} \tilde Q_j^{(p)}$ for the intersection of the algebraic  sets.

\begin{lemma}\label{01-34}
	Let  $N\in\{3,4\}$, 
		$(z_1,v)\in Q_{1 \, L}^0\cap A_2$ (see \eqref{A2}) and let $p $ be
		a prime satisfying $p> \max(|Lz_1|_\infty, |Lv|_\infty)$. Then 
	\begin{equation}\label{1.1-34}  \quad \quad\quad\quad \quad    \dim
		\,\tilde Q^{(p)}(Lz_1,Lv)=(N-2)(d-1) \ .
		\hskip 3 cm \end{equation}
\end{lemma}
The assumption $p> \max(|Lz_1|_\infty, |Lv|_\infty)$ ensures that $Lz_1$ and $Lv$ are different from zero in $K^d$.
	In particular, for  $(z_1,v) \in B^{2d}_{(N-1)R}$ this assumptions is satisfied if $p>2M$  with $M=NRL/2$.

{\em Proof of Lemma~\ref{01-34}}. 
Let $N=3$. Then $N-2=1$ and $\tilde Q^{(p)}$ is given by the unique equation $q_2^{(p)}(z_2;L z_1,Lv)=0$, for a fixed $(z_1,v)$.  By Lemma~\ref{l:ind} the equation is non-trivial, so 
the conclusion  follows from Lemma~\ref{deg}. 

{$  N=4.$} The codimension of the intersection of  two quadrics is at most two. We have to show that it is two 
(and not one). The result will  follow from the next three lemmas.

{\begin{lemma}\label{l}
		Let $\mathcal{Q}_1=\{\tilde q_1=0\},\mathcal{Q}_2=\{\tilde q_2=0\}$ be two  linearly independent quadrics over $K$.   Then the codimension of $\mathcal{Q}_1\cap \mathcal{Q}_2$ is one if and only if  $\tilde q_1$ and $\tilde q_2$   have a mutual affine linear factor $l(x)$.\end{lemma}}

{\em Proof.} Let the codimension of the intersection be one. In this case if 
one of $\mathcal{Q}_1,\mathcal{Q}_2 $ is irreducible, then $\mathcal{Q}_1= \mathcal{Q}_2 $
by Lemma \ref{dim} with $Y=\mathcal{Q}_1 \cap \mathcal{Q}_2$. 
However this is impossible by Lemma \ref{dec}.\,i)  since $\tilde q_1$ and $\tilde q_2$ are independent. Therefore by Lemma~\ref{dec}.\,ii) 
$\mathcal{Q}_1=H_1\cup H_2$ and $  \mathcal{Q}_2=H'_1\cup H'_2$, with hyperplanes
$H_1, \dots, H'_2$. If all $H_i\cap H'_j$ are of codimension two then
$${\rm codim}\, Q_1\cap Q_2= {\rm codim}\,\left(\cup (H_i\cap H'_j)\right)=\min({\rm codim}\, H_i \cap H'_j)=2.$$  Therefore, at least one of  $H_i\cap H'_j$ is of codimension one and then we have $\ker (l(x))=H_i\cap H'_j\subset \mathcal{Q}_1\cap \mathcal{Q}_2$ for
an affine linear $l(x)$. Hence $l(x)$ divides both $\tilde q_1$ and $\tilde q_2$ by Lemma \ref{dec}.\,ii).

The inverse statement is obvious. 
\qed

\begin{lemma}\label{l:ind}
	For any $N>2$, if the matrix $\alpha$ is
		irreducible and $(z_1,v)\in  Q_{1\,L}^0\cap A_2$  is such that $Lz_1,\, Lv \ne 0$ in $K^d$, then  the polynomials $q^{(p)}_j(\cdot, Lz_1, Lv)$, $1< j< N$ are linearly independent    over $K$. 
	In particular, each $q_j^{(p)}$ is a non-zero polynomial. 
\end{lemma}

{\it Proof.}
Consider a linear combination $\sum_{1<j<N}c_jq^{(p)}_j$. By the homogeneity in $(z_2,\ldots,z_{N-1})$ it vanishes identically if and only if
\begin{equation}\label{eq:lin_comb}
	\begin{split}
		\sum_{1<j<N}c_jz_j\cdot(\alpha_{j1} (Lz_1)+\alpha_{jN} (Lv))\equiv 0\,,\\
		\sum_{1<i,j<N}c_j(\alpha_{ji}-\alpha_{jN}\alpha_{1i})z_j\cdot z_i\equiv 0\,.
	\end{split}
\end{equation}
Arguing by induction and using that the matrix $\alpha$ is irreducible we  construct a partition $E_0, \dots, E_M$, $M\ge1$, of the set 
$\{1,\ldots,N\}$ such that $E_0 =\{1,N\}$ and for $n\ge1$, 
$$
E_n =\{ j: \alpha_{jl} =0\; \forall\, l\in E_{n'}, \; n' \le n-2,\;\; \text{and $\exists\, l'\in E_{n-1}$\; such that \, $\alpha_{jl'} \ne0\}$. 
}
$$
Since  $(z_1,v)\in Q_{1\,L}^0$ and $Lz_1,\, Lv \ne 0$ in $K^d$, then the term in brackets in
the first line of \eqref{eq:lin_comb} is not identically zero for each  $j\in E_1$, so  $c_j=0$ for every $j\in E_1$. Using this
in the second line of \eqref{eq:lin_comb} we get:
$$
\sum_{n=2}^M\sum_{m=n-1}^M
\sum_{j\in
	E_n}\sum_{i\in E_m}c_j\alpha_{ji}z_j\cdot z_i\equiv0\,.\\ 
$$
This relation holds  if and only if
$
(c_j-c_i)\alpha_{ji}=0
$
for all $j\in E_n$, $2\le n\le M$, and $i\in E_m$, $n-1\le m\le M$.
We know that  $c_j=0$ if $j\in E_1$.
Starting from $n=2$ and arguing  by induction in $n$ we find that if $c_i=0$ for all $i\in E_{n-1}$, then $c_j=0$ for all
$j\in E_n$. Indeed,  for any $j\in E_n$ there exists at least one
$i\in E_{n-1}$ such that $\alpha_{ji}\ne0$  by the definition of $E_i$, so relation 
$(c_j-c_i)\alpha_{ji}=0$ implies that $c_j=0$ if $j\in E_n$. That is,  $c_j\equiv0$. 
\qed
\medskip

\begin{lemma}\label{l:irr}
	For any $N>2$, if the matrix $\alpha$ is
		irreducible and $(z_1,v)\in  Q_{1\,L}^0\cap A_2$ is such that $Lz_1,\, Lv \ne 0$ in $K^d$, then the polynomials $q^{(p)}_j(\cdot, Lz_1, Lv)$, $1< j< N$, are irreducible.
\end{lemma}
\begin{proof}
	Each polynomial  $q^{(p)}_{j}$ has degree one or two. If its degree is one the assertion is obvious.  Now let the degree be two. 
	Note that in view of \eqref{eq:q_tilde} 
	$q^{(p)}_{j}$  can be written as the scalar products $q^{(p)}_{j} =z_j\cdot l_j(z_2,\ldots, z_{N-1};z_1,v)$ mod 
	$p$, 
	where $l_j$ are surjective affine functions  $l_j:K^{d(N-2)}\longrightarrow K^d$. But
	such scalar product cannot vanish for $d\ge 2>1$ on a hyperplane $H \subset K^{d(N-2)} $ which  by Lemma~\ref{dec}.\,ii) 
	would be the case for a reducible quadric. Indeed,  only   two cases can occur:
	\begin{enumerate}[a)]
		\item the coefficient $\al$ of $z_j$ in $l_j$ is non-zero, or
		\item it is zero but then the coefficient $ \beta$ of some other  $z_i$  is non-zero. 
	\end{enumerate}
	In case a) take the 2-dimensional plane $P(x_1,x_2)$ in
	the whole space,  generated by  two  orthogonal  vectors from the
	$z_j$-space, where the first basis vector is parallel to $\alpha_{1j}z_1+\alpha_{Nj}v\neq 0$ (this vector is non-zero since  $(z_1,v)\in Q_{1\,L}^0$ and $Lz_1,\, Lv \ne 0$ in $K^d$, and for the case a) we have $\alpha_{1j},  \alpha_{Nj}\ne 0$). Then  the restriction of $q^{(p)}_{j}=0$ on $P$ is 
	$\al(x_1^2+x_2^2)+c_1x_1=0$ with $c_1\neq 0$, which is isomorphic
	to $x_1^2+x_2^2=C\neq 0$. 
	This plane quadric in $P(x_1,x_2)$   cannot contain $P(x_1,x_2)\cap H$ (a line or the whole  $P(x_1,x_2)$).
	Indeed, otherwise, supposing by symmetry that the quadric contains a line $x_1=ax_2+b$, we would have that the polynomial 
	$$a^2x_2^2+2abx_2+b^2+x_2^2 -C =(a^2+1)x_2^2+2abx_2+b^2 -C$$
	vanishes identically. This  implies $ab=0$, and if $a=0$ then the term \linebreak
	$(a^2+1)x_2^2=x_2^2\ne 0$, while for $b=0$ the term $b^2 -C= 
	-C\ne 0$.
	
	Similarly, in case b) we take the 4-dimensional vector subspace $P'$  generated by the two first basis vectors in the $z_j$ space and  the  two first basis 
	vectors in the $z_i$ space. The restriction of $q^{(p)}_{j}=0$ on $P'$ is then $\beta(x_1y_1+x_2y_2)+c_1x_1+c_2x_2=0$,  isomorphic to $x_1y_1+x_2y_2=C$
	which   can not contain $P'(x_1,x_2,y_1,y_2)\cap H$. Indeed, else, supposing  by symmetry that $P'(x_1,x_2,y_1,y_2)\cap H\supset \{x_1=a_1x_2+b_1y_1+b_2y_2-c\}$ 
	we get that the following quadratic function of $x_2,y_1,y_2$:
	$$(a_1x_2+b_1y_1+b_2y_2-c)y_1+x_2y_2-C=a_1x_2y_1+b_1y_1^2+b_2y_1y_2-cy_1
	+x_2y_2-C$$  vanishes identically, which is clearly wrong. 
\end{proof}
\medskip

\noindent {\it End of the proof of Lemma \ref{01-34}.} 
Since each  $q_j^{(p)}$ is a non-zero polynomial of degree one or two, then 
to prove Lemma \ref{01-34} we have to consider  three cases. In the first case  both polynomials  $q_2^{(p)}$ and $q_3^{(p)}$ are linear.
Then the codimension of the intersection $\tilde Q^{(p)}$ is two since  they are linearly independent.
In the second case both $q_2^{(p)}$ and $q_3^{(p)}$ are quadratic. Then, according to Lemma~\ref{l}, the codimension  still is two since the polynomials 
are irreducible  by Lemma~\ref{l:irr}.   Finally in the  last case, when one polynomial is linear and another one is quadratic, the assertion is clear since  then the AAS in question 
is an intersection of a quadratic irreducible  surface with a hyperplane. Thus  its codimension  is two by  Lemma \ref{dec}.\,ii).

\begin{remark} \label{r_N>4}
	The proof of Lemma \ref{01-34} follows from three lemmas. Two of them are valid for any $N>2$, but Lemma~\ref{l} holds only for $N=4$ (and
	tautologically holds  for smaller $N$). Still the bi-linear (or linear) nature of the polynomials $q^{(p)}_j$ and direct analysis of the AAS $\tilde Q^{(p)}$,
	jointly with the two lemmas, valid for any $N>2$, allow to prove by hand Lemma~\ref{01-34} for ``not too high" values of $N$, and thus, to prove 
	for those $N$'s Theorem~\ref{t:countingterms}. Unfortunately, for the moment we cannot prove the theorem for all $N>2$; cf. 
	Conjecture~\ref{con_3.6}.
\end{remark}

\section{Proof of Proposition \ref{l:N2} and  Lemma \ref{l:Z^j-growth}} \lbl{sec:Q_s-prop} 

We prove Proposition \ref{l:N2} in Sections~\ref{s:beg_pr_lN2}-\ref{s:end_pr_lN2} and  Lemma~\ref{l:Z^j-growth} in Section~\ref{s:Z^j-growth}.

\subsection{Beginning of the proof of Proposition \ref{l:N2}}
\lbl{s:beg_pr_lN2}

The proof of the proposition is somewhat cumbersome since  we have 
to consider a number of different terms and different cases. 
During the proof  we will often skip the upper index $(0)$, 
so by writing $a$ and $a_s$ we  will mean $a^{(0)}$ and  $a_s^{(0)}$. 
 We will also skip the dependence on $\tau_0$ by writing $c_s^{(i)}(\tau;\tau_0)$ and 
 $\Del a_s^{(i)}(\tau;\tau_0)$ as $c_s^{(i)}(\tau)$ and $\Del a_s^{(i)}(\tau)$. 
Besides, 
 for a complex  function $(w_{s_1,\ldots,s_k}, s_j\in \Z^d_L)$ we denote 
$$
\ssum_{s_1,\ldots,s_k\in\Z^d_L}w_{s_1,\ldots,s_k} = L^{-kd} {\sum}_{s_1,\ldots,s_k\in\Z^d_L} w_{s_1,\ldots,s_k}\, ,
$$
and we introduce the symmetrisation
$$
\cY_s^{sym}(u,v,w;t) =\frac{ L^{-d}}3 \sum_{1,2,3}   \delta'^{1 2}_{3
  s}\de(\oms) \left( 
u_{1} v_{2} \bar w_{3}+v_{1} w_{2} \bar u_{3}+w_{1}
u_{2} \bar v_{3} \right)\,.
$$

We recall that $Q_{s,L}$ is given by  formula \eqref{Q_s-def} and first
consider the term $\EE \Del a^{(2)}_s(\tau) \bar a_s (\tauz+\tau)$. 
Inserting the identity
$a^{(1)}(\tauz+l)=c^{(1)}(l) + \Del a^{(1)}(l)$
into  formula  \eqref{deltas_a} for  $\Del a^{(2)}_s $, we obtain
$$
\EE \Del a^{(2)}_s (\tau) \bar a_s(\tauz+\tau)= N_s + \wt N_s,
$$
where
\be\lbl{NN22}
N_s:=i\, \EE \Big(\bar a_s(\tauz+\tau)
\int_0^\tau e^{-\ga_s(\tau-l)} 3\cY_s^{sym}(a(\tauz+l),a(\tauz+l),\Del a^{(1)}(l))\,dl
\Big)
\ee
and
$$
\wt N_s:= i\,\EE \Big(\bar a_s(\tauz+\tau)
\int_{0}^\tau e^{-\ga_s(\tau-l)} 3\cY_s^{sym}(a(\tauz+l),a(\tauz+l),c^{(1)}(l) )\,dl
\Big).
$$
Thus, 
\be\label{Q_s-def2}
Q_{s,L}= L^2\left(\EE |\Del a_s^{(1)}(\tau)|^2 + 2 \Re N_s + 2\Re \EE\Del a^{(1)}_s (\tau)\bar c^{(1)}_s(\tau) +2\Re \wt N_s\right),\quad
s\in\R^d. 
\ee
We will analyse the four terms above term by term.

\subsection{The first term of $Q_{s,L}$ in \eqref{Q_s-def2}} \label{ss_2}
Due to \eqref{deltas_a}, we  have
\be\lbl{E Del a}
\EE  |\Del a^{(1)}_s(\tau)|^2
=\EE\int_{\tauz}^{\tauz+\tau} dl \int_{\tauz}^{\tauz+\tau} dl'\,
e^{-\ga_s(2\tauz+2\tau-l-l') } \cY_s(a(l))
\,\ov{\cY_s(a(l'))}.
\ee
Writing the functions $\cY_s$ explicitly and applying the Wick theorem, in view of \eqref{corr_a_in_time} we find
\begin{align}\non
	\EE  |\Del a^{(1)}_s(\tau)|^2=
	2 L^{-2d} \sum_{1,2}&
	\dep\delta(\oms)\int_{\tauz}^{\tauz+\tau} dl \int_{\tauz}^{\tauz+\tau} dl'\,
	e^{-\ga_s(2\tauz+2\tau-l-l')}
	\\\non
	&\EE a_1(l) \bar a_{1}(l') \,
	\EE a_2(l) \bar a_{2}(l') \,
	\EE \bar a_3(l) a_{3}(l'),
\end{align}
and note that
$$
\int_{\tauz}^{\tauz+\tau} dl \int_{\tauz}^{\tauz+\tau} dl' \,
e^{-\ga_s(2\tauz+2\tau-l-l')}\le \tau^2.
$$
On account of \eqref{corr_a_in_time}, we can bound
$$
\EE  |\Del a^{(1)}_s(\tau)|^2\le 2\tau^2  \ssum_{1,2}\dep\delta(\oms)
B_{123},
$$
where $B_{123}=B_1B_2B_3$. Since $B_{123}$  with $s_3=s-s_1-s_2$ is a
Schwartz function of 
$s,s_1,s_2$ then Theorem~\ref{t:countingterms} with $N=2$ applies and we find 
\be\label{f_1}
\EE  |\Del a^{(1)}_s(\tau)|^2 \le C^\#(s) L^{-2} \tau^2 .
\ee

\subsection{The second term of $Q_{s,L}$  in \eqref{Q_s-def2}}  \label{ss_3}
To study the term  $2\Re N_s$ we  use the same strategy as above. Namely, 
expressing  in \eqref{NN22}  the function $3\cY_s^{sym}$ via   $\cY_s$, we write $N_s$ as 
$
N_s= N^1_s + 2N_s^2, \; s\in\R^d, 
$
where
\begin{align*}
  N^1_s&=i\,\EE \Big(\bar a_s(\tauz+\tau) \int_0^\tau e^{-\ga_s(\tau-l)} 
  \cY_s(a(\tauz+l),a(\tauz+l),\Del a^{(1)}(l))\, dl\Big),
  \\
  N^2_s&=i\,\EE \Big(\bar a_s(\tauz+\tau) \int_0^\tau e^{-\ga_s(\tau-l)} 
  \cY_s(\Del a^{(1)}(l),a(\tauz+l),a(\tauz+l))\, dl
  \Big).
\end{align*}

{\it Term $N_s^1$.} Writing explicitly the function $\cY_s$ and then $\Del \bar a_3^{(1)}$ we get
\begin{align}\non
  N^1_s&=
  i\,L^{-d}\sum_{1,2}\dep\delta(\oms) \int_0^\tau dl\, e^{-\ga_s(\tau-l)} \times
  \\\non
 &\qquad\quad\qquad\qquad\qquad \EE \big(a_1(\tauz+l)a_2(\tauz+l) \Del\bar a_3^{(1)}(l)\bar a_s(\tauz+\tau)\big)
  \\\lbl{N_s-odin}
  &=L^{-2d}\sum_{1,2}\sum_{1',2'}\dep\de'^{1'2'}_{3'3}\delta(\oms)
  \delta(\om^{1'2'}_{3'3})
  \int_0^\tau dl\,\int_{0}^l dl'\, 
  e^{-\ga_s(\tau-l)} e^{-\ga_3(l-l')} 
  \\\non
  &{}\qu  \times \EE \big(a_1(\tauz+l)a_2(\tauz+l)
  \bar a_{1'}(\tauz+l')\bar a_{2'}(\tauz+l') a_{3'}(\tauz+l')
  \bar a_s(\tauz+\tau)\big).
\end{align}
By the Wick theorem, we need to take the summation only over
$s_{1'},s_{2'},s_{3'}$ satisfying 
$s_{1'}=s_1$, $s_{2'}=s_2$, $s_{3'}=s$ or $s_{1'}=s_2$, $s_{2'}=s_1$, $s_{3'}=s$.
Since in both cases we get 
$\de'^{1'2'}_{3'3}=\dep$
and 
$\om^{1'2'}_{3'3}=\oms$,
we find
\begin{align}\non
  N^1_s=
  2& \ssum_{1,2}\dep\delta(\oms)
  \int_0^\tau dl\,\int_{0}^l dl'\, 
  e^{-\ga_s(\tau-l)-\ga_3(l-l')}  
  \\\non
  &\times
  \EE a_1(\tauz+l)\bar a_{1}(\tauz+l')\,
  \EE a_2(\tauz+l)\bar a_{2}(\tauz+l')\,
  \EE a_s(\tauz+l')\bar a_{s}(\tauz+\tau).
\end{align}
Arguing as in Section \ref{ss_2}  we find
\be\lbl{N221}
| N_s^1| \le C^\#(s) L^{-2} \tau^2.
\ee

{\it Term $N_s^2$.} By  literally repeating the argument  we have applied to  $N_s^1$ we find that 
\begin{align}\non
	N_s^2&=
	i\,L^{-d}\sum_{1,2}\dep\delta(\oms) \int_0^\tau dl\, e^{-\ga_s(\tau-l)} \times
	\\\non
	&\qquad\quad\qquad\qquad\qquad
	\EE \big(\Del a_1^{(1)}(l)a_2(\tauz+l)\bar a_3(\tauz+l)\bar a_s(\tauz+\tau)\big)
	\\\non
	&=-L^{-2d}\sum_{1,2}\sum_{1',2'}\dep\de'^{1'2'}_{3'1}\delta(\oms)
        \delta(\om^{1'2'}_{3'1}) 
	\int_0^\tau dl\,\int_0^l dl'\, 
	e^{-\ga_s(\tau-l)} e^{-\ga_1(l-l')} 
	\\\non
	& \qu\times\EE \big(a_{1'}(\tauz+l')a_{2'}(\tauz+l')
	\bar a_{3'}(\tauz+l')a_{2}(\tauz+l) \bar a_{3}(\tauz+l)
	\bar a_s(\tauz+\tau)\big).
\end{align}
By the Wick theorem we should  take  summation either 
under the condition
$s_{1'}=s_3$, $s_{2'}=s$, $s_{3'}=s_2$ or $s_{1'}=s$, $s_{2'}=s_3$, $s_{3'}=s_2$.
Since in both cases 
$\de'^{1'2'}_{3'1}=\dep$
and 
$\om^{1'2'}_{3'1}=-\oms$,
then 
\begin{align}\lbl{same_expr}
	N_s^2=&
	-2\ssum_{1,2}\dep\delta(\oms)	\int_0^\tau dl\,\int_0^l dl'\, 
        e^{-\ga_s(\tau-l)} e^{-\ga_1(l-l')} \\
        \non
&	\EE a_2(\tauz+l)\bar a_{2}(\tauz+l')\,
	\EE a_3(\tauz+l')\bar a_{3}(\tauz+l)\,
	\EE a_s(\tauz+l')\bar a_{s}(\tauz+\tau)\,.
\end{align}
Again we get
\be\lbl{N222}
|  N_s^2| 
\leq  C^{\#}(s)L^{-2}\tau^2.
\ee

\subsection{The third term of $Q_{s,L}$  in \eqref{Q_s-def2}}\label{ss_1}
We have
$$
\EE \Del a^{(1)}_s \bar c^{(1)}_s(\tau)
=\EE \int_{\tauz}^{\tauz+\tau} e^{-\ga_s(\tauz+\tau-l)} \cY_s(a(l))\,dl 
\int_{0}^{\tauz} e^{-\ga_s(\tauz+\tau-l')} \ov{\cY_s(a(l'))}\,dl'.
$$
This expression coincides with \eqref{E Del a} in which the integral
$\int_{\tauz}^{\tauz+\tau} dl'$ is replaced by  $\int_0^{\tauz} dl'$.
Then,
\begin{align}\non
  \EE  \Del a^{(1)}_s(\tau)\bar c^{(1)}_s(\tau)=
  2\ssum_{1,2}&
  \dep\delta(\oms)\int_{\tauz}^{\tauz+\tau} dl \int_{0}^{\tauz} dl'\,
  e^{-\ga_s(2\tauz+2\tau-l-l')}
  \\\non
  &\EE a_1(l) \bar a_{1}(l') \,
  \EE a_2(l) \bar a_{2}(l') \,
  \EE \bar a_3(l) a_{3}(l'),
\end{align}
Expressing the correlations $\EE a_j(l)\bar a_j(l')$ through
\eqref{corr_a_in_time}, we get
\begin{align}\non
  \EE  \Del a^{(1)}_s(\tau)\bar c^{(1)}_s(\tau)=&
  2\ssum_{1,2}
  \dep\delta(\oms) B_{123}\int_0^\tau dl\,
  e^{-2\ga_s(\tau-l)-\ga_{123s}l}
  \\\non
  & \int_0^{\tauz} dl'\,e^{-\ga_{123s}\tauz} e^{\ga_s l'}\prod_{j=1,2,3}
  \left(e^{\ga_jl'} -e^{-\ga_jl'}\right)\,.
\end{align}
For the integral in the first line we have
\begin{equation}\lbl{N^2,1,tfactor}
T_{s} :=  \int_0^\tau dl\,
  e^{-2\ga_s(\tau-l)-\ga_{123s}l} =\left\{ \begin{array}{cc} \tau e^{-2\ga_s \tau} & \mbox{if }
  2\ga_s=\ga_{123s}\\
  \frac{e^{-2\ga_s\tau}-
    e^{-\ga_{123s}\tau}}{\ga_{123s} - 2 \ga_s}&
  \mbox{elsewhere}\end{array}\right. \,.
\end{equation}
 For the integral in the
second line, let us denote
\begin{equation}\lbl{N^2,1,Tfactor}
  \begin{split}
\cT^j:=& \int_0^{\tauz} dl\,e^{-\ga_{123s}\tauz} e^{\ga_j l}\prod_{k\neq j}
\left(e^{\ga_kl} -e^{-\ga_kl}\right)\,,
  \end{split}
\end{equation}
where $j,k \in \{1,2,3,s\}$.
Then, 
\be\lbl{Tfactor-estimate}
0\leq\cT^j\leq 1/\ga_{123s}.
\ee
Due to \eqref{N^2,1,tfactor} and \eqref{N^2,1,Tfactor} we get
\be\lbl{Ea^1c^1-fin}
\EE  \Del a^{(1)}_s(\tau)\bar c^{(1)}_s(\tau) =  2\ssum_{1,2}
  \dep\delta(\oms)  B_{123}T_{s} \cT^s.
\ee

\subsection{The fourth term of $Q_{s,L}$  in \eqref{Q_s-def2}}\label{ss_4}
To study the term  $2\Re \wt N_s$, as in Section~\ref{ss_3},  we write
$\wt N_s$ as 
$
\wt N_s= \wt N^1_s + 2\wt N_s^2, \; s\in\R^d, 
$
where
\begin{align*}
	\wt N^1_s&=i\,\EE \Big(\bar a_s(\tauz+\tau) \int_0^\tau e^{-\ga_s(\tau-l)} 
	\cY_s(a(\tauz+l),a(\tauz+l), c^{(1)}(l))\, dl\Big),
	\\
	\wt N^2_s&=i\,\EE \Big(\bar a_s(\tauz+\tau) \int_0^\tau e^{-\ga_s(\tau-l)} 
	\cY_s( c^{(1)}(l),a(\tauz+l),a(\tauz+l))\, dl
	\Big).
\end{align*}

{\it Term $\wt N_s^1$.} Writing explicitly the function $\cY_s$ and then $\bar c^{(1)}$ we get
\begin{align*}\non
  \wt N^1_s&=
  L^{-2d}\sum_{1,2}\sum_{1',2'}\dep\de'^{1'2'}_{3'3}\delta(\oms)
  \delta(\om^{1'2'}_{3'3})
  \int_0^\tau dl\,\int_{-\tauz}^{0} dl'\, 
  e^{-\ga_s(\tau-l)} e^{-\ga_3(l-l')} 
  \\\non
  &{}\qu  \times \EE \big(a_1(\tauz+l)a_2(\tauz+l)
  \bar a_{1'}(\tauz+l')\bar a_{2'}(\tauz+l') a_{3'}(\tauz+l')
  \bar a_s(\tauz+\tau)\big).
\end{align*}
Again, this is the same expression as \eqref{N_s-odin}, with the
integration over $dl'$ ranging from $-\tauz$ to 0 instead of from $0$ to
$l$. Thus, by the Wick theorem, we obtain
\begin{align}\non
  \wt N^1_s&=
	2 \ssum_{1,2}\dep\delta(\oms)
	\int_0^\tau dl\,\int_{-\tauz}^0 dl'\, 
	e^{-\ga_s(\tau-l)-\ga_3(l-l')}  
	\\\non
	&\times
	\EE a_1(\tauz+l)\bar a_{1}(\tauz+l')\,
	\EE a_2(\tauz+l)\bar a_{2}(\tauz+l')\,
	\EE a_s(\tauz+l')\bar a_{s}(\tauz+\tau).
\end{align}
Following the line of Section \ref{ss_1}  we express the correlations 
through \eqref{corr_a_in_time} and get 
\begin{equation}
  \begin{split} \lbl{tN^2,1_wick,2}
  \wt N^1_s=&
  2\ssum_{1,2}
  \dep\delta(\oms) B_{12s}\int_0^\tau dl\,
  e^{-2\ga_s(\tau-l)-\ga_{123s}l}
  \\
  & \int_0^{\tauz} dl'\,e^{-\ga_{123s}\tauz} e^{\ga_3 l'} \left(e^{\ga_sl'}
  -e^{-\ga_sl'}\right)\prod_{j=1,2} 
  \left(e^{\ga_jl'} -e^{-\ga_jl'}\right)
  \\
   = & 2\ssum_{1,2} \dep\delta(\oms) B_{12s} T_{s}\cT^3\,.
  \end{split}
\end{equation}

{\it Term $\wt N_s^2$.} Literally repeating the argument which we have
applied to  $\wt N_s^1$,  we find that the term $\wt N_s^2$ is given by the same expression as \eqref{same_expr}
with the integral $\int_0^l$ replaced by $\int_{-\tau_0}^0$:
\begin{align}\non
\wt	N_s^2=&
	-2\ssum_{1,2}\dep\delta(\oms)	\int_0^\tau dl\,\int_{-\tauz}^0 dl'\, 
        e^{-\ga_s(\tau-l)} e^{-\ga_1(l-l')} \\
        & \non
	\EE a_2(\tauz+l)\bar a_{2}(\tauz+l')\,
	\EE a_3(\tauz+l')\bar a_{3}(\tauz+l)\,
	\EE a_s(\tauz+l')\bar a_{s}(\tauz+\tau).
\end{align}
Again we get
\be\lbl{wt N_s^2}
 \wt N_s^2 =    - 2\ssum_{1,2} \dep\delta(\oms) B_{23s} T_{s}\cT^1\,.
\ee

\subsection{End of the proof}
\lbl{s:end_pr_lN2}
Inserting formulas \eqref{Ea^1c^1-fin}, \eqref{tN^2,1_wick,2}
and \eqref{wt N_s^2}, as well as \eqref{f_1},
\eqref{N221}, \eqref{N222} in \eqref{Q_s-def2}, we get
\begin{equation}\lbl{Q_s-appr}
  \begin{split}
\left|Q_{s,L}-4L^2T_{s}\ssum_{1,2} \dep\delta(\oms) \left(B_{123}\cT^s +
B_{12s}\cT^3 - 2B_{23s}\cT^1\right)\right| \le C^\#(s)\tau^2\,.
  \end{split}
\end{equation}
 Note that the terms $\cZ^j$ defined in \eqref{Z-edi} can be written as
\be\lbl{fromTtoZ}
\cZ^j=\frac{\cT^j}{\prod_{k\neq j}(1-e^{-2\ga_k\tauz})}\,.
\ee
The relations
\eqref{corr_a_in_time}-\eqref{corr()} imply that   for any permutation $(k_1,k_2,k_3,k_4)$ of $(1,2,3,s)$ we have  $B_{k_1k_2k_3}= 
n^{(0)}_{k_1}n^{(0)}_{k_2}n^{(0)}_{k_3}/\prod_{m=k_1,k_2,k_3}(1-e^{-2\ga_m\tauz})$, 
where $n^{(0)}_{k_i} = n_{k_i,L}^{(0)}(\tauz)$. Together with \eqref{fromTtoZ} this implies
\be\lbl{B-T-Z}
B_{k_1k_2k_3}\cT^{k_4}=  \cZ^{k_4}n^{(0)}_{k_1}n^{(0)}_{k_2}n^{(0)}_{k_3} \,.
\ee
 By symmetry, the term $2B_{23s}\cT^1$ in \eqref{Q_s-appr} can be replaced by $B_{23s}\cT^1 + B_{13s}\cT^2$. Then,
  inserting \eqref{B-T-Z} in \eqref{Q_s-appr} we get
\begin{equation*}
  \begin{split}
\left|Q_{s,L}-\cX_s\right|\le &\left|4L^2(T_s-\tau)\ssum_{1,2}\dep\delta(\oms)\left( \cT^sB_{123}\
+ \cT^3B_{12s}- 2\cT^sB_{23s}\right)\right|\\
&+ C^\#(s)\tau^2\,,
  \end{split}
\end{equation*}
with $\cX_s$ defined in \eqref{ZZ_s}.
Finally, we point out that $|(T_s-\tau)\cT^j|\leq 3\tau^2$, due to
\eqref{Tfactor-estimate} and since  $| T_{s} -\tau|\le
3\tau^2\ga_{123s}$. So, the bound
$$
    \left|L^2(T_s-\tau)\ssum_{1,2}\dep\delta(\oms)\left( \cT^sB_{123}\
    + \cT^3B_{12s}- 2\cT^sB_{23s}\right)\right|\le C^\#(s)\tau^2\,,
$$
is a consequence of Theorem~\ref{t:countingterms}.
This concludes the proof of Proposition~\ref{l:N2}.

\subsection{ Proof of Lemma \ref{l:Z^j-growth}}
\lbl{s:Z^j-growth}
Note that
$
\p_{s_j}f(\ga_j) = f'(\ga_j) \p_{s_j}\ga_j,
$
where $\p_{s_j}\ga_j$   (as well as higher order derivatives of $\ga_j$) 
have at most polynomial growth at infinity. Then, using
the definition \eqref{Z-edi} of $\cZ^j$ we find
\begin{equation}\lbl{Z-der}
	\begin{split}
		\left| \p^{\mu}_{\vs} \cZ^j(\tauz,\vs)\right| = \sum_{n_1+n_2+n_3+n_s=  1}^{|\mu|}&P(\vs; n_1,\ldots,n_s) \int_0^{\tauz}dl\,(\tauz-l)^{n_j}\e^{-\ga_j(\tauz-l)}\\
		&\prod_{m\neq j} \frac{d^{n_m}}{d\ga_m^{n_m}}\left(\frac{\sinh(\ga_m
			l)}{\sinh(\ga_m \tauz)}\right)\,,
	\end{split}
\end{equation}
where $P(\vs;a_1,\ldots,a_n)$ denotes a function of $\vs$, dependent on parameters $(a_1,\ldots,a_n)$, having at most a 
polynomial growth at infinity.
Using the relation
$$
\frac{\sinh(\ga
	l)}{\sinh(\ga \tauz)} = \frac{e^{-\ga(\tauz-l)} -e^{-\ga(l+\tauz)}}{1-e^{-2\ga \tauz}} 
$$
we find by induction 
\begin{equation*} 
	\frac{d^{n}}{d\ga^{n}}\left(\frac{\sinh(\ga
		l)}{\sinh(\ga \tauz)}\right) = \sum_{k+m+p=n} c_{k,m,p} I_{k,m,p}(l,\tau_0,\ga),
\end{equation*}
where $c_{k,m,p}$ are constants,
\begin{equation}\lbl{I_kmn}
	I_{k,m,p} = \left( (\tauz-l)^ke^{-\ga(\tauz-l)} - (l+\tauz)^ke^{-\ga(l+\tauz)} \right)
	\frac{\tauz^{m+p}e^{-2\ga m \tauz}}{(1-e^{-2\ga \tauz})^{m+1}}\ ,
\end{equation}
and $p\neq 0$ only if $m\neq 0$.
For  $\tau_0\geq \ga^{-1}$ the terms $I_{k,m,p} $ are bounded in absolute values by absolute constants $C_{k,m,p}$, where we recall that $0\leq l \leq \tau_0$ and $\ga \geq 1$. 
Let now $\tau_0\leq \ga^{-1}$. 
In this case, since $k+m+p=n$,
$$
|I_{k,m,p}| \leq 2\, \frac{(l+\tau_0)^k\tau_0^{m+p}}{(1-e^{-2\ga \tauz})^{m+1}} \leq 2^{k+1} \,\frac{\tau_0^{n}}{(1-e^{-2\ga \tauz})^{m+1}} .
$$
So, in the case $m\leq n -1 $ we have $|I_{k,m,p}|\leq C_{k,m,p}$ uniformly in  $\tau_0\leq \ga^{-1}$.
If $m=n$ (so $k=p=0$) we use another estimate, following from \eqref{I_kmn}:
$$
|I_{k,m,p}| \leq C\, \frac{e^{-\ga(\tauz-l)} - e^{-\ga(l+\tauz)}}{(1-e^{-2\ga \tauz})^{m+1}}\tau_0^m \! =\! C\tau_0^m\,e^{-\ga(\tauz-l)} \frac{1 - e^{-2\ga l}}{(1-e^{-2\ga \tauz})^{m+1}}\, \leq C_{k,m,p},
$$
uniformly in $\tau_0\leq \ga^{-1}$.

We have seen that the product in
 \eqref{Z-der} is bounded uniformly in $\vs,\, l$ and $\tau_0$, so the integral over $l$ is also bounded uniformly in $\vs$ and $\tau_0$.

\section{Proof of Proposition~\ref{p:approx}}
\lbl{app:approx}

The proof uses the theory of Feynman diagrams, presented  in Section~\ref{sec:diagram}. For $N=0$ the assertion is trivial. For $N\geq 1$ in Proposition 8.7 of \cite{DK2} it is proven that  $\EE\big(\mathfrak{a}_s^{(m)}(\tau_1)\bar {\mathfrak{a}}_s^{(n)}(\tau_2)- a_s^{(m)}(\tau_1)\bar a_s^{(n)}(\tau_2)\big)$ equals to
\be\lbl{a---A_bound}
\sum_{\gF\in\gF^+_{m,n}\setminus \gF_{m,n}}c_\gF\,\mathcal{J}_s(\gF)
+\sum_{\gF\in\gF_{m,n}}c_\gF\,\mathcal{J}^2_s(\gF),
\ee 
where $\gF^+_{m,n}$ is a certain (finite) set of {\it extended} Feynman diagrams, 
\footnote{These diagrams are defined similarly to the Feynman diagrams from Section~\ref{s:FD_def}, but now we allow to couple leaves not only from different blocks but also from the same block.} 
$c_\gF$ is a complex number of unit norm
and $\mathcal{J}_s(\gF)$, $\mathcal{J}_s^2(\gF)$ are sums, 
 similar to \eqref{J(F)-z}. In Section 8.6.3 of \cite{DK2} are established the following bounds for these sums: 
 \be\lbl{app:J_2-est}
|\mathcal{J}^2_s(\gF)|\leq C^\#(s)L^{-Nd}
\sum_{\substack{z\in\cZ^+(\gF):\\  z_j=0 \mbox{ {\footnotesize for some} }j, \\
		\om^\gF_k(z)=0 \,\forall 1\leq k\leq N}}C^\#(z),
\ee
where
$$
\cZ^+(\gF)=\big\{z\in(\Z_L^{d})^{N}:\ z_{k}\ne 0 \Leftrightarrow\sum_{i=1}^N\al_{ki}^{\gF} z_i\ne 0 \quad \forall 1\leq k\leq N\big\}
$$
 while the quadratic forms $\om_k^\gF$ and the skew-symmertic matrix $\al^\gF$ are defined in Section~\ref{s:transform}.
 Note that possibly the diagram $\gF$ does not belong to the set $\gF^{\,true}_{m,n}$, so that the 
 matrix $\al^\gF$ may have zero columns and lines.
On the other hand,
\be\lbl{app:J-est}
|\mathcal{J}_s(\gF)|\leq C^\#(s)L^{-Nd}
\sum_{\substack{z\in\wt\cZ^+(\gF):\\ \wt\om^\gF_k(z)=0 \,\forall 1\leq k\leq \wt N}} C^\#(z).
\ee
Here $\wt N=\wt N(\gF)<N$,  quadratic forms $\wt \om_k^\gF(z)$ are defined  by relations \eqref{om_j^F(z)}, 
where $N$ is replaced by $\wt N$ and the matrix 
$(\al_{ij}^\gF)$~-- by a certain $\wt N\times \wt N$-matrix $(\wt\al_{ji}^\gF)$, also satisfying $\wt\al_{ji}^\gF=-\wt\al_{ij}^\gF\in\{0,\pm 1\}$ for all $i,j$. 
Accordingly the 
 set $\wt\cZ^+(\gF)\subset(\Z^d_L)^{\wt N}$ is defined as  $\cZ^+(\gF)$ above, 
 but with  $N$ and $\al_{ij}^\gF$  replaced by $\wt N$ and $\wt\al_{ij}^\gF$.

We first show that 
 the term $\mathcal{J}^2_s(\gF)$ is bounded by the r.h.s. of \eqref{eq:comp_moment}. To this end 
we write $\cZ^+(\gF)=\cup_{\cK}\cZ_\cK$, where the union is taken over all subsets $\cK\subset\{1,\ldots,N\}$ and
$$
\cZ_\cK(\gF)=\{z:\, z_k=\sum_{i=1}^N\al_{ki}^\gF z_i=0\; \forall k\in\cK \;\;\text{and}\;\;
z_k\ne 0,\; \sum_{i=1}^N\al_{ki}^\gF z_i\ne 0 \;\forall k\notin\cK\}.
$$
Then the r.h.s. of \eqref{app:J_2-est} takes the form
\be\lbl{app:J_2-est1}
C^\#(s)L^{-Nd} \sum_{\cK\ne\emptyset}\sum_{\substack{z\in\cZ_\cK(\gF):\\
		\om^\gF_k(z)=0 \,\forall 1\leq k\leq N}}C^\#(z).
\ee
Note that on the set $\cZ_\cK(\gF)$ we have $\om_k(z)=0$ for all $k\in\cK$ and $\om_k(z)=2z_k\cdot\sum_{i\notin\cK}\al_{ki}^\gF z_i$ for $k\notin\cK$.
Thus, the sum over $z$ in \eqref{app:J_2-est1} takes the form of the sum in \eqref{S_L}, where $z=(z_j)_{j\notin\cK}$ and 
 $N$ is replaced by
$$
N- \varkappa,  \quad \varkappa = \#\cK.
$$ 
 We recall that $N\le4$ and $\cK\not\equal\emptyset$, so that $N-\varkappa$ takes values $0,\,1,\, 2$ or $3$. 
	For the sets $\cK$ satisfying $N-\varkappa = 0$ we have $\cZ_\cK(\gF) = \{0\}$, so the sum \eqref{app:J-est} is bounded by $C^\#(s)L^{-Nd}$.
	Since the matrix $\al^\gF$ is skew-symmetric, then 
	 in the case $N-{\varkappa}=1$ we have $\cZ_\cK(\gF)=\emptyset$, so the sum   \eqref{app:J-est} vanishes.
   When $N-{\varkappa} = 2$ or $3$ we apply
 Theorem~\ref{t:countingterms} and see that the sum over $z$ in
 \eqref{app:J_2-est1}  is bounded by $C L^{-(N-{\varkappa})(1-d)}$. So 
\be\non
\begin{split}
|\mathcal{J}^2_s(\gF)|\leq C^\#(s) L^{-Nd}\sum_{\cK\ne\emptyset}L^{-(N-{\varkappa})(1-d)}
&=C^\#(s)\sum_{\cK\ne\emptyset} L^{-N+{\varkappa}(1-d)}\\
&\leq C_1^\#(s)L^{-N+1-d}.
\end{split}
\ee

Same argument implies that the r.h.s. of \eqref{app:J-est}  also is bounded by the quantity 
$C^\#(s)L^{-N+1-d}$ 
 (note that  decomposing the r.h.s. of \eqref{app:J-est} as in  
   \eqref{app:J_2-est1} we get a new term with 
 $\cK \equal\emptyset$, but for it $N-\varkappa=N\le4$ and  Theorem~\ref{t:countingterms} still applies).

\section{Case $d=2$}
\lbl{app:d=2}

 A difference between the cases $d\geq 3$ and $d=2$ comes from Theorem~\ref{t:numbertheory} since in the asymptotic,
  given by the latter,  an additional log-factor appears when $d=2$. To handle it we 
  redefine the sum in \eqref{S_L}, defining $S_{L,N}(\Phi)$, 
  by multiplying it by $(\ln L)^{-N/2}$. So when  $d=2$ \,  $S_{L,N}$ takes the form
\be\lbl{S_L'}
S_{L,N}(\Phi) := \frac{L^{N(1-d)}}{(\ln L)^{N/2}} \sum_{z\in\cZ: \  \omega_j (z) =0 \, \forall j} \Phi(z). 
\ee 
Accordingly the $(d=2)$-analogy of  \eqref{S_L,2 approx} reads 
\be\lbl{z12=0}
\Big| S_{L,2}(\Phi) - \frac{L^{2(1-d)}}{\ln L}\!\! \sum_{z\in\Z^{2d}_L: \  z_1\cdot z_2 =0} \Phi(z) \Big| \leq \frac{CL^{2-d}}{\ln L}\|\Phi\|_{0,d+1}=\frac{C}{\ln L}\|\Phi\|_{0,3}.
\ee 
This approximation, jointly with  a modification of the Heath-Brown result from \cite{HB}, given in 
 Theorem 1.4 of \cite{number_theory}, implies the following version of Theorem~\ref{t:numbertheory} for 
$d=2$: 
\begin{theorem}\lbl{t:numbertheory'}
	Let $d=2$. Then there exist constants $N_1,N_2>4$ such that if 
	$\|\Phi\|_{N_1,N_2}< \infty$,
	\be\label{H-BB}
	\left|S_{L,2}(\Phi) -  C_2 \int_{\Sigma_0}
	{\Phi(z)} \, 
	\mu^{\Sigma_0} (
	dz_1dz_2) 
	\right|\le K_2\frac{\|\Phi\|_{N_1,N_2}}{\ln L} \,,
	\ee
	where $C_2>0$ is a number theoretical constant and $K_2>0$. 
\end{theorem}
Note that  estimate \eqref{int_est} stays true when $d=2$.
\begin{theorem}\lbl{t:countingterms'}
	In the case $d=2$ assertion of Theorem~\ref{t:countingterms} remains true, if the sum $S_{L,N}$ is defined as in \eqref{S_L'} and $N_2$ is the constant from Theorem~\ref{t:numbertheory'}.
\end{theorem}
{\it Proof}. The only difference with the proof of Theorem~\ref{t:countingterms} comes from estimate \eqref{est box} since the latter is obtained by applying Theorem~\ref{t:numbertheory}, and in the case $d=2$ we should apply Theorem~\ref{t:numbertheory'} instead. Namely, now the r.h.s. of \eqref{est box} takes the form $CL^{2(d-1)} \ln L\big[ R^{2d} + R^{N_2}(\ln L)^{-1}\big]\leq C'R^{N_2}L^{2(d-1)}\ln L $. 
Since   Lemma~\ref{l:intersection} remains unchanged, then for $d=2$  the r.h.s. of estimate \eqref{sum_w_est}, 
which holds for irreducible matrices $\al$, should be multiplied by $\ln L$.  In the case of reducible matrix $\al$ we apply the latter estimate to each irreducible block, which gives the factor $(\ln L)^{\lfloor N/2\rfloor}$ in the r.h.s. of \eqref{est_prop3.3}, since the number of blocks does not exceed $\lfloor N/2\rfloor$.
However, the final estimate of Theorem~\ref{t:countingterms} remains unchanged because of the factor $(\ln L)^{-N/2}$ in the definition \eqref{S_L'} of the sum $S_{L,N}$. 
\qed

Since in the case $d=2$ we choose $\rho=\eps L/\sqrt{\ln L}$, then the terms $n_{s,L}^{(k)}$ are given by formula \eqref{n_s^(k)}, 
multiplied by $(\ln L)^{-k/2}$.  The proof of
  Proposition~\ref{p:approx} is analogous to that presented in Appendix~\ref{app:approx}  for $d=2$. The 
    only difference being
  the use of Theorem~\ref{t:countingterms'}  in place of
  Theorem~\ref{t:countingterms}.  Lemma~\ref{l:E-fin} remains
unchanged, so  the correlations
 $L^k (\ln L)^{-k/2} \EE a^{(m)}_s \bar a^{(n)}_s(\tau_2)$, $m+n=k$, 
  are given by formula \eqref{final_Ea^ma^n}, multiplied by $(\ln L)^{-k/2}$ (recall that the sum of these correlations makes $n_{s,L}^{(k)}$). 
We see that the correlations take the form \eqref{S_L'}, so Theorems~\ref{t:numbertheory'} and \ref{t:countingterms'} apply to study them. 

The rest of the proof of Theorem~\ref{t_k1} literally repeats that for 
 the case $d\geq 3$, except the appearance of the $(\ln L)^{-k/2}$
 factors, coming from the new definition of $\rho$. Now the estimates,
 using Theorem~\ref{t:numbertheory}, should be relaxed since the
 estimate provided by Theorem~\ref{t:numbertheory'} is slightly weaker
 than that of Theorem~\ref{t:numbertheory}. In particular, in the
 r.h.s. of \eqref{decompp}, \eqref{diff-limit<2} and \eqref{diff-limit} the factor 
  $L^{-1/2}$ should be replaced by  $(\ln L)^{-1}$. This results in the stronger
    lower bound for $L$ in Theorem~\ref{t_k1}: now it is $L\geq  e^{\eps^{-1}}$ instead of $L\geq \eps^{-2}$ (see Theorem A).

Theorem~\ref{t_k2}, as Theorem~\ref{t_k1}, remains unchanged, except the lower bound for $L$ which is modified as above. Indeed, the theorem follows from Theorem~\ref{t_k1} and Proposition~\ref{p:approx}, 
 and the term $\chi_2(L)^{-N+1}=(\ln L)^{(N-1)/2}$,
  appearing  in estimate~\eqref{eq:comp_moment} for $d=2$ does not change the assertion of Theorem~\ref{t_k1}.

\subsection{Discussion of Remark~\ref{rem:d=2}}
In fact, Theorem 1.4 from \cite{number_theory} provides  more delicate information about $S_{L,2}$ than what is 
 stated in Theorem~\ref{t:numbertheory'}. Namely, if $d=2$ then due to \cite{number_theory}, 
\be\non	\left| \frac{L^{2(1-d)}}{\ln L} \sum_{z: \  z_1\cdot z_2 =0} \Phi(z) -  C_2 \int_{\Sigma_0}
{\Phi(z)} \, 
\mu^{\Sigma_0} (
dz_1dz_2) 
-\frac{\sigma_1^\Phi(L)}{\ln L} \right|
\leq C \frac{\|\Phi\|_{N_1,N_2}}{L^{1/6}},
\ee
 where $\sigma_1^\Phi$ is a certain function satisfying $|\sigma_1^\Phi(L)|\leq C_1\|\Phi\|_{N_1,N_2}$, uniformly in $L$. See \cite{number_theory} for an explicit (but complicated) formula for $\sigma_1^\Phi$.
Consequently, 
\be\non	\left| S_{L,2}(\Phi) -  C_2 \int_{\Sigma_0}
{\Phi(z)} \, 
\mu^{\Sigma_0} (
dz_1dz_2) 
-\frac{\wt\sigma_1^\Phi(L)}{\ln L} \right|
\leq C \frac{\|\Phi\|_{N_1,N_2}}{L^{1/6}},
\ee
where 
 \be\non
\wt\sigma_1^\Phi(L) := \sigma_1^\Phi(L) - L^{2(1-d)}\sum_{z:\ z_1=0 \mbox{ \footnotesize or }z_2=0}\Phi(z)
 \ee
 still satisfies  $|\wt\sigma_1^\Phi(L)|\leq C\|\Phi\|_{N_1,N_2}$ in view of \eqref{z12=0}.
 Then estimate \eqref{diff-limit} refines as
\be\lbl{2.12new}
\Big|n^{\le 2}_s - \nL{s}-\frac{f(\tau, L)}{\ln L}\Big| \le C^\#(s)\eps^2(L^{-  1/6} + \eps)\,,
\ee
where $f(\tau,L):=\wt\sigma_1^{\Phi(\tau)}(L)$ and $\Phi(\tau)$ is the function satisfying $n_{s,L}^{\le 2}(\tau)=S_{L,2}(\Phi(\tau))$ that comes from Corollary~\ref{cor:Diag}.
By \eqref{cor_diag_est} and the estimate for $\wt\sigma_1^\Phi$ above, the function $f(\tau,L)$ is bounded  uniformly in $\tau$.
The rest of the proofs of Theorems~\ref{t_k1} and \ref{t_k2} remain unchanged while the estimate \eqref{2.12new} leads to the assertion of the remark.

\medskip 

{\bf Acknowledgements.} The work of AD was performed at the Steklov International Mathematical Center and supported by the Ministry of Science and Higher Education of the Russian Federation (agreement no. 075-15-2019-1614). SK was supported  by  Agence Nationale de la Recherche through   the grant    17-CE40-0006.


\begin{thebibliography}{99}
    

  
\bibitem{BGHS18} T. Buckmaster,  P. Germain,  Z. Hani,  J. Shatah, {\it Effective dynamics of the nonlinear Schrödinger equation on large domains}, 
Comm. Pure Appl. Math. {\bf 71} 1407--1460, (2018).

\bibitem{BGHS} T. Buckmaster, P. Germain, Z. Hani, J. Shatah, {\it Onset of the wave turbulence description of the longtime behaviour of the nonlinear Schr\"odinger equation}, Invent. Math. (2021) https://doi.org/10.1007/s00222-021-01039-z.

\bibitem{CG1} C. Collot, P. Germain, {\it On the derivation of the homogeneous kinetic wave equation}, (2019),  arXiv:1912.10368.

\bibitem{CG2} C. Collot, P. Germain, {\it Derivation of the homogeneous kinetic wave equation: longer time scales}, (2020), arXiv:2007.03508.   


\bibitem{DH} Y. Deng, Z. Hani, {\it  On the derivation of the wave kinetic equation for NLS}, (2019), arXiv:1912.09518.

\bibitem{DH1} Y. Deng, Z. Hani, {\it Full derivation of the wave kinetic equation}, arXiv:2104.11204, \ 2021.


\bibitem{DK1} A. Dymov, S. Kuksin, {\it Formal expansions in
	stochastic model for wave turbulence 1: kinetic limit}, Comm. Math. Phys.	{\bf 382}, 951--1014 (2021).
\bibitem{DK2} A. Dymov, S. Kuksin, {\it Formal expansions in stochastic model for wave turbulence 2: 
	method of diagram decomposition}, 	arXiv:1907.02279,\  2019.
\bibitem{DKsmall} A. Dymov, S. Kuksin, {\it On the Zakharov-L'vov stochastic model for wave turbulence}, Dokl. Math. {\bf 101}, 102-109 (2020).


\bibitem{number_theory}  A. Dymov, S. Kuksin, A. Maiocchi, S. Vl\u adu\c t,
{\it A refinement of Heath-Brown's theorem on quadratic forms}, arXiv:2110.13873, 2021.

\bibitem{EY00} L. Erd\"os, H.T. Yau,  {\it Linear Boltzmann equation as the weak coupling limit of a random Schr\"odinger equation}, Commun. Pure Appl. Math. {\bf 53}, 667-735 (2000).

\bibitem{ESY07} L. Erd\"os, B. Schlein, H.-T. Yau, {\it Derivation of the cubic non-linear Schrödinger equation from quantum dynamics of many-body systems}, Invent. Math. {\bf 167}, 515-614 (2007).

\bibitem{Fu} W. Fulton, {\it Intersection Theory,} Springer, Berlin, 1998.

\bibitem{HB} D.R. Heath--Brown, {\it A new form of the circle method, and its application
	to quadratic forms}, J. Reine Angew. Math. 481 (1996), 149-206.

\bibitem{Faou}
E. Faou,
\textit{Linearized wave turbulence convergence results for three-wave systems},  Comm. Math. Physics {\bf 378},  807-849 (2020)

\bibitem{FGH}
E. Faou, P. Germain, Z. Hani,
\textit{The weakly nonlinear large-box limit of the 2D cubic nonlinear Schr\"odinger equation},  J. Amer. Math. Soc.,  {\bf 29},   915--982  (2016)



\bibitem{HKM}
G. Huang, S. Kuksin, A. Maiocchi,
\textit{Time-averaging for weakly nonlinear {CGL} equations with
	arbitrary potentials}, Fields Inst. Commun.  \textbf{75},
323--349 (2015).


\bibitem{Jan}
S.  Janson,
\textit{Gaussian Hilbert Spaces}, Cambridge University Press 1997.




\bibitem{Kuk}  S. Kuksin,	{\it Ergodicity, mixing and KAM}, S\'eminaire Laurent Schwartz EDP et applications, 2018-2019, Exp. No. 8, 9 (2019); 
 doi: 10.5802/slsedp.128.

\bibitem{KM15}  S. Kuksin, A. Maiocchi,
\textit{Derivation of a wave kinetic equation from the  resonant-averaged stochastic {NLS} equation}, Physica~D  \textbf{309},
65-70 (2015).

\bibitem{KM16}
S. Kuksin and A. Maiocchi,
\textit{Resonant averaging for small solutions of stochastic NLS equations}, Proc. Royal Soc. Edinburgh \textbf{147A}, 1--38  (2017).

\bibitem{LR}  G. Lachaud, R. Rolland, {\it On the number of points of algebraic sets over finite fields}, J. Pure  Appl. Algebra, 219(2015), 5117--5136.

\bibitem{LS}
J. Lukkarinen, H. Spohn,
\textit{Weakly nonlinear Schr\"odinger equation with random initial data},  Invent. Math.  {\bf 183}, 79--188  (2015).
\bibitem{La} S. Lang,  {\it Algebra},  Springer, New York,  2002.

\bibitem{Naz11}
S. Nazarenko, \textit{Wave {T}urbulence}, Springer 2011.



\bibitem{NR} A. C. Newell,  B. Rumpf, {\it Wave Turbulence},  Annu. Rev. Fluid Mech. {\bf 43}, 59-78 (2011).


\bibitem{Sh} I. Shafarevitch, {\it Basic algebraic geometry}, Springer, Berlin, 1994.



\bibitem{ZL75}
V.  Zakharov, V.  L'vov, \textit{Statistical description of nonlinear wave fields},
Radiophys. Quan. Electronics \textbf{18}, 1084-1097 (1975).


\bibitem{ZLF92}
V. Zakharov, V. L'vov,  G. Falkovich, \textit{Kolmogorov Spectra of Turbulence}, Springer 1992.	

\end{thebibliography}
\end{document}